\documentclass[reqno, 10pt]{amsart}
\input epsf
\usepackage[arrow, matrix, all, cmtip]{xy}
\usepackage{graphicx}
\usepackage{amsmath}
\usepackage{hyperref}
\usepackage{comment}
\usepackage{color}
\usepackage{amsmath, latexsym, amssymb}
\usepackage{epstopdf}
\input xypic

\numberwithin{equation}{section}
\theoremstyle{plain}
\newtheorem{lemma}{Lemma}[section]
\newtheorem{proposition}[lemma]{Proposition}
\newtheorem{theorem}[lemma]{Theorem}
\newtheorem{corollary}[lemma]{Corollary}

\theoremstyle{definition}
\newtheorem{definition}[lemma]{Definition}
\newtheorem{remark}[lemma]{Remark}
\newtheorem{example}[lemma]{Example}

\newcommand{\R}{{\mathbb R}}

\newcommand{\F}{{\mathbb F}}

\newcommand{\Q}{{\mathbb Q}}

\newcommand{\D}{{\mathbb D}}
\newcommand{\G}{{\mathbb G}}
\newcommand{\W}{{\mathbb W}}
\newcommand{\bbS}{{\mathbb S}}

\newcommand{\Cc}{{\mathcal C}} 

\newcommand{\Ff}{{\mathcal F}}

\newcommand{\Kk}{{\mathcal K}}

\newcommand{\Ll}{{\mathcal L}} 

\newcommand{\om}{{\omega}}
\newcommand{\eps}{{\varepsilon}}

\newcommand{\im}{{\rm Im\,}}

\newcommand{\la}{\langle}
\newcommand{\ra}{\rangle}

\newcommand{\INTO}{\hookrightarrow}

\newcommand{\Der}{\mathcal{D}}
\newcommand{\der}{D}
\newcommand{\id}{\operatorname{id}}
\newcommand{\Mod}{\, \textnormal{mod}\, }

\makeatletter
\def\@tocline#1#2#3#4#5#6#7{\relax
  \ifnum #1>\c@tocdepth 
  \else
    \par \addpenalty\@secpenalty\addvspace{#2}%
    \begingroup \hyphenpenalty\@M
    \@ifempty{#4}{%
      \@tempdima\csname r@tocindent\number#1\endcsname\relax
    }{%
      \@tempdima#4\relax
    }%
    \parindent\z@ \leftskip#3\relax \advance\leftskip\@tempdima\relax
    \rightskip\@pnumwidth plus4em \parfillskip-\@pnumwidth
    #5\leavevmode\hskip-\@tempdima
      \ifcase #1
       \or\or \hskip 1.5 em \or \hskip 2em \else \hskip 3em \fi%
      #6\nobreak\relax
    \hfill\hbox to\@pnumwidth{\@tocpagenum{#7}}\par
    \nobreak
    \endgroup
  \fi}
\makeatother

\title[Deformations of Coisotropic Submanifolds]
{Deformations of Coisotropic Submanifolds
\\ in Jacobi Manifolds}

\author{H\^ong V\^an L\^e}
\address{Institute of Mathematics of ASCR, Zitna 25, 11567 Praha 1, Czech Republic.}
\email{hvle@math.cas.cz}

\author{Yong-Geun Oh}
\address{Center for Geometry and Physics, Institute for Basic Sciences (IBS), 77 Cheongam-ro, Nam-gu,
Pohang, Korea \&
Department of Mathematics, POSTECH, Pohang, Korea}
\email{yongoh1@postech.ac.kr}

\author{Alfonso G.~Tortorella}
\address{Dipartimento di Matematica e Informatica ``U.~Dini'', Universit\`a degli Studi di Firenze, Viale Morgagni 67/a
50134 Firenze, Italy.}
\email{alfonso.tortorella@math.unifi.it}

\author{Luca Vitagliano}
\address{DipMat, Universit\`a degli Studi di Salerno \& Istituto Nazionale di Fisica Nucleare, GC Salerno, via Giovanni Paolo II n${}^\circ$ 123, 84084 Fisciano (SA) Italy.}
\email{lvitagliano@unisa.it}

\thanks{The first named author is partially supported by RVO: 67985840,
the second named author is supported by the IBS project \#IBS-R003-D1}

\begin{document}

\begin{abstract}
	In this paper, we attach an $L_\infty$-algebra to any coisotropic submanifold in a Jacobi manifold.
	Our construction generalizes and unifies analogous constructions by Oh-Park (symplectic case), Cattaneo-Felder (Poisson case), L\^e-Oh (locally conformal symplectic case).
	As a new special case, we attach an $L_\infty$-algebra to any coisotropic submanifold in a contact manifold.
	The $L_\infty$-algebra of a coisotropic submanifold $S$ governs the (formal) deformation problem of $S$.
\end{abstract}



\maketitle

\tableofcontents

\section{Introduction}\label{sec:intro}

Jacobi structures were independently introduced by Lichnerowicz \cite{Lich1978} and
Kirillov \cite{Kiri1976}, and they are a combined generalization of symplectic or Poisson
structures and contact structures. Note that Kirillov \emph{local Lie algebras with one dimensional fiber} \cite{Kiri1976} are slightly more general than Lichnerowicz \emph{Jacobi manifolds}. \color{black} In this note we will adopt the following definition, which is equivalent to Kirillov's one: a Jacobi manifold is a manifold $M$ equipped with a Jacobi structure, i.e.~a pair $(L, \{-,-\})$ consisting of a line bundle $L \to M$ and a Lie bracket $\{-,-\}$ on sections of $L$ which is a first order differential operator in each entry (see Definition \ref{definition:abstractj}). Jacobi manifolds \emph{\`a la} Lichnerowicz correspond to the case when $L = M \times \mathbb R$ is the trivial line bundle, and are, somehow, more popular. So we reserve the terminology \emph{standard Jacobi manifolds} for them. \color{black} While general Jacobi manifolds encompass non-coorientable contact manifolds, standard Jacobi manifolds do not. 

Coisotropic submanifolds in (standard) Jacobi manifolds have been first studied by Ib\'a\~nez-de Le\'on-Marrero-Mart\'in de Diego \cite{ILMM1997}. They showed that these submanifolds play a similar role as coisotropic submanifolds in Poisson manifolds. For instance, the graph of a conformal Jacobi morphism $f: M_1 \to M_2$ between Jacobi manifolds is a coisotropic submanifold in $M_1 \times M_2 \times \R$ equipped with an appropriate Jacobi structure. 
Other important examples of coisotropic submanifolds in a Jacobi manifold $M$ are leaves of the characteristic distribution, and zero level sets of equivariant momentum maps.
Since the property of being coisotropic does not change in the same conformal class of a standard Jacobi manifold (see Remark \ref{rem:conf} and Lemma \ref{lem:cois}), it seems to us that we should not restrict the study of coisotropic submanifolds to those inside Poisson manifolds, and, even more, we should in fact consider the case of coisotropic submanifolds in general (i.e.~non-necessarily standard) Jacobi manifolds.

One purpose of the present article is to extend the construction of an $L_\infty$-algebra attached to a coisotropic submanifold $S$ to the Jacobi case, generalizing analogous constructions in
\cite{OP2005} (symplectic case), \cite{CF2007} (Poisson case), \cite{LO2012} (locally conformal symplectic case).
Our construction encompasses all the known cases as special cases and
reveals the prominent role of the gauge algebroid $\der L$ of a line bundle $L$.
In all previous cases $L$ is a trivial line bundle while it is not necessarily so for general Jacobi manifolds. As a new special case, our construction canonically applies to coisotropic submanifolds in
any (not necessarily co-orientable) contact manifold. We also provide a
global tensorial description of our $L_\infty$-algebra, in the spirit of \cite{CF2007}, originally given in the language of (formal) $Q$-manifolds \cite{AKSZ1997} for the symplectic case (see \cite[Appendix]{OP2005}).

The $L_\infty$-algebra of a coisotropic submanifold $S$ governs the formal deformation problem of $S$. In this respect, another purpose of the present article is to present necessary and sufficient conditions under which the $L_\infty$-algebra of $S$ governs the \emph{non-formal} deformation problem as well. Our Proposition~\ref{prop:fana} extends - even in the Poisson setting - the sufficient condition given by Sch\"atz and Zambon in \cite{SZ2012} to a necessary and sufficient condition. We also discuss
the relation between Hamiltonian equivalence of coisotropic sections and gauge equivalence of Maurer-Cartan elements. We obtain a satisfactory description of this relation (Proposition \ref{prop:MC_gauge}) and discuss its consequences (Theorem \ref{prop:fana1} and Corollary \ref{cor:infequi}).

Note that Jacobi manifolds can be understood as homogeneous Poisson manifolds (of a special kind) via the ``Poissonization construction'' (see, e.g.~\cite{DLM1991, Marle1991}). However, not all coisotropic submanifolds in the Poissonization come from coisotropic submanifolds in the original Jacobi manifold. On the other hand, if we regard a Poisson manifold as a Jacobi manifold, all its coisotropic submanifolds are coisotropic in the Jacobi sense as well. In particular, the deformation problem of a coisotropic submanifold in a Jacobi manifold is genuinely more general than its analogue in the Poisson setting.

Our paper is organised as follows. In Section \ref{sec:abstract_jac_mfd} we attach important algebraic and geometric structures to a Jacobi manifold. Our approach, via gauge algebroids and first order multi-differential calculus on non-trivial line bundles, unifies and simplifies previous, analogous constructions for Poisson manifolds and locally conformal symplectic manifolds.
In Section \ref{sec:coiso}, using results in Section \ref{sec:abstract_jac_mfd}, we attach an $L_\infty$-algebra to any closed coisotropic submanifold in a Jacobi manifold.
In Section \ref{sec:deform} we study the deformation problem of coisotropic submanifolds. In particular we discuss the relation between smooth coisotropic deformations and formal coisotropic deformations as well as the moduli problem under Hamiltonian equivalence.
In Section \ref{sec:contact} we apply the theory to the contact case, which is, in a sense, analogous to the symplectic case analysed by Oh-Park \cite{OP2005}. In Section \ref{subsec:} we present an example of a coisotropic submanifold in a contact manifold whose deformation problem is obstructed.

Finally, the paper contains two appendices. The first one collects some facts about gauge algebroids and Schouten-Jacobi algebras that are needed in the main body of the paper. 
In the second one we compute explicitly the multi-brackets in the $L_\infty$-algebra of a pre-contact manifold, thus providing a proof of Theorem \ref{prop:multi}.

\section[Jacobi manifolds and associated algebraic and geometric structures]{Jacobi manifolds and associated algebraic and geometric structures}\label{sec:abstract_jac_mfd}

In this section we recall the definition of Jacobi manifolds and present important examples (Definition \ref{definition:abstractj}, Examples \ref{ex:1}) of them. Our primary sources are \cite{Kiri1976}, \cite{Lich1978}, \cite{Marle1991}, \cite{GM2001}, and the recent paper by Crainic and Salazar \cite{CS2013} whose philosophy/approach \emph{\`a la} Kirillov we adopt. Accordingly, we retain the terms \emph{standard Jacobi manifolds} for Jacobi manifolds in the sense of Lichnerowicz. Generically non-trivial line bundles and first order multi-differential calculus on them 
play a prominent role in Jacobi geometry. We also associate important algebraic and geometric structures with Jacobi manifolds. Namely, we recall the notion of \emph{Jacobi algebroid} (see \cite{GM2001} and \cite{IM2000a} for the equivalent notion of \emph{Lie algebroid with a $1$-cocycle}), but we adopt a slightly more general approach to incorporate the non-trivial line bundle case. We discuss the existence of a Jacobi algebroid structure on the first jet bundle $J^1 L$ of the Jacobi bundle of a Jacobi manifold $(M, L, \{-,-\})$ (Example \ref{ex:jacb2}), first discovered by Kerbrat and Souici-Benhammadi in the \emph{standard} case $L = M \times \R $ \cite{KS1993} (see \cite{CS2013} for the general case). Finally, we discuss the notion of morphisms of Jacobi manifolds. 

\subsection{Jacobi manifolds and their canonical bi-linear forms}\label{sec:a_jac_mfd}

Let $M$ be a smooth manifold.

\begin{definition}\label{definition:abstractj}
	A \emph{Jacobi structure} on $M$ is a pair $(L,\{-,-\})$ where $L\to M$ is a (generically non-trivial) line bundle, and $\{-,-\} : \Gamma (L) \times \Gamma (L) \rightarrow \Gamma (L)$ is a Lie bracket which, moreover, is a first order differential operator in both entries. A \emph{Jacobi manifold} is a manifold equipped with a Jacobi structure. The bundle $L$ and the bracket $\{-,-\}$ will be referred to as the \emph{Jacobi bundle} and the \emph{Jacobi bracket} respectively.
\end{definition}

A Jacobi bracket $\{-,-\}$ is, by definition, a (first order) bi-differential operator. We collect basic facts, including our notations and conventions, about (multi-)differential operators in Appendix \ref{sec:app_0}. In the following, we will often refer to it for details.

\begin{example}\label{ex:1} \
	
	\begin{enumerate}
		\item Any (possibly non-coorientable) contact manifold $(M,C)$ is naturally equipped with a Jacobi structure, with Jacobi bundle given by the (possibly non-trivial) line bundle $TM/C$ (see Section \ref{sec:contact}).
		
		\item Recall that a locally conformal symplectic (l.c.s.) manifold is naturally equipped with a standard Jacobi structure sometimes called the \emph{associated locally conformal Poisson structure}. There is a slight generalization of a l.c.s.~manifold in the same spirit as Jacobi manifolds (see Appendix A of \cite{V2014}). Call it an \emph{l.c.s.~manifold} as well. Then, any l.c.s.~manifold is naturally equipped with a Jacobi structure \cite{V2014}.
		\item Let $\{ \om_t\}_{t \in I}$ be a smooth l.c.s.~deformation of a l.c.s.~form $\om_0$ on a manifold $M$, where $I$ is an open interval in $\R$ containing $0$. Denote by $J_t$ the standard Jacobi structure on $M$ associated with $ \om_t$, and let $\tilde J: C^\infty (M \times I) \times C^\infty (M\times I)\to C^\infty(M\times I)$ be defined by $\tilde J (\tilde g, \tilde f) (x, t) : = J_t (\tilde f (-, t), \tilde g (-, t)) (x)$.
		Then it is not hard to verify that $(M\times I, \tilde J)$ is a standard Jacobi manifold.
	\end{enumerate}
\end{example}

Let $(M, L, \{-,-\})$ be a Jacobi manifold and $\lambda \in \Gamma (L)$. Then
$\Delta_\lambda := \{ \lambda , -\} $ 
is a derivation of $L$. The symbol of $\Delta_\lambda$ (see Appendix \ref{sec:app_0}) will be denoted by $X_\lambda$.

\begin{remark}\label{rem:jalgebra}
	By definition, a Jacobi bracket $\{-,-\}$ on sections of a line bundle $L \to M $ satisfies the following generalized Leibniz rule
	\begin{equation}
		\{ \lambda , f \mu \} = f \{ \lambda , \mu \} + X_\lambda (f) \mu,
	\end{equation}
	$\lambda, \mu \in \Gamma (L)$, $f \in C^\infty (M)$. 
\end{remark}

Denote by $J^1 L$ the bundle of $1$-jets of sections of $L$ and let $j^1 : \Gamma (L) \to \Gamma (J^1 L)$ be the first jet prolongation. The bi-differential operator $\{-,-\}$ can be interpreted as an $L$-valued, skew-symmetric, bi-linear form $J  : \wedge^2 J^1 L \to L$. Namely, $J$ is uniquely determined by
\[
J  (j^1 \lambda, j^1 \mu) = \{\lambda, \mu \},
\]
for all $\lambda, \mu \in \Gamma (L)$. \color{black} 

\begin{remark}
	As $\{-,-\}$ and $J$ contain the same information, we will sometimes identify them and write $J \equiv \{-,-\}$. For instance we will write $[J, \square]^{SJ}$ for the \emph{Schouten-Jacobi bracket} of $\{-,-\}$ and another (first order) multi-differential operator $\square$ (see Appendix \ref{sec:app_0}). On the other hand, we will always use the symbol $J$ for the bi-linear form $\wedge^2 J^1 L \to L$, and we will always use the symbol $\{-,-\}$ when we want to act with the bracket on sections of $L$.
\end{remark}
\color{black}

Denote by $\der L = \mathrm{Hom} (J^1 L, L)$ the gauge algebroid of the line bundle $L$ (see Appendix \ref{sec:app_0} for details). Then, the bi-linear form $J$ determines an obvious morphism of vector bundles $J ^\# : J^1 L \to \der L$, defined by $ J ^\# (\alpha) \lambda := J  (\alpha, j^1\lambda )$, where $\alpha \in \Gamma (J^1 L)$ and $\lambda \in \Gamma (L)$. The \emph{bi-symbol} $\Lambda_J$ of $\{-,-\}$ will be also useful. It is defined as follows. \color{black} Recall that there is a natural vector bundle embedding $\gamma : T^\ast M \otimes L \to J^1 L$, sometimes called the \emph{co-symbol}, well-defined by $\gamma (df \otimes \lambda) := j^1 (f \lambda) - f j^1 \lambda$, for all $f \in C^\infty (M)$, and $\lambda \in \Gamma (L)$. The co-symbol fits in the exact sequence
\[
0 \longrightarrow T^\ast M \otimes L \overset{\gamma}{\longrightarrow} J^1 L \longrightarrow L \longrightarrow 0,
\]
where $J^1 L \to L$ is the natural projection. Then
\color{black}
$
\Lambda_J : \wedge^2 (T^\ast M \otimes L) \to L
$
is the bi-linear form obtained by restricting $J $ to $T^\ast M \otimes L$ regarded as a subbundle of $J^1 L$ via the \emph{co-symbol}. Namely,
\[
\Lambda_J (\eta, \theta) := J  (\gamma (\eta), \gamma (\theta)),
\]
for all $\eta,\theta \in T^\ast M \otimes L$. It immediately follows from the definition that
\begin{equation}\label{eq:LambdaJ}
\Lambda_J (df\otimes \lambda, dg\otimes \mu) = \{f \lambda, g\mu \} -fg \{\lambda, \mu\} - f X_\lambda (g) \mu + g X_\mu (f)\lambda = \left( X_{f\lambda} (g) -f X_\lambda (g)\right) \mu,
\end{equation}
where $f,g \in C^\infty (M)$, and $\lambda, \mu \in \Gamma (L)$.

The skew-symmetric form $\Lambda_J$ determines an obvious morphism of vector bundles
$ \Lambda_J^\# : T^\ast M \otimes L \to TM$, implicitly defined by
$
\langle \Lambda_J^\# (\eta \otimes \lambda) , \theta \rangle \mu := \Lambda_J (\eta \otimes \lambda , \theta \otimes \mu)
$,
where $\eta, \theta \in \Omega^1 (M)$, $\lambda, \mu \in \Gamma (L)$, and $\langle -,- \rangle$ is the duality pairing. In other words,
\begin{equation}\label{eq:Lambdash}
	\Lambda_J^\# (df \otimes \lambda) = X_{f\lambda} -f X_\lambda,
\end{equation}
$f \in C^\infty (M)$, $\lambda \in \Gamma (L)$. The morphism $\Lambda^\#_J$ can be alternatively defined as follows. Recall that $\der L$ projects onto $TM$ via the symbol $\sigma$. It is easy to see that the diagram
\[
\xymatrix{ T^\ast M \otimes L \ar[r]^-{\Lambda_J^\#} \ar[d]_-{\gamma} & TM \\
	J^1 L \ar[r]^-{J^\#} & \der L \ar[u]_-{\sigma}}
\]
commutes, i.e.~$\Lambda^\#_J = \sigma \circ J^\# \circ \gamma$, which can be used as an alternative definition of $\Lambda^\#_J$. Finally, note that
\[
(J ^\# \circ \gamma)(df \otimes \lambda) = \Delta_{f\lambda} -f \Delta_\lambda.
\]

\subsection{Jacobi algebroid associated with a Jacobi manifold}

\begin{definition}\label{def:jacb} A \emph{Jacobi algebroid} is a pair $(A,L)$ where $A \to M$ is a Lie algebroid, and $L \to M$ is a line bundle equipped with a representation of $A$.
\end{definition}

\begin{remark}
	Jacobi algebroids are equivalent to Grabowski's \emph{Kirillov algebroids} \cite[Section 8]{G2013}.
\end{remark}

Let $A \to M$ be a Lie algebroid with anchor $\rho$ and Lie bracket $[-,-]_A$, and let $E \to M$ be a vector bundle equipped with a representation of $A$. In the following we denote by $(\Gamma (\wedge^\bullet A^\ast), d_A)$ the de Rham complex of $A$ and by $(\Gamma (\wedge^\bullet A^\ast \otimes E), d_{A, E})$ the \emph{de Rham complex of $A$ with values in $E$}. Its cohomology, the \emph{de Rham cohomology of $A$ with values in $E$}, will be denoted by $H(A,E)$.

Now, let $M$ be a manifold and let $L \to M$ be a line bundle. Denote by $J_1 L$ the dual bundle of $J^1 L$. Sections of $J_1 L$ are first order differential operators $\Gamma (L) \to C^\infty (M)$. Moreover, denote by $\Der^\bullet L = \Gamma (\wedge^\bullet J_1 L \otimes L)$ the space of alternating, first order multi-differential operators $\Gamma (L) \times \cdots \times \Gamma (L) \to \Gamma (L)$ (see Appendix \ref{sec:app_0} for more details).

\begin{example}\label{ex:jacb2} (cf.~\cite[Theorem 1]{KS1993}, \cite[(2.7)]{IM2000}, \cite[Theorem 13]{GM2001}) Let $(M, L, J \equiv \{-,-\})$ be a Jacobi manifold. It is not hard to see (see, e.g., \cite{CS2013}) that there is a unique Jacobi algebroid structure on $(J^1 L, L)$ with anchor $\rho_J$, Lie bracket $[-,-]_J$, and flat $J^1 L$-connection $\nabla^J$ in $L$ such that
	\begin{align}
		\rho_J (j^1 \lambda) & = X_\lambda, \nonumber \\
		[j^1 \lambda , j^1 \mu]_{J} & = j^1 \{\lambda , \mu \}, \label{eq:jaclie2}\\
		\nabla^J_{j^1 \lambda} \mu & = \{\lambda , \mu \}, \nonumber
	\end{align}
	for all $\lambda, \mu \in \Gamma (L)$. If $\psi, \chi \in \Gamma (J^1 L)$ are generic sections, we have
	\[
	\rho_J (\psi) = \sigma (J^\sharp \psi)
	\] 
	and
	\begin{equation}\label{eq:Liej}
		[\psi, \chi]_J = \Ll_{J^\sharp \psi} \chi - \Ll_{J^\sharp \chi} \psi - j^1 J(\psi, \chi).
	\end{equation}
\end{example} 

\begin{lemma}\label{lem:GM1} Let $J \in \Der^2 L$ be an alternating, first order bi-differential operator: $J: \Gamma (L) \times \Gamma (L) \to \Gamma (L)$. Then
	\begin{enumerate}
		\item for all $\lambda, \mu \in \Gamma (L)$,
		\begin{equation}\label{eq:Jlm}
			J(\lambda, \mu) = - [[J, \lambda] ^{SJ}, \mu]^{SJ}.
		\end{equation}
		\item (cf. \cite[Theorem 1.b, (28), (29)]{GM2001}) $J$ is a Jacobi bracket, i.e.~it defines a Lie algebra structure on $\Gamma (L)$ iff
		\begin{equation}
			[J,J]^{SJ} = 0, \label{eq:gm2}
		\end{equation}
	\end{enumerate}
	where $[-,-]^{SJ}$ is the Schouten-Jacobi bracket (see Appendix \ref{sec:app_0}).
\end{lemma}

\begin{proof} The first assertion is a consequence of the explicit form of the\linebreak Schouten-Jacobi bracket. The second assertion is a particular case of Theorem 3.3 in \cite{LMS1992}.
\end{proof}

\begin{remark}\label{rem:multi_trivial}
	Denote by $\mathfrak{X}^\bullet (M) = \bigoplus_k \mathfrak{X}^k (M)$ the space of (skew-sym\-met\-ric) multi-vector fields on $M$. When $L = \R_M := M \times \R$, the trivial line bundle, then the space $\Der^{k+1} L$ of alternating first order multi-differential operators on $\Gamma (L)$ with $k+1$ entries, identifies with $\mathfrak X^{k+1} (M) \oplus \mathfrak X^k(M)$ (see Appendix \ref{sec:app_0}). In particular, an alternating, first order bi-differential operator $J$ identifies with a pair $(\Lambda, \Gamma)$ where $\Lambda$ is a bi-vector field and $\Gamma$ is a vector field on $M$. In this case, Equation (\ref{eq:gm2}) is equivalent to
	\[
	[\Gamma, \Lambda]^{SN} = 0 \quad \text{and} \quad [\Lambda,\Lambda]^{SN} = 2 \Lambda \wedge \Gamma
	\]
	where $[-,-]^{SN}$ is the Schouten-Nijenhuis bracket on multi-vectors.
\end{remark}

\begin{remark}\label{rem:cohalg} Let $(M, \pi)$ be a Poisson manifold, with Poisson bi-vector $\pi$, and Poisson bracket $\{-,-\}_\pi$. The differential $d_\pi := [\pi, -]^{SN} : \mathfrak{X}^\bullet (M) \to \mathfrak{X}^\bullet (M)$ has been introduced by Lichnerowicz. The cohomology of $(\mathfrak{X}^\bullet (M), d_\pi)$ is the Lichnerowicz-Poisson cohomology of $(M, \pi)$. For more general Jacobi manifolds $(M, L, J \equiv \{-,-\})$
	it is natural to replace multi-vectors with 
	multi-differential operators, i.e.~elements of $ \Der^\bullet L$, and the Lichnerowicz-Poisson differential by the differential $d_J := [J, -]^{SJ}$. The resultant cohomology is called the \emph{Chevalley-Eilenberg cohomology} of $(M, L, \{-,-\})$ \cite{GL1984, Lich1978}.
	Furthermore, the action of $ (\Der^\bullet L)[1]$ on $ \Gamma (\wedge^\bullet J_1 L)$ (see Appendix \ref{sec:app_0}) gives rise to another cohomology, namely the cohomology of the complex $( \Gamma (\wedge^\bullet J_1 L), X_J)$, also called the \emph{Lichnerowicz-Jacobi cohomology of $(M, L, \{-,-\})$} (see, e.g., \cite{LLMP1999}). It is easy to see that the complex $( \Gamma (\wedge^\bullet J_1 L), X_J)$ is nothing but the de Rham complex of the Lie algebroid $(J^1 L , \rho_J, [-,-]_J)$. Similarly, the complex $( \Der^\bullet L, d_J)$ is the de Rham complex of $(J^1 L , \rho_J, [-,-]_J)$ with values in $L$.
\end{remark}

\subsection{Morphisms of Jacobi manifolds}
Let $(M_1, L_1, \{-,-\}_1)$ and $(M_2, L_2, \{-,-\}_2)$ be Jacobi manifolds
\begin{definition}
	\label{definition:Jacobi_mfd_morphism}
	A \emph{morphism of Jacobi manifolds}, or a \emph{Jacobi map},
	\[
	(M_1,L_1,\{-,-\}_1) \rightarrow (M_2,L_2,\{-,-\}_2)
	\]
	is a vector bundle morphism $\phi:L_1\to L_2$, covering a smooth map $\underline{\phi}: M_1 \to M_2$, such that $\phi$ is an isomorphism on fibers, and $ \phi^\ast \{ \lambda, \mu \}_2 = \{ \phi^\ast \lambda , \phi^\ast \mu \}_1$ for all $\lambda,\mu \in\Gamma(L_2)$.
\end{definition}

\begin{definition}
	\label{def:Jacobi_mfd_imorphism}
	An \emph{infinitesimal automorphism}, or a \emph{Jacobi derivation}, of a Jacobi manifold $(M,L,\{-,-\})$ is a derivation $\Delta$ of the line bundle $L$, equivalently, a section of the gauge algebroid $\der L$ of $L$, such that $\Delta$ generates a flow by automorphisms of $(M, L, \{-,-\})$  (see Appendix \ref{sec:app_0}). A \emph{Jacobi vector field} is the symbol of a Jacobi derivation.
\end{definition}

\begin{remark}\label{rem:inf_aut}
	Let $\Delta$ be a derivation of $L$, let $\{ \varphi_t \}$ be its flow, and let $\square$ be a first order multi-differential operator on $L$ with $k $ entries, i.e.~$\square \in \Der^k L$. It is easy to see that (similarly as for vector fields)
	\begin{equation}\label{Lie1}
		\left. \frac{d}{dt} \right |_{t=0} (\varphi_t)_\ast \square = [\square, \Delta]^{SJ}
	\end{equation}
	where $\varphi_\ast \square$ denotes the \emph{push forward} of $\square $ along a line bundle isomorphism $\varphi : L \to L^\prime$, defined by
	$(\varphi_\ast \square) (\lambda_1^\prime, \dots, \lambda_{k}^\prime) := (\varphi^{-1})^\ast (\square (\varphi^\ast \lambda_1^\prime, \dots, \varphi^\ast \lambda_{k}^\prime))$, for all $\lambda_1^\prime, \dots, \lambda_k^\prime \in \Gamma (L^\prime)$ (see also Appendix \ref{sec:app_0} about pushing forward derivations along vector bundle morphisms). In particular, $\Delta$ is an infinitesimal automorphism of $(M,L,\{-,-\})$ if and only if $[J, \Delta]^{SJ} = 0$. Since
	\begin{equation}\label{eq:DeltaJ}
		[J, \Delta]^{SJ} (\lambda, \mu ) = \{ \Delta \lambda , \mu\} + \{ \lambda , \Delta \mu\} - \Delta \{\lambda , \mu \},
	\end{equation}
	we conclude that $\Delta$ is an infinitesimal automorphism of $(M,L,\{-,-\})$ iff
	\begin{equation}
		\Delta \{ \lambda, \mu \} = \{\Delta \lambda , \mu \} + \{ \lambda , \Delta \mu\}
	\end{equation}
	for all $\lambda, \mu \in \Gamma (L)$. In other words $\Delta$ is a \emph{derivation} of the Jacobi bracket.
\end{remark}

\begin{remark}
	More generally, let $\{ \Delta_t \}$ be a one parameter family of derivations of $L$, generating the one parameter family of automorphisms $\{ \varphi_t \}$, and let $\square \in \Der^\bullet L$. Then
	\begin{equation}\label{Lie2}
		\frac{d}{dt} (\varphi_t)_\ast \square = [(\varphi_t)_\ast \square, \Delta_t]^{SJ}.
	\end{equation}
\end{remark}

\begin{remark}\label{rem:conf}
	Definitions \ref{definition:Jacobi_mfd_morphism} and \ref{def:Jacobi_mfd_imorphism} encompass the notions of \emph{conformal morphisms} and \emph{infinitesimal conformal automorphisms} of standard Jacobi manifolds, respectively. In particular two standard Jacobi structures are \emph{conformally equivalent} if and only if they are isomorphic as Jacobi structures.
\end{remark}

Let $(M,L, J \equiv \{-,-\})$ be a Jacobi manifold and $\lambda \in \Gamma (L)$. Note that
\begin{equation}
	\Delta_\lambda = \{ \lambda , -\} = - [ J , \lambda]^{SJ}. \label{eq:trianglelambda}
\end{equation}
The Jacobi identity for the Jacobi bracket immediately implies that not only $\Delta_\lambda$ is a derivation of $L$, but even more, it is an infinitesimal automorphism of $(M,L,\{-,-\})$, called the \emph{Hamiltonian derivation associated with the section $\lambda$}. Similarly, the symbol $X_\lambda$ of $\Delta_\lambda$ will be called the \emph{Hamiltonian vector field associated with $\lambda$}. Clearly we have
\begin{equation}\label{eq:iso}
	[\Delta_\lambda , \Delta_\mu] = \Delta_{\{\lambda, \mu\}}, \quad \text{and} \quad [X_\lambda , X_\mu] = X_{\{\lambda, \mu\}},
\end{equation}
for all $\lambda, \mu \in \Gamma (L)$. Jacobi automorphisms $L \to L$ generated by Hamiltonian derivations will be called \emph{Hamiltonian automorphisms}. Similarly, diffeomorphisms $M \to M$ generated by Hamiltonian vector fields will be called \emph{Hamiltonian diffeomorphisms}.

\begin{example}\label{ex:jmap} Let $(M,L,\{-,-\})$ be a Jacobi manifold. The values of all Hamiltonian vector fields generate a distribution $\Kk \subset TM$ which is, generically, non-constant-dimensional. Distribution $\Kk$ is called \emph{the characteristic distribution} of $(M, L, \{-,-\})$.
	The Jacobi manifold $(M, L, \{-,-\})$ is said to be \emph{transitive} if its characteristic distribution $\Kk$ is the whole tangent bundle $TM$. Identity (\ref{eq:iso}) implies that $\Kk$ is involutive. Moreover, it is easy to see that $\Kk$ is constant-dimensional along the flow lines of a Hamiltonian vector field. Hence, it is completely integrable in the sense of Stefan and Sussmann. In particular, it defines a (singular) foliation, also denoted $\Kk$. Each leaf $\Cc$ of $\Kk$, is called a \emph{characteristic leaf} and possesses a unique transitive Jacobi structure defined by the restriction of the Jacobi bracket to $L|_{\Cc}$, see Corollary~\ref{cor:char}.2 for a precise expression. In other words, the inclusion $L|_\Cc \INTO L$ is a Jacobi map.
	Moreover, a transitive Jacobi manifold $(M, L, \{-,-\})$ is either an l.c.s.~manifold (if $\dim M$ is even) or a contact manifold (if $\dim M$ is odd) \cite{Kiri1976}.
\end{example} 

\section[Coisotropic submanifolds in Jacobi manifolds and their invariants]{Coisotropic submanifolds in Jacobi manifolds and their invariants}\label{sec:coiso}

In this section we propose some equivalent characterizations of coisotropic submanifolds $S$ in a Jacobi manifold $(M, L, \{-,-\})$ (Lemma \ref{lem:cois}, Corollary~\ref{cor:char}.(3)). Then we establish a one-to-one correspondence between coiso\-tropic submanifolds of $(M, L, \{-,-\})$ and certain Jacobi subalgebroids of the Jacobi algebroid $(J^1 L, L)$ (Proposition \ref{prop:conormal}). In particular, this yields a natural $L_\infty$-isomorphism class of $L_\infty$-algebras associated with each coiso\-tropic submanifold (Proposition \ref{prop:linfty} and Proposition \ref{prop:gauge_invariance}).

\subsection{Differential geometry of a coisotropic submanifold}

Let $(M, L, J \equiv \{-,-\})$ be a Jacobi manifold, and let $x \in M$. A subspace $T \subset T_xM$ is said to be \emph{coisotropic} (with respect to~the Jacobi structure $(L,J \equiv \{-,-\})$), if $\Lambda_J ^\# (T^0 \otimes L_{{x}}) \subset T$, where $T^0 \subset T^*_xM$ denotes the annihilator of $T$ (cf.~\cite[Definition 4.1]{ILMM1997}). Equivalently, $T^0 \otimes L_{{x}}$ is isotropic with respect to~the $L$-valued bi-linear form $\Lambda_J$.

A submanifold $S\subset M$ is called \emph{coisotropic} (with respect to~the Jacobi structure $(L,J \equiv \{-,-\})$), if its tangent space $T_xS$ is coisotropic for all $ x \in S$.

\begin{lemma}\label{lem:cois}
	Let $S \subset M$ be a submanifold, and let $\Gamma_S$ denote the set of sections $\lambda$ of the Jacobi bundle such that $\lambda |_S = 0$. The following three conditions are equivalent:
	\begin{enumerate}
		\item \label{1} $S$ is a coisotropic submanifold,
		\item \label{2} $\Gamma_S$ is a Lie subalgebra in $\Gamma (L)$,
		\item \label{3} $X_\lambda$ is tangent to $S$, for all $\lambda \in \Gamma_S$.
	\end{enumerate}
\end{lemma}

\begin{proof}
	Let $S \subset M$ be a submanifold. We may assume, without loss of generality, that $L$ is trivial. Then $\Gamma_S=I(S)\cdot\Gamma(L)$, where $I(S)$ denotes the ideal in $C^\infty(M)$ consisting of functions that vanish on $S$. In particular, if $\lambda$ is a generator of $\Gamma (L)$, then every section in $\Gamma_S$ is of the form $f \lambda$ for some $f \in I(S)$. Now, let $f,g \in I(S)$. Putting $\mu = \lambda$ in (\ref{eq:LambdaJ}) and restricting to $S$, we find
	\[
	\{ f\lambda, g\lambda \}|_S = \langle \Lambda_J^\# (df \otimes \lambda) , dg \rangle \lambda |_S .
	\]
	This shows that $(1) \Longleftrightarrow (2)$. The equivalence $(2) \Longleftrightarrow (3)$ follows from the identity $X_\lambda(f)\mu|_S=\{\lambda,f\mu\}|_S$, for all $\lambda\in\Gamma_S$, $\mu\in\Gamma(L)$, and $f\in I(S)$.
\end{proof}

Now, let $S \subset M$ be a coisotropic submanifold and let $T^0 S \subset T^\ast M |_S$ be the annihilator of $TS$. The (generically non constant-dimensional) distribution $\mathcal{K}_S := \Lambda^\#_J (T^0 S \otimes L) \subset TS$ on $S$ is called the \emph{characteristic distribution of $S$}.

\begin{remark}
	In view of (\ref{eq:Lambdash}), $\Kk_S$ is generated by the (restrictions to $S$ of) the Hamiltonian vector fields of the kind $X_\lambda$, with $\lambda \in \Gamma_S$.
\end{remark}

From Lemma \ref{lem:cois} one can easily derive the following

\begin{corollary}\label{cor:char}\
	\begin{enumerate}
		\item (cf.~\cite[\S 2]{CF2007}) The characteristic distribution $\mathcal K_S$ of any coisotropic submanifold $S$ is integrable (hence, it determines a foliation on $S$, called the \emph{characteristic foliation} of $S$).
		
		\item (cf.~\cite{Kiri1976}) Every characteristic leaf $\Cc$, i.e.~any leaf of the characteristic distribution $\Kk = \Kk_M$ has an induced Jacobi structure $(L|_\Cc, \{-,-\}_\Cc)$ well-defined by $\{ \lambda |_{\Cc}, \mu|_{\Cc}\}_\Cc= \{ \lambda, \mu \}|_C $, for all $\lambda, \mu \in \Gamma (L)$. The induced Jacobi structure is transitive.
		
		\item A submanifold $S \subset M$ is coisotropic, if and only if $TS \cap T\Cc$ is coisotropic in the tangent bundle $T\Cc$, for all characteristic leaves $\Cc$ intersecting $S$, where $\Cc$ is equipped with
		the induced Jacobi structure.
	\end{enumerate}
\end{corollary}

\begin{example}\label{ex:simul} \
	\begin{enumerate}
		\item Any coisotropic submanifold (in particular a Legendrian submanifold) in a contact manifold is a coisotropic submanifold with respect to~the associated Jacobi structure (see Section \ref{sec:cois_cont} for details).
		
		\item Let $S$ be a coisotropic submanifold of a Jacobi manifold $(M,L,\{-,-\})$, and let $X \in \mathfrak{X}(M)$ be a Jacobi vector field such that $X_x \notin T_x S$, for all $x\in S$. Then $\mathcal T$, the flowout of $S$ along $X$, is a coisotropic submanifold as well. Indeed, let $\{\phi_t\}$ be the flow of $X$. Clearly, whenever defined, $\phi_t (S)$ is a coisotropic submanifold, and the claim immediately follows from Lemma \ref{lem:cois}.
	\end{enumerate}
\end{example}

\subsection{Jacobi subalgebroid associated with a closed coisotropic submanifold}
We are interested in deformations of a \emph{closed} coisotropic submanifold, so, from now on, we assume that $S$ is a closed submanifold in a smooth manifold $M$. Let $A \to M$ be a Lie algebroid. Recall that a \emph{subalgebroid of $A$ over $S$} is a vector subbundle $B \to S$, with embeddings $j : B \INTO A$ and $\underline{j} : S \INTO M$, such that the anchor $\rho : A \to TM$ descends to a (necessarily unique) vector bundle morphism $\rho_B : B \to TS$, making diagram
\[
\xymatrix{B \ar[r]^-j \ar[d]_-{\rho_B} & A \ar[d]^-{\rho} \\
	TS \ar[r]^-{d \underline{j}} & TM}
\]
commutative and, moreover, for all $\beta, \beta^\prime \in \Gamma (B)$ there exists a (necessarily unique) section $[\beta,\beta^\prime]_B \in \Gamma (B)$ such that whenever $\alpha, \alpha^\prime \in \Gamma (A)$ are $j$-related to $\beta, \beta^\prime$ (i.e.~$j \circ \beta = \alpha \circ \underline{j}$, in other words $\alpha |_S = \beta$, and similarly for $\beta^\prime, \alpha^\prime$) then $[\alpha, \alpha^\prime]_A$ is $j$-related to $[\beta, \beta^\prime]_B$. In this case $B$, equipped with $\rho_B$ and $[-,-]_B$, is a Lie algebroid itself. One can also give a notion of \emph{Jacobi subalgebroid} as follows.

Let $(A,L)$ be a Jacobi algebroid with representation $\nabla$.

\begin{definition}
	A \emph{Jacobi subalgebroid of $(A,L)$ over S} is a pair $(B, \ell)$, where $B \to S$ is a Lie subalgebroid of $A$ over $S \subset M$, and $\ell := L|_S \to S$ is the pull-back line subbundle of $L$, such that $\nabla$ descends to a (necessarily unique) vector bundle morphism $\nabla|_\ell$ making diagram
	\[
	\xymatrix{B \ar[r]^-j \ar[d]_-{\nabla |_\ell} & A \ar[d]^-{\nabla} \\
		\der \ell \ar[r]^-{\der j_\ell} & \der L}
	\]
	commutative. Here $j_\ell : \ell \INTO L$ is the inclusion (see Appendix \ref{sec:app_0} for a definition of the morphism $\der j_\ell$).
\end{definition}

If $(B,\ell)$ is a Jacobi subalgebroid, then the restriction $\nabla|_\ell$ is a representation so that $(B, \ell)$, equipped with $\nabla|_\ell$, is a Jacobi algebroid itself.

Now, let $(M,L,J \equiv \{-,-\})$ be a Jacobi manifold, and let $S$ be a submanifold. In what follows, we denote by
\begin{itemize}
	\item $\ell :=L|_S$ the restricted line bundle,
	\item $NS:=TM|_S/TS$ the normal bundle of $S$ in $M$,
	\item $N^\ast S:= (NS)^\ast\cong T^0 S \subset T^\ast M$ the conormal bundle of $S$ in $M$,
	\item $N_\ell S := NS \otimes \ell^\ast$, and by
	\item $N_\ell{}^\ast S := (N_\ell S)^\ast = N^\ast S \otimes \ell$ the $\ell$-adjoint bundle of $NS$.
\end{itemize}
The vector bundle $N_\ell {}^\ast S$ will be also regarded as a vector subbundle of $(J^1 L)|_S$ via the vector bundle embedding
\begin{equation*}
	N_\ell {}^\ast S \lhook\joinrel\longrightarrow (T^\ast M \otimes L)|_S \overset{\gamma}{\longrightarrow} J^1 L|_S,
\end{equation*}
where $\gamma$ is the co-symbol. If $\lambda\in\Gamma(L)$, we have that $(j^1\lambda)|_S\in \Gamma (N_\ell {}^\ast S) $ if and only if $\lambda|_S=0$, i.e.~$\lambda\in\Gamma_S$.

The following proposition establishes a one-to-one correspondence be\-tween coisotropic submanifolds and certain Lie subalgebroids of $J^1 L$.

\begin{proposition}\label{prop:conormal}
	(cf.~\cite[Proposition 5.2]{IM2003}) The submanifold $S \subset M$ is coisotropic if and only if $(N_\ell {}^\ast S, \ell)$ is a Jacobi subalgebroid of $(J^1 L, L)$.
\end{proposition}

\begin{proof} Let $S \subset M$ be a coisotropic submanifold. We want to show that $N_\ell {}^\ast S$ is a Jacobi subalgebroid of $J^1 L$. We propose a proof which is shorter than the one in \cite{IM2003}. Since $S$ is coisotropic, we have
	\begin{equation}
		\rho_J ( N_\ell {}^\ast S ) \subset TS , \label{eq:anchor1}
	\end{equation}
	and similarly
	\begin{equation}
		\nabla^J ( N_\ell {}^\ast S ) \subset \der \ell . 
	\end{equation}
	Next we shall show that for any $\alpha, \beta \in \Gamma (J^1 L)$ such that $\alpha|_S , \beta |_S \in \Gamma (N_\ell {}^\ast S)$ we have
	\begin{equation}
		[\alpha, \beta]_{J} |_S \in \Gamma (N_\ell {}^\ast S). 
	\end{equation}
	First we note that if $\alpha |_S \in \Gamma (N_\ell {}^\ast S)$ then $\alpha = \sum f j^1 \lambda$ for some $\lambda \in \Gamma_S$. Using the Leibniz properties of the Jacobi bracket we can restrict to the case $\alpha, \beta \in j^1 \Gamma_S$. The latter case can be handled taking into account (\ref{eq:jaclie2}) and Lemma \ref{lem:cois}. Moreover, using (\ref{eq:Liej}), we easily check that
	\[
	[\alpha, \beta]_{J} |_S = 0 \text{ if } \alpha|_S = 0 \text{ and } \beta |_S \in \Gamma (N_\ell {}^\ast S).
	\]
	This completes the ``only if part'' of the proof.
	
	To prove the ``if part'' it suffices to note that condition (\ref{eq:anchor1}), regarded as a condition on the image of the anchor map of the Lie subalgebroid $N_\ell {}^\ast S$, implies, in view of (\ref{eq:Liej}), that
	$S$ is a coisotropic submanifold.
\end{proof}

\begin{remark}\label{rem:defcois}\
	Different versions of Proposition \ref{prop:conormal} were proved for the Poisson case \cite[Proposition 3.1.3]{Weinstein1988}, \cite[Proposition 5.1]{Cattaneo}, \cite[Theorem 10.4.2]{Mackenzie2005}.
\end{remark}

\subsection{$L_\infty$-algebra associated with a coisotropic submanifold}
Let $M$ be as above, and let $S \subset M$ be a closed submanifold.
Let
\[
P_0 : \Gamma (J_1 L) \longrightarrow \Gamma (N_\ell S)
\]
be the projection adjoint to the embedding 
$$\gamma : N_\ell {}^\ast S \INTO J^1 L,\ \text{i.e.}\ \la P_0 (\Delta)_x , \alpha_x \ra = \la \Delta_x , \gamma (\alpha_x) \ra,$$
where $\Delta \in \Gamma (J_1 L)$, $\alpha \in \Gamma (N_\ell {}^\ast S)$, and $x \in S$. Tensorizing by $\Gamma (L)$ we also get a projection
\[
P : \Der L \longrightarrow \Gamma (NS).
\]
It is not hard to see that $P$ coincides with the composition
\begin{equation}\label{eq:seq_P}
	\Der L \overset{\sigma}{\longrightarrow} \mathfrak{X}(M) \longrightarrow \Gamma (TM |_S) \longrightarrow \Gamma (NS),
\end{equation}
where the second arrow is the restriction, and the last arrow is the canonical projection. Projection $P_0$ extends uniquely to a (degree zero) morphism of graded algebras $ \Gamma (\wedge^\bullet J_1 L) \to \Gamma (\wedge^\bullet N_\ell S)$ which we denote again by $P_0$. Similarly, $P$ extends uniquely to a (degree zero) morphism of graded modules $ (\Der^\bullet L)[1] \to \Gamma (\wedge^\bullet N_\ell S \otimes \ell)[1]$ which we denote again by $P$. As in the Poisson case (see, e.g., \cite{CS2008}), the projection $P :  (\Der^\bullet L)[1] \to \Gamma (\wedge^\bullet N_\ell S \otimes \ell)[1]$ allows to formulate a further characterization of coisotropic submanifolds.

\begin{proposition}\label{prop:PJ=0}
	The submanifold $S$ is coisotropic if and only if $P (J) = 0$.
\end{proposition}

\begin{remark}\label{rem:PJ}
	Let $S \subset M$ be any submanifold, then $P (J)$ does only depend on the bi-symbol $\Lambda_J$ of $J$. To see this, note, first of all, that the symbol $\sigma : \Der L \to \mathfrak X (M)$ induces an obvious projection $ \Der^\bullet L \to \Gamma (\wedge^\bullet (TM \otimes L^\ast) \otimes L)$. Moreover, in view of its very definition, $P :  (\Der^\bullet L)[1] \to \Gamma (\wedge^\bullet N_\ell S \otimes \ell)[1]$ descends to an obvious projection
	\[
	\Gamma (\wedge^\bullet (TM \otimes L^\ast) \otimes L)[1] \longrightarrow \Gamma (\wedge^\bullet N_\ell S \otimes \ell)[1],
	\]
	which, abusing the notation, we denote again by $P$. Now, recall that $\Lambda_J \in \Gamma (\wedge^2 (TM \otimes L^\ast) \otimes L)$. It immediately follows from the definition of $P$ that, actually,
	\[
	P(J) = P(\Lambda_J).
	\]
	In particular $S$ is coisotropic if and only if $P (\Lambda_J)=0$.
\end{remark}

From now on we assume that $S$ is coisotropic. In this case, the Jacobi algebroid structure on $(N_\ell {}^\ast S, \ell)$ (Proposition \ref{prop:conormal}) turns the graded space $\Gamma (\wedge^\bullet N_\ell S \otimes \ell)$ into the de Rham complex of $N_\ell {}^\ast S$, with values in $\ell$. To express
the differential $d_{N_\ell {}^\ast S, \ell}$ in terms of the
differential $d_J = [J,-]^{SJ}$ on $\Der^\bullet L$ it suffices to find
a right inverse $I : \Gamma (\wedge^\bullet N_\ell S \otimes \ell)[1] \to (\Der^\bullet L)[1]$ of $P$. However, there is no natural way to do this unless further structure is available. In what follows we use a \emph{fat tubular neighborhood} as an additional structure. Before giving a definition, recall that a \emph{tubular neighborhood of $S$} is an embedding of the normal bundle $NS$ into $M$ which identifies the zero section $\mathbf 0$ of $NS \to S$ with the inclusion $i : S \INTO M$. Denote by $\pi : NS \to S$ the projection and consider the pull-back line bundle $L_{NS}:= \pi^\ast \ell = NS \times_S \ell$ over $NS$. Moreover, let $i_L : \ell \INTO L$ be the inclusion.

\begin{definition}
	A \emph{fat tubular neighborhood of $\ell \to S$ in $L \to M$ over a tubular neighborhood $\underline \tau : NS \INTO M$} is an embedding $\tau : L_{NS} \INTO L$ of vector bundles over $\underline \tau : NS \INTO M$ such that the diagram
	\[
	\xymatrix@R=20pt{ L_{NS} = \pi^\ast \ell \ar[rr]^(.5){\tau} \ar[dd] \ar[dr] & & L \ar[dd] \\
		& \ell \ar@/^/@<0.7 ex>[ul] \ar[ur]_-{i_L} \ar[dd]& \\
		NS \ar[dr]^-{\pi} \ar[rr]|!{[ur];[dr]}\hole^(.5){\underline{\tau}} & & M \\
		& S \ar@/^/@<0.7 ex>[ul]^-{\mathbf 0} \ar[ur]_-{i} & }
	\]
	commutes.
\end{definition}

In particular, it follows from the above definition that $\tau$ is an isomorphism when restricted to fibers. A fat tubular neighborhood can be understood as a ``tubular neighborhood in the category of line bundles''. In the following we regard $S$ as a submanifold of $NS$ identifying it with the image of the zero section $\mathbf 0 : S \to NS$.

\begin{lemma}
	There exist fat tubular neighborhoods of $\ell$ in $L$.
\end{lemma}
\begin{proof}
	Since fibers of $NS \to S$ are contractible, for every vector bundle $V \to NS$ over $NS$ there is a, generically non-canonical, isomorphism of vector bundles $NS \times_S V |_S \cong V $ over the identity of $NS$. Now, let $\underline \tau : NS \INTO M$ be a tubular neighborhood of $S$. According to the above remark, the pull-back bundle $\underline \tau {}^\ast L \to NS$ is (non-canonically) isomorphic to $L_{NS}$. Pick any isomorphism $\phi : L_{NS} \to \underline \tau {}^\ast L $. Then the composition
	\[
	L_{NS} \overset{\phi}{\longrightarrow} \underline \tau {}^\ast L \longrightarrow L,
	\]
	where the second arrow is the canonical map, is a fat tubular neighborhood of $\ell$ over $\underline \tau$.
\end{proof}

Choose once for all a fat tubular neighborhood $\tau : L_{NS} \INTO L$ of $\ell$ over a tubular neighborhood $\underline \tau : NS \INTO M$ of $S$. We identify $NS$ with the open neighborhood $\underline \tau (NS)$ of $S$ in $M$. Similarly, we identify $L_{NS}$ with $L|_{\underline \tau (NS)}$. In particular $NS$ inherits from $\underline \tau (NS)$ a Jacobi structure with Jacobi bundle given by $L_{NS}$. Abusing the notation we denote by $J$ again the Jacobi bracket on $\Gamma (L_{NS})$. Moreover, in view of Proposition \ref{prop:PJ=0}, there is a projection $P : (\Der^\bullet L_{NS})[1] \to \Gamma (\wedge^\bullet N_\ell S \otimes \ell)[1]$ such that $P (J) = 0$.

Now, regard the vertical bundle $V (NS) := \ker d\pi$ as a Lie algebroid and note preliminarily that
\begin{enumerate}
	\item There is a natural splitting $T(NS)|_S = TS \oplus NS$, where the projection $T(NS)|_S \to TS$ is $d \pi$, while the projection $T (NS)|_S \to NS$ is the natural one. In particular, sections of $NS$ can be understood as vector fields on $NS$ along the submanifold $S$ and vertical with respect to~$\pi$.
	\item Since $\pi : NS \to S$ is a vector bundle, the vertical bundle $V(NS)$ identifies canonically with the induced bundle $\pi^\ast NS \to NS$. In particular, there is an embedding $\pi^\ast : \Gamma (NS) \INTO \mathfrak{X} (NS)$ that takes a section $\nu$ of $NS$ to the unique vertical vector field $\pi^\ast \nu$ on $NS$, which is constant along the fibers of $\pi$, and agrees with $\nu$ on $S$.
	\item Since $L_{NS} = \pi^\ast \ell = NS \times_S \ell$, there is a natural flat connection $\mathbb D$ in {$L_{NS}$}, along the Lie algebroid $V(NS)$, uniquely determined by $\mathbb D_X \pi^\ast \lambda = 0$, for all vertical vector fields $X$ on $NS$, and all fiber-wise constant sections $\pi^\ast \lambda$ of $L_{NS}$, $\lambda \in \Gamma (\ell)$.
\end{enumerate}

With these preliminary remarks we are finally ready to define a right inverse $I : \Gamma (\wedge^\bullet N_\ell S \otimes \ell)[1] \to (\Der^\bullet L_{NS})[1]$ of $P : (\Der^\bullet L_{NS})[1] \to \Gamma (\wedge^\bullet N_\ell S \otimes \ell)[1]$. First of all, let
\[
I : \Gamma (NS) \INTO  \Der L_{NS}
\]
be the embedding given by $I (\nu) := \mathbb{D}_{\pi^\ast \nu}$. Tensorizing it by $\Gamma (L_{NS}^\ast)$ we also get an embedding
\[
I_0 : \Gamma (N_\ell S) \INTO \Gamma (J_1 L_{NS}).
\]
The inclusion $I_0$ extends uniquely to a (degree zero) morphism of graded algebras $\Gamma (\wedge^\bullet N_\ell S) \to  \Gamma (\wedge^\bullet J_1 L_{NS})$ which we denote again by $I_0$. Similarly, $I$ extends uniquely to a (degree zero) morphism of graded modules $\Gamma (\wedge^\bullet N_\ell S \otimes \ell)[1] \to (\Der^\bullet L_{NS})[1]$ which we denote again by $I$. It is straightforward to check that
\[
P_0 \circ I_0 = \operatorname{id} \quad \text{and} \quad P \circ I = \operatorname{id}.
\]
Using $I$ and the explicit expression for the Schouten-Jacobi bracket, one can check that
\begin{equation}\label{eq:d*2}
	d_{N_\ell {}^\ast S, \ell} \alpha = (P \circ d_J \circ I )(\alpha) = P [J, I (\alpha)]^{SJ}
\end{equation}
for all $\alpha \in \Gamma (\wedge^\bullet N_\ell S \otimes \ell)[1]$.

The rightmost hand side of (\ref{eq:d*2}) reminds us of the Voronov construction of $L_\infty$-algebras via derived brackets. We refer the reader to \cite{Voronov2005} for details. Our conventions about $L_\infty$-algebras are the same as those in \cite{Voronov2005}. In particular, multi-brackets in $L_\infty$-algebras in this paper will always be (graded) symmetric. Now, using the derived bracket construction, we are going to define an $L_\infty$-algebra structure $\{ \mathfrak{m}_k \}$
on $\Gamma (\wedge^\bullet N_\ell S \otimes \ell)[1]$ whose first (unary) bracket $\mathfrak{m}_1 $ coincides with the differential $d_{N_\ell {}^\ast S , \ell}$. The following Proposition is an analogue of Lemma 2.2 in \cite{FZ2012}, see also \cite{CF2007} and \cite[Appendix]{OP2005}.

\begin{proposition}\label{prop:linfty}
	Let $I: \Gamma (\wedge^\bullet N_\ell S \otimes \ell)[1] \INTO (\Der^\bullet L_{NS})[1]$ be the embedding defined above.
	There is an $L_\infty$-algebra structure on $\Gamma (\wedge^\bullet N_\ell S \otimes \ell)[1]$ given by the following family of graded multi-linear maps $\mathfrak{m} _k : \Gamma (\wedge^\bullet N_\ell S \otimes \ell)[1]^{\otimes k} \to \Gamma (\wedge^\bullet N_\ell S \otimes \ell)[1]$
	\begin{equation}\label{eq:higher_derived_brackets}
		\mathfrak{m}_k (\xi_1, \dots, \xi_k): =P [\cdots [[J, I(\xi_1)] ^{SJ}, I(\xi_2)]^{SJ}\cdots , I (\xi_k)]^{SJ}.
	\end{equation}
\end{proposition}

\begin{proof}
	First, we observe that the image of $I$ is an abelian subalgebra of the graded Lie algebra $((\Der^\bullet L_{NS})[1] , [-,-]^{SJ})$, or equivalently, the Schouten-Jacobi bracket $[I( \alpha), I (\beta)] ^{SJ}$ vanishes for any two sections $\alpha, \beta \in \Gamma (\wedge^\bullet N_\ell S \otimes \ell)[1]$. The last assertion is a consequence of the (generalized) Leibniz property
	(\ref{eq:genleib2}) for the Schouten-Jacobi bracket, and the fact that if $\alpha$ and $\beta$ are sections of $NS$ then derivations $I (\alpha)$ and $I (\beta)$ commute.
	
	Next, we will show that the kernel of the projection $P$ is a graded Lie subalgebra of $(\Der^\bullet L_{NS})[1]$. Clearly, $\ker P$ is the $\Gamma (\wedge^\bullet J_1 L_{NS})$-submodule generated by those sections of $\der L_{NS}$ whose symbol is tangent to $S$. Since such sections are preserved by the Schouten-Jacobi bracket, it is easy to check that $\ker P$ is also preserved, using the generalized Leibniz property (\ref{eq:genleib2}) again.
	
	Finally, recall that $J \in \ker P$.
	It follows that $((\Der^\bullet L_{NS})[1], \operatorname{im}I, P, J)$ are $V$-data \cite[Theorem 1, Corollary 1]{Voronov2005}. See also
	\cite[\S 1.2, Lemma 2.2]{FZ2012} and \cite{CS2008} where the terminology \emph{V-data} has been introduced for the first time. This completes the proof.
\end{proof}

\begin{remark}\label{rem:deform} \
	\begin{enumerate}
		\item In view of (\ref{eq:d*2}), the differential $\mathfrak{m}_1$ coincides with the Jacobi algebroid differential $d_{N_\ell {}^\ast S, \ell}$.
		
		\item If $(M, \om)$ is a l.c.s.~manifold and $S$ is a coisotropic submanifold in $M$, then $\mathfrak{m}_1$ can be identified, via $\Lambda ^\#$, with a deformation of the foliation differential of the characteristic foliation of $S$ \cite{LO2012}. \color{black}
	\end{enumerate}
\end{remark}

\subsection{Coordinate formulas for the multi-brackets}
\label{sec:multi-brackets_coordinates}

In this subsection we propose some more efficient formulas for the multi-brackets in the $L_\infty$-algebra of a coisotropic submanifold. Let $(M, L, J \equiv \{-,-\})$ be a Jacobi manifold and let $S \subset M$ be a coisotropic submanifold. Moreover, as in the previous subsection, we equip $S$ with a fat tubular neighborhood $\tau : L_{NS} \INTO L$.

\begin{remark}
	\label{rem:multi-brackets_coordinates_1}
	By their very definition, the $\mathfrak m_k$'s satisfy the following properties:
	\begin{itemize}
		\item[(a)] $\mathfrak m_k$ is a graded $\R$-linear map of degree one,
		\item[(b)] $\mathfrak m_k$ is a first order differential operator with scalar-type symbol in each entry separately.
	\end{itemize}
	
	Because of (b) the $\mathfrak m_k$'s are completely determined by their action on all $\lambda\in\Gamma(\ell) = \Gamma (\wedge^0 N_\ell S \otimes \ell)$, and on all $s \in \Gamma (NS) = \Gamma (\wedge^1 N_\ell S \otimes \ell)$. Moreover (a) implies that, if $\xi_1,\dots,\xi_k \in \Gamma (\wedge^\bullet N_\ell S \otimes \ell)[1]$ have non-positive degrees, then $\mathfrak m_k (\xi_1, \dots, \xi_k) = 0$ whenever more than two arguments have degree $-1$.
\end{remark}

From now on, in this section, we identify
\begin{itemize}
	\item a section $\lambda \in \Gamma (\ell)$, with its pull-back $\pi^\ast \lambda \in \Gamma (L_{NS})$,
	\item a section $s \in \Gamma (NS)$, with the corresponding vertical vector field $\pi^\ast s \in \Gamma (\pi^\ast NS) \cong \Gamma (V (NS))$,
	\item a section $\varphi \in \Gamma (N_\ell {}^\ast S)$ of the $\ell$-adjoint bundle $N_\ell {}^\ast S = N^\ast S \otimes \ell$ with the corresponding fiber-wise linear section of $L_{NS}$.
\end{itemize}
Moreover, we denote by $\langle-,-\rangle : NS \otimes N_\ell {}^\ast S \to \ell$ the obvious ($\ell$-twisted) duality pairing.

\begin{proposition}\label{prop:CF}
	The multi-bracket $\mathfrak m_{k+1}$ is completely determined by
	\begin{align}
	\label{eq:CF1}
	\mathfrak{m}_{k+1}(s_1, \dots, s_{k-1}, \lambda,\nu )  &=(-)^{k} \D_{s_1} \cdots \D_{s_{k-1}} \{ \lambda , \nu \} |_S,\\
	\label{eq:CF2}
	\left \langle \mathfrak{m}_{k+1} (s_1,\dots,s_{k},\lambda ) , \varphi \right \rangle & = -(-)^{k} \left. \left( \D_{s_1} \cdots \D_{s_{k}} \{ \lambda , \varphi \} - \sum_{i} \D_{ s_1} \cdots \widehat{\D_{s_{i}}} \cdots \D_{s_{k}} \{ \lambda , \langle s_{i}, \varphi \rangle \} \right) \right |_S,\\
	\notag
	\left \langle \mathfrak{m}_{k+1} (s_1,\dots,s_{k+1}) , \varphi \otimes \psi \right \rangle & = -(-)^k \left( \D_{s_1} \cdots \D_{s_{k+1}} \{ \varphi, \psi \} \text{\textcolor{white}{$\sum_{i = 1}^{k+1}$}}\right.\\
	\notag &\quad  + \sum_{i<j } \D_{s_1} \cdots \widehat{\D_{s_i}}\cdots \widehat{\D_{s_j}} \cdots \D_{s_{k+1}} \left(\{ \langle s_i, \varphi \rangle, \langle s_j , \psi \rangle \} + \{\langle s_j, \varphi \rangle, \langle s_i , \psi \rangle \} \right) \\
	\label{eq:CF3}
	&\quad  \left. \left. - \sum_{i} \D_{s_1} \cdots \widehat{\D_{s_i}} \cdots \D_{s_{k+1}} \left(\{ \langle s_i , \varphi \rangle, \psi \} + \{ \varphi , \langle s_i , \psi \rangle \} \right) \right) \right |_S, 
	\end{align}
%
%
	where $\lambda, \nu \in \Gamma (\ell)$, $s_1, \dots, s_{k+1} \in \Gamma (NS)$, $\varphi, \psi \in \Gamma (N_\ell {}^\ast S)$, and a hat ``$\widehat{-}$'' denotes omission.
\end{proposition}
\begin{proof}
	Equation (\ref{eq:CF1}) immediately follows from (\ref{eq:higher_derived_brackets}), (\ref{eq:Jlm}), and the easy remark that $[\Delta, \lambda]^{SJ} = \Delta (\lambda)$ for all $\Delta \in  \Der L_{NS} = \Der^1 L_{NS}$, and $\lambda \in \Gamma (L_{NS}) = \Der^0 L_{NS}$. Equation (\ref{eq:CF2}) follows from (\ref{eq:higher_derived_brackets}), (\ref{eq:DeltaJ}), and the obvious remark that
	$ \langle s, \varphi \rangle = \D_s \varphi $,
	hence $\D_{s_1} \D_{s_2} \varphi = 0$, for all $s, s_1, s_2 \in \Gamma (NS)$, and $\varphi \in \Gamma (N_\ell {}^\ast S)$. Equation (\ref{eq:CF3}) can be proved in a similar way.
\end{proof}

Let $z^\alpha$ be local coordinates on $M$, and let $\mu$ be a local generator of $\Gamma (L)$.
Define local sections $\mu^\ast$ and $\nabla_\alpha$ of $J_1 L$ by putting
\begin{equation*}
	\mu^\ast(f\mu)= f,\quad \nabla_\alpha(f\mu)=\partial_\alpha f,
\end{equation*}
where $f\in C^\infty(M)$, and $\partial_\alpha = \partial / \partial z^\alpha$.
Then $ \Gamma (\wedge^\bullet J_1 L)$ is locally generated, as a $C^\infty (M)$-module, by
\begin{equation*}
	\nabla_{\alpha_1}\wedge\cdots\wedge\nabla_{\alpha_k},\qquad\nabla_{\alpha_1}\wedge\cdots\wedge\nabla_{\alpha_{k-1}}\wedge\mu^\ast, \quad k>0,
\end{equation*}
with $\alpha_1<\cdots<\alpha_k$. In particular, any $\Delta\in  \Gamma (\wedge^\bullet J_1 L)$ is locally expressed as
\begin{equation*}
	\Delta=X^{\alpha_1\cdots \alpha_k}\nabla_{\alpha_1}\wedge\cdots\wedge\nabla_{\alpha_k}+g^{\alpha_1\cdots \alpha_{k-1}}\nabla_{\alpha_1}\wedge\cdots\wedge\nabla_{\alpha_{k-1}}\wedge\mu^\ast,
\end{equation*}
where $X^{\alpha_1\cdots \alpha_k},g^{\alpha_1\cdots \alpha_{k-1}}\in C^\infty(M)$. Here and in what follows, we adopt the Einstein summation convention over pair of upper-lower repeated indexes.
Hence, $ (\Der^\bullet L)[1]$ is locally generated, as a $C^\infty(M)$-module, by
\begin{equation*}
	\nabla_{\alpha_1}\wedge\cdots\wedge\nabla_{\alpha_k}\otimes\mu,\qquad\nabla_{\alpha_1}\wedge\cdots\wedge\nabla_{\alpha_{k-1}}\wedge\operatorname{id}, \quad k > 0,
\end{equation*}
with $\alpha_1<\cdots<\alpha_k$, and any $\square\in \mathcal  (\Der^\bullet L)[1]$ is locally expressed as
\begin{equation*}
	\square=X^{\alpha_1\cdots \alpha_k}\nabla_{\alpha_1}\wedge\cdots\wedge\nabla_{\alpha_k}\otimes\mu+g^{\alpha_1\cdots \alpha_{k-1}}\nabla_{\alpha_1}\wedge\cdots\wedge\nabla_{\alpha_{k-1}}\wedge\operatorname{id}.
\end{equation*}

\begin{remark}
	\label{oss:components_of_J}
	Let $J\in \Der^2 L$. Locally,
	\begin{equation}
		\label{eq:components_of_J}
		J=J^{\alpha\beta}\nabla_{\alpha}\wedge\nabla_{\beta}\otimes\mu+J^\alpha\nabla_\alpha\wedge\operatorname{id},
	\end{equation}
	for some local functions $J^{\alpha \beta}, J^\alpha$.
\end{remark}

Now, identify $L_{NS}$ with its image in $L$ under $\tau$ and assume that:
\begin{itemize}
	\item coordinates $z^\alpha$ are fibered, i.e.~$(z^\alpha) = (x^i,y^a)$, with $x^i$ coordinates on $S$, and $y^a$ linear coordinates along the fibers of $\pi : NS \to S$,
	\item the local generator $\mu$ is fiber-wise constant so that, locally, $\Gamma(\ell) \subset \Gamma (L_{NS})$ consists exactly of sections $\lambda$ such that $\nabla_a\lambda=0$.
\end{itemize}
In particular, the local expression (\ref{eq:components_of_J}) for $J$ expands as
\begin{align}
	\label{eq:components_of_J_bis}
	J &= \left( J^{a b}\nabla_a\wedge\nabla_b +2J^{ai}\nabla_a\wedge\nabla_i +J^{ij}\nabla_i\wedge\nabla_j \right) \otimes\mu\\
	\notag &\quad + \left( J^a\nabla_a +J^i\nabla_i \right)\wedge\operatorname{id}.
\end{align}
We have the following

\begin{corollary}
	\label{prop:multi-brackets_coordinates}
	Locally, the multi-bracket $\mathfrak m_{k+1}$ is uniquely determined by
	\begin{align*}
		&\mathfrak{m}_{k+1}\left(\partial_{a_1},\dots,\partial_{a_{k+1}} \right) =-(-)^k \partial_{a_1} \cdots \partial_{a_{k+1}}J^{ab} \big|_S \delta_a\wedge\delta_b\otimes\mu,\\
		&\mathfrak{m}_{k+1}\left( \partial_{a_1},\dots,\partial_{a_{k}}, f \mu \right) =(-)^k \partial_{a_1} \cdots \partial_{a_{k}}\left.\left(2J^{ai}\partial_i f+J^a f\right)\right|_S \partial_a,\\
		&\mathfrak{m}_{k+1}\left(\partial_{a_1},\dots,\partial_{a_{k-1}},f \mu,g \mu \right)=(-)^k \partial_{a_1}\cdots \partial_{a_{k-1}}\left.\left[2J^{ij} \partial_i f \partial_j g - J^i(f\partial_i g-g \partial_i f)\right]\right|_
			S \mu,
	\end{align*}
	where $f, g \in C^\infty (S)$, and $\delta_a := \partial_a \otimes \mu^\ast$.
\end{corollary}

\subsection{Independence of the tubular embedding} \label{sec:gauge_invariance}
Now we show that, as already in the symplectic \cite[Appendix]{OP2005}, the Poisson \cite{CS2008}, and the l.c.s.~\cite[Theorem 9.5]{LO2012} cases, the $L_\infty$-algebra in Proposition~\ref{prop:linfty} does not really depend on the choice of a fat tubular neighborhood, in the sense clarified by Proposition \ref{prop:gauge_invariance} below. As a consequence, its $L_\infty$-isomorphism class is an \emph{invariant} of the coisotropic submanifold.

\begin{proposition}
	\label{prop:gauge_invariance}
	Let $S$ be a coisotropic submanifold of the Jacobi manifold $(M,L,J \equiv \{-,-\})$.
	Then the $L_\infty$-algebra structures on $\Gamma (\wedge^\bullet N_\ell S \otimes \ell)[1]$ associated to different choices of the fat tubular neighborhood $L_{NS} \INTO L$ of $\ell$ in $L$ are $L_\infty$-isomorphic.
\end{proposition}

The proof is an adaptation of the one given by Cattaneo and Sch\"atz in the Poisson setting (see Subsections 4.1 and 4.2 of \cite{CS2008}, see also Remark~\ref{rem:CS} below) and it is based on Theorem 3.2 of \cite{CS2008} and the fact that \emph{any two fat tubular neighborhoods are isotopic} (in the sense of Lemma \ref{lem:isotopy} below). Before proving Proposition \ref{prop:gauge_invariance}, let us recall Cattaneo--Sch\"atz Theorem. We will present a ``minimal version'' of it, adapted to our purposes. The main ingredients are the following.

We work in a category of real topological vector spaces. Let $(\mathfrak h,\mathfrak a, P, \Delta_0)$ and $(\mathfrak h, \mathfrak a, P, \Delta_1)$ be $V$-data \cite{FZ2012}. We identify $\mathfrak a$ with the target space of $P$. Note that $(\mathfrak h,\mathfrak a, P, \Delta_0)$ and $(\mathfrak h, \mathfrak a, P, \Delta_1)$ differ by the last entry only. Voronov construction associates $L_\infty$-algebras to $(\mathfrak h,\mathfrak a, P, \Delta_0)$ and $(\mathfrak h, \mathfrak a, P, \Delta_1)$. Denote them $\mathfrak a_0$ and $\mathfrak a_1$ respectively. Cattaneo and Sch\"atz main idea is proving that if
\begin{itemize}
	\item $\Delta_0$ and $\Delta_1$ are gauge equivalent elements of the graded Lie algebra $\mathfrak h$, and
	\item they are intertwined by a gauge transformation preserving $\ker P$,
\end{itemize}
then $\mathfrak a_0$ and $\mathfrak a_1$ are $L_\infty$-isomorphic. Specifically, $\Delta_0$ and $\Delta_1$ are \emph{gauge equivalent} if they are interpolated by a smooth family $\{ \Delta_t \}_{t \in [0,1]}$ of elements $\Delta_t \in \mathfrak h$, and there exists a smooth family $\{ \xi_t \}_{t \in [0,1]}$ of degree zero elements $\xi_t \in \mathfrak h$ such that the following evolutionary differential equation is satisfied:
\begin{equation}\label{eq:gaugeMC00}
	\frac{d}{dt} \Delta_t = [\xi_t , \Delta_t].
\end{equation}
One usually assumes that the family $\{ \xi_t \}_{t \in [0,1]}$ integrates to a family\linebreak $\{ \phi_t \}_{t \in [0,1]}$ of automorphisms $\phi_t : \mathfrak h \to \mathfrak h$ of the Lie algebra $\mathfrak h$, i.e.~$\{ \phi_t \}_{t \in [0,1]}$ is a solution of the Cauchy problem
\begin{equation}
	\label{eq:Cauchy}
	\left\{
	\begin{aligned}
		& \frac{d}{dt}\phi_t (-)=[\phi_ t (-), \xi_t] \\
		& \phi_0= \id
	\end{aligned}
	\right. .
\end{equation}

Finally we say that $\Delta_0$ and $\Delta_1$ are intertwined by a gauge transformation preserving $\ker P$ if the family $\{ \xi_t \}_{t \in [0,1]}$ above satisfies the following conditions:
\begin{enumerate}
	\item the only solution $\{ a_t \}_{t \in [0,1]}$, where $a_t \in \mathfrak a$, of the Cauchy problem
	\begin{equation}
		\label{eq:equivalence_2a}
		\left\{
		\begin{aligned}
			& \frac{d}{dt}a_t =P [a_t , \xi_t] \\
			& a_0 =0
		\end{aligned}
		\right.
	\end{equation}
	is the trivial one: $a_t = 0$ for all $t \in [0,1]$,
	\item $[ \xi_t, \ker P] \subset \ker P$ for all $t \in [0,1]$.
\end{enumerate}

\begin{theorem}[cf.~{\cite[Theorem 3.2]{CS2008}}]\label{theor:CS}
	Let $(\mathfrak h,\mathfrak a, P, \Delta_0)$ and $(\mathfrak h, \mathfrak a, P, \Delta_1)$ be $V$-data, and let $\mathfrak a_0$ and $\mathfrak a_1$ be the associated $L_\infty$-algebras. If $\Delta_0$ and $\Delta_1$ are gauge equivalent and they are intertwined by a gauge transformation preserving $\ker P$, then $\mathfrak a_0$ and $\mathfrak a_1$ are $L_\infty$-isomorphic.
\end{theorem}

The last ingredient needed to prove Proposition \ref{prop:gauge_invariance} is provided by the following

\begin{lemma}\label{lem:isotopy}
	Any two fat tubular neighborhoods $ \tau{}_0$ and $\tau{}_1$ of $S$ are \emph{isotopic}, i.e.~there is a smooth one parameter family of fat tubular neighborhoods $\mathcal T{}_t $ of $\ell$ in $L$, and an automorphism $\psi : L_{NS} \to L_{NS}$ of $L_{NS}$ covering an automorphism $ \underline \psi : NS \to NS$ of $NS$ over the identity, such that $\mathcal T_0 = \tau_0$, and $\mathcal T_1 = \tau_1 \circ \psi$.
\end{lemma}

\begin{proof}
	In view of the tubular neighborhood Theorem \cite[Theorem 5.3]{Hirsch}, there is a smooth one parameter family of tubular neighborhoods $\underline{\mathcal T}{}_t : NS \INTO M$ of $S$ in $M$, and an automorphism $\underline \psi : NS \to NS$ over the identity such that $\underline{\mathcal T}{}_0 = \underline{\tau}{}_0$, and $\underline{\mathcal T}{}_1 = \underline \tau{}_1 \circ \underline \psi $.
	Denote by $\underline{\mathcal T}:NS\times[0,1] \to M$ the map defined by $\underline{\mathcal T} (\nu,t)=\underline{\mathcal T}{}_t(\nu)$ and consider the line bundle
	\[
	p : L_{NS}^\ast \otimes_{NS} \underline{\mathcal T}^\ast L \longrightarrow NS \times [0,1].
	\]
	Note that
	\begin{enumerate}
		\item fibers of $NS \times [0,1]$ over $S \times [0,1]$ are contractible,
		\item $L_{NS}^\ast \otimes_{NS} \underline{\mathcal T}^\ast L $ reduces to $\operatorname{End} \ell \times [0,1] = \R_{S \times [0,1]}$ over $S \times [0,1]$.
	\end{enumerate}
	It follows from 1) and 2) that $L_{NS}^\ast \otimes_{NS} \underline{\mathcal T}^\ast L$ is isomorphic to the pull-back over $NS \times [0,1]$ of the trivial line bundle $\R_{S \times [0,1]}$ over $S \times [0,1]$. In particular, $p$ is a trivial bundle. Moreover, $p$ admits a nowhere zero section $\upsilon$ defined on $(S\times[0,1])\cup(NS\times\{0,1\})$ and given by $\operatorname{id}_{\ell} $ on $S\times[0,1]$, by ${\mathcal T}_0$ on $NS\times\{0\}$ and by ${\mathcal T}_1$ on $NS \times \{ 1 \}$.
	By triviality, $\upsilon$ can be extended to a nowhere zero section $\Upsilon$ on the whole $NS \times [0,1]$. The section $\Upsilon$ is the same as a one parameter family of vector bundle isomorphisms $\Upsilon_t : L_{NS} \to \underline{\mathcal T}{}_t^\ast L$ over the identity of $NS$. Denote by ${\mathcal T}_t : L_{NS} \to L$ the composition
	\[
	L_{NS} \overset{\Upsilon_t}{\longrightarrow} \underline{\mathcal T}{}_t^\ast L \lhook\joinrel\longrightarrow L,
	\]
	where the second arrow is the natural inclusion. By construction, the ${\mathcal T}_t$'s are line bundle embeddings covering the $\underline{\mathcal T}{}_t$'s. Finally, there exists a unique automorphism $\psi : L_{NS} \to L_{NS}$ over $\underline \psi$ such that ${\mathcal T}_1 = \tau_1 \circ \psi$. We conclude that the ${\mathcal T}_t$'s and $\psi$ possess all the required properties.
\end{proof}

\begin{proof}[Proof of Proposition {\ref{prop:gauge_invariance}}]
	Let $\tau_0, \tau_1 \!:\! L_{NS} \!\INTO\! L$ be fat tubular neigh\-bor\-hoods over tubular neighborhoods $\underline \tau{}_0, \underline \tau{}_1 : NS \INTO M$. Denote by $J_0$ and $J_1$ the Jacobi brackets induced on $\Gamma (L_{NS})$ by $\tau_0$ and $\tau_1$ respectively, i.e.~$J_0 = (\tau_0^{-1})_\ast J$, and $J_1 = (\tau_1^{-1})_\ast J$ (see Remark \ref{rem:inf_aut} about pushing forward a multi-differential operator along a line bundle isomorphism). In view of Lemma~\ref{lem:isotopy} it is enough to consider the following two cases:
	
	{\textbf Case I:} \emph{$\tau_1 = \tau_0 \circ \psi$ for some automorphism $\psi : L_{NS} \to L_{NS}$ covering an automorphism $ \underline \psi : NS \to NS$ of $NS$ over the identity}. Obviously, $\psi$ identifies the $V$-data $((\Der^\bullet L_{NS})[1], \im I, P, J_0)$ and $((\Der^\bullet L_{NS})[1], \im I, P, J_1)$. As an immediate consequence, the $L_\infty$-algebra structures on $\Gamma (\wedge^\bullet N_\ell S \otimes \ell)[1]$ determined by $\tau_0$ and $\tau_1$ are (strictly) $L_\infty$-isomorphic.
	
	{\textbf Case II:} \emph{$\tau_0$ and $\tau_1$ are interpolated by a smooth one parameter family of fat tubular neighborhoods $\tau_t$.}
	Consider $\phi_t := \tau_t^{-1} \circ \tau_0 $. It is a local automorphism of $L_{NS}$ covering a local diffeomorphism $\underline \varphi{}_t = \underline \tau{}_t^{-1} \circ \underline \tau{}_0$, well defined in a suitable neighborhood of $S$ in $NS$, fixing $S$ point-wise and such that $\varphi_0 = \id$. Let $\xi_t$ be infinitesimal generators of the family $\{ \varphi_t \}$. They are derivations of $L_{NS}$ well defined around $S$. Our strategy is using $\xi_t$ and $\varphi_t$ to prove that $J_0$ and $J_1$ are gauge equivalent Maurer-Cartan elements of $(\Der^\bullet L_{NS})[1]$ intertwined by a gauge transformation preserving $\ker P$, and then applying Theorem \ref{theor:CS}. However, the $\varphi_t$'s are well-defined only around $S$ in $NS$. In order to remedy this minor drawback, we slightly change the graded space $\Der^\bullet L_{NS}$ underlying our $V$-data, passing to the graded space $\Der_{\text{for}}^\bullet L_{NS}$ of \emph{alternating, first order, multi-differential operators on $L_{NS}$ in a formal neighborhood of $S$ in $NS$}. The space $\Der_{\text{for}}^\bullet L_{NS}$ is defined as the inverse limit
	\[
	\lim_{\longleftarrow} \Der^\bullet L_{NS}/ I(S)^n \Der^\bullet L_{NS},
	\]
	where $I(S) \subset C^\infty (NS)$ is the ideal of functions vanishing on $S$. In a sense, $\Der_{\text{for}}^\bullet L_{NS}$ consists of ``Taylor series normal to $S$'' of multi-differential operators. Our $V$-data $((\Der^\bullet L_{NS})[1], \im I, P, J)$ induce in an obvious way new $V$-data $((\Der_{\text{for}}^\bullet L_{NS})[1], \im I_{\text{for}}, P_{\text{for}}, J_{\text{for}})$. In particular, $J_{\text{for}}$ is the class of $J$ in $(\Der_{\text{for}}^\bullet L_{NS})[1]$, and $I_{\text{for}} : \Gamma (\wedge^\bullet N_\ell S \otimes \ell)[1] \INTO (\Der_{\text{for}}^\bullet L_{NS})[1]$ is the natural embedding. Moreover, in view of Corollary \ref{prop:multi-brackets_coordinates}, the $L_\infty$-algebra determined by $((\Der^\bullet L_{NS})[1], \im I, P, J)$ does only depend on $J_{\text{for}}$. Therefore, the $V$-data $((\Der_{\text{for}}^\bullet L_{NS})[1], \im I_{\text{for}}, P_{\text{for}}, J_{\text{for}})$ and $((\Der^\bullet L_{NS})[1], \im I, P, J)$ determine the same $L_\infty$-algebra.
	
	Now, being well defined around $S$, the $\varphi_t$'s determine well-defined automorphisms $\phi_t := (\varphi_t)_\ast : (\Der_{\text{for}}^\bullet L_{NS})[1]\longrightarrow(\Der_{\text{for}}^\bullet L_{NS})[1]$ such that $\phi_0 = \id$. Similarly the $\xi_t$'s descend to zero degree elements of $(\Der_{\text{for}}^\bullet L_{NS})[1]$ which we denote by $\xi_t$ again. Clearly, the family $\{ \phi_t (J_0)_{\text{for}} \}$ interpolates between $(J_0)_{\text{for}}$ and $(J_1)_{\text{for}}$ and, in view of Equation (\ref{Lie2}), the $\phi_t$'s satisfy the Cauchy problem~(\ref{eq:Cauchy}). Finally,
	\begin{enumerate}
		\item from uniqueness of the one parameter family of automorphisms $\varphi_t$ generated by the one parameter family of derivation $\xi_t$, it follows that the Cauchy problem (\ref{eq:equivalence_2a}) possesses a unique solution,
		\item $\varphi_t |_{\ell} = \id$ so that the $\xi_t$'s vanish on $S$, hence $[\xi_t, \ker P ] \subset \ker P$ for all~$t$.
	\end{enumerate}
	The above considerations show that $(J_0)_{\text{for}}$ and $(J_1)_{\text{for}}$ are gauge equivalent and they are intertwined by a gauge transformation preserving $\ker P$. Hence, from Theorem \ref{theor:CS}, the $L_\infty$-algebra structures on $\Gamma (\wedge^\bullet N_\ell S \otimes \ell)[1]$ associated to the two choices $\tau_0$ and $\tau_1$ of the fat tubular neighborhood $L_{NS} \INTO L$ are actually $L_\infty$-isomorphic.
\end{proof}

\begin{remark}
	\label{rem:CS}
	In the contact case, as already in the l.c.s.~one, there exists a tubular neighborhood theorem for coisotropic submanifolds. As a consequence, the proof of Proposition \ref{prop:gauge_invariance} simplifies. In particular, it does not require using any \emph{formal neighborhood technique}.
\end{remark}

\section[Deformations of coisotropic submanifolds in Jacobi manifolds]{Deformations of coisotropic submanifolds in Jacobi manifolds}\label{sec:deform}
In this section, we introduce the notion of formal coisotropic deformation of a coisotropic submanifold (Definition \ref{def:coisoformal}). We prove that formal coisotropic deformations are in one-to-one correspondence with (degree $0$) Maurer-Cartan elements
of the associated $L_\infty$-algebra (Proposition \ref{prop:mcformal}). We also give a necessary and sufficient condition for the convergence of the Maurer-Cartan series $MC(s)$ for any smooth section $s$ (Proposition \ref{prop:fana}), extending a previous sufficient condition given by Sch\"atz and Zambon in the Poisson case \cite {SZ2012}. Analysing the notion of Hamiltonian equivalence of coisotropic deformations (Proposition \ref{prop:Hameq}) leads to a definition of Hamiltonian equivalence of formal deformations (Definition \ref{def:formHeq}). We show that Hamiltonian equivalence of formal coisotropic deformations coincides with gauge equivalence of the corresponding Maurer-Cartan elements (Proposition \ref{prop:MC_gauge}) and derive consequences of this fact (Theorem \ref{prop:fana1}, Corollary \ref{cor:infequi}). Finally we compare our results with related results obtained by other authors (Remarks~\ref{rem:moduli_1} and~\ref{rem:moduli_2}).

\subsection{Smooth coisotropic deformations}\label{sec:cois_def}
Let $(M, L, J \equiv \{-,-\})$ be a Jacobi manifold and let $S \subset M$ be a closed coisotropic submanifold. We equip $S$ with a fat tubular neighborhood $\tau : L_{NS} \INTO L$ and use it to identify $L_{NS}$ with its image. Accordingly, and similarly as above, from now on in this section, we abuse the notation and denote by $(L, J \equiv \{-,-\})$ (instead of $(L_{NS}, \tau_\ast^{-1} J)$) the Jacobi structure on $NS$ (unless otherwise specified). A $C^1$-small deformation of $S$ in $NS$ can be identified with a section $S \to NS$ of $NS$. We say that a section $s : S \to NS$ is \emph{coisotropic} if its image $s (S)$ is a coisotropic submanifold in $(NS, L, J)$.

\begin{definition}\label{def:smcoiso}
	A smooth one parameter family of smooth sections of $NS \to S$ starting from the zero section is a \emph{smooth coisotropic deformation of} $S$ if each section in the family is coisotropic. A section $s$ of $NS \to S$ is an \emph{infinitesimal coisotropic deformation of} $S$ if $\eps s$ is a coisotropic section up to infinitesimals $O(\eps^2)$, where $\eps$ is a formal parameter.
\end{definition}

\begin{remark}
	Let $\{ s_t \}$ be a smooth coisotropic deformation of $S$. Then
	\[
	\left. \frac{d}{dt}\right|_{t=0} s_t
	\]
	is an infinitesimal coisotropic deformation.
\end{remark}

Recall that a section $s:S \to NS$ is mapped, via $I : \Gamma (\wedge^\bullet N_\ell S \otimes \ell)[1] \to  (\Der^\bullet L)[1]$, to a derivation $I(s) := \mathbb D_{\pi^\ast s}$ of $L$, where $\pi : NS \to S$ is the projection. Let $\{\Phi_t\}$ be the one parameter group of automorphisms of $L$ generated by $I(s)$ and denote $\exp I (s) := \Phi_{1}$. Clearly $\exp I(s) (\nu , \lambda) = (\nu + s(x) , \lambda)$, for all $(\nu , \lambda) \in L = NS \times_S \ell$, $x = \pi (\nu)$. Further, let $\mathrm{pr} : J^1 L \to NS$ be the projection, denote by $j^1 \exp I(s) : J^1 L \to J^1 L$ the first jet prolongation of $\exp I(s)$, and consider the following commutative diagram
\[
\xymatrix@C=50pt{
	N_\ell {}^\ast S \ar[d] \ar@/^/[r]^{\gamma} & J^1 L \ar[d] \ar@/^/[r] ^{j^1 \exp I(s)} & J^1 L \ar[d] \\
	S\ar@/^/@<0.7 ex>[r]^{\mathbf 0} & NS \ar[l]^{\pi} \ar@/^/@<0.7 ex>[r]^{\underline{\exp I (s)}} & NS \ar[l]^{\underline{\exp I(-s)}}
}
\]
where $\mathbf 0$ is the zero section. Note that $s = \underline{\smash{\exp I(s)}} \circ \mathbf 0$.

\begin{proposition}\label{prop:coiss} Let $s : S \to NS$ be a section of $\pi$. The following three conditions are equivalent
	\begin{enumerate}
		\item $s$ is coisotropic,
		\item $P (\exp I (-s)_* J) = 0$ (cf.~\cite{SZ2012}),
		\item the vector bundle $\mathrm{pr} \circ j^1 \exp I(s) \circ \gamma : N_\ell {}^\ast S \to s(S)$ is a Jacobi subalgebroid of $J^1 L$.
	\end{enumerate}
\end{proposition}

\begin{proof} 
	1) $\Longleftrightarrow$ 2). Let $P^s : \Der L \to\Gamma(NS)$ be the composition
	\begin{equation*}
		\Der L \overset{\sigma}{\longrightarrow} \mathfrak{X}(M) \longrightarrow \Gamma (TM |_{s(S)})\longrightarrow \Gamma (NS),
	\end{equation*}
	where the second arrow is the restriction, and the last arrow is the canonical projection (cf.~(\ref{eq:seq_P})). The surjection $P^s$ extends to a surjection of graded modules $(\Der^\bullet L)[1] \to \Gamma (\wedge^\bullet N_\ell S \otimes \ell)[1]$ which we denote again by $P^s$ (and is defined analogously as $P : (\Der^\bullet L)[1] \to \Gamma (\wedge^\bullet N_\ell S \otimes \ell)[1]$). By Proposition \ref{prop:PJ=0}, $s$ is coisotropic if and only if $ P^s (J) = 0$. Since
	\[
	\der \ell = \exp I (-s)_* \der L|_{s(S)} \quad \text{and} \quad \underline{\smash{\exp I(-s)}}_* NS = NS,
	\]
	we obtain
	\begin{equation}\label{eq:Ps}
		P^s = P \circ \exp I(-s)_*.
	\end{equation}
	In particular, $P^s (J) = P (\exp I(-s)_*J) = 0$ if and only if $s$ is coisotropic.
	
	1) $\Longleftrightarrow$ 3). Note that $\mathrm{pr} \circ j^1 \exp I(s) \circ \gamma : N_\ell {}^\ast S \to s(S)$ is the $\ell$-adjoint bundle of the normal bundle of $s(S)$ in $NS$. Now the claim follows immediately from Proposition \ref{prop:conormal}.
\end{proof}

\begin{remark}\label{rem:cois_Lambda}
	Let $s$ be a section of $NS$. In view of Remark \ref{rem:PJ}, $P^s (J) = P^s (\Lambda_J)$, where, in the right hand side, $P^s$ denotes the extension $\Gamma (\wedge^\bullet (T(NS) \otimes L^\ast) \otimes L) \to \Gamma (\wedge^\bullet N_\ell S \otimes \ell)$ of the composition $\mathfrak{X}(NS) \to \Gamma(T(NS)|_{s(S)}) \to\Gamma(NS)$ defined analogously as $P : (\Der^\bullet L)[1] \to \Gamma (\wedge^\bullet N_\ell S \otimes \ell)[1]$. Moreover, it is clear that
	\[
	\Lambda_{\exp I(-s)_\ast J} = \exp I(-s)_\ast \Lambda_J,
	\]
	where $\Lambda_{\exp I(-s)_\ast J} $ is the bi-symbol of $\exp I(-s)_\ast J$, and, in the right hand side,
	$$\exp I(-s)_\ast : \Gamma (\wedge^\bullet (T(NS) \otimes L^\ast)\otimes L) \to \Gamma (\wedge^\bullet (T(NS) \otimes L^\ast)\otimes L)$$
	denotes the isomorphism induced by the line bundle automorphism $\exp I(-s)$. It immediately follows that $s$ is coisotropic if and only if $P(\exp I (-s)_\ast \Lambda_J)=0$.
\end{remark}

\subsection{Formal coisotropic deformations}
Let $\eps$ be a formal parameter.

\begin{definition}
	A formal series $s (\eps) = \sum _{i =0} ^ \infty \eps ^i s_i \in \Gamma (NS) [[ \eps]]$, $s_i \in \Gamma (NS)$, such that $s_0 = 0$, is called a \emph{formal deformation of $S$}.
\end{definition}

The formal series $I (s (\eps)) : = \sum_{i =0}^\infty \eps ^i I (s_i) \in (\Der L) [[ \eps ]]$ is a formal derivation of $L$. It is easy to see that the space $(\Der L) [[\eps]]$ of formal derivations of $L$ is a Lie algebra, which has a linear representation in the space $ (\Der^\bullet L)[[\eps]]$ of formal first order multi-differential operators on $L$ via the following \emph{Lie derivative}:
\begin{equation}
	\Ll_{\xi(\eps)} \Delta (\eps) \equiv [\xi (\eps), \Delta (\eps)]^{SJ} := \sum _{ k =0} ^\infty \eps ^{k} \sum_{i+j = k} [\xi_i , \Delta_j]^{SJ},
\end{equation}
for $\xi (\eps)= \sum_{ i= 0} ^ \infty\eps ^i \xi_i$, $\xi_i \in  \Der L$, and $\Delta (\eps) = \sum _{ i =0}^\infty \eps ^ i \Delta_i$, $\Delta_i\in  \Der^\bullet L$.

We define the exponential of the Lie derivative $\Ll_{\xi (\eps)}$ as the following formal power series
\begin{equation}
	\exp \Ll_{\xi(\eps)} = \sum _{ n = 0}^\infty \frac{1}{n!} \Ll_{\xi (\eps)}^n.\label{eq:exp}
\end{equation}

Proposition \ref{prop:coiss} motivates the following

\begin{definition}\label{def:coisoformal} A formal deformation $s (\eps)$ of $S$ is said \emph{coisotropic}, if
	$P (\exp \Ll_{I(s (\eps))} J) = 0$.
\end{definition}

\begin{remark}
	Let $\xi (\eps) \in ( \der L) [[ \eps]]$. Define a Lie derivative
	\[
	\Ll_{\xi (\eps)} : \Gamma (\wedge^\bullet (T(NS) \otimes L^\ast) \otimes L) [[\eps]] \to \Gamma (\wedge^\bullet (T(NS) \otimes L^\ast) \otimes L) [[\eps]],
	\]
	in the obvious way. It is easy to see that
	\begin{equation}\label{eq:PJvsPLambda}
		P (\exp \Ll_{I(s (\eps))} J) = P (\exp \Ll_{I(s (\eps))} \Lambda_J),
	\end{equation}
	for all formal deformations $s (\eps)$ of $S$ (cf.~Remarks \ref{rem:PJ} and \ref{rem:cois_Lambda}). In particular, $s (\eps)$ is coisotropic if and only if $ P (\exp \Ll_{I(s (\eps))} \Lambda_J) = 0$.
\end{remark}

\begin{remark}[Formal deformation problem]
	The \emph{formal deformation problem for a coisotropic submanifold $S$} consists in finding formal coisotropic deformations of $S$. Let $s( \eps) = \sum_{i=0}^\infty \eps^i s_i$ be a formal coisotropic deformation of $S$. Then $s_1$ is an infinitesimal coisotropic deformation. On the other hand, in general, not all infinitesimal coisotropic deformations can be ``prolonged'' to a formal coisotropic deformation. If this is the case, one says that \emph{the formal deformation problem is unobstructed}. Otherwise, \emph{the formal deformation problem is obstructed}. The formal deformation problem of $S$ is \emph{governed by the $L_\infty$-algebra} $(\Gamma (\wedge^\bullet N_\ell S \otimes \ell)[1], \{ \mathfrak{m}_k\})$ in the sense clarified by the following proposition.
\end{remark}
\begin{proposition}\label{prop:mcformal} A formal deformation $s (\eps)$ of $S$ is coisotropic if and only if $-s (\eps)$
	is a solution of the (formal) Maurer-Cartan equation
	\begin{equation}\label{eq:MC}
		MC (-s (\eps)):=\sum _{ k =1} ^\infty \frac{1}{k!}\mathfrak{m}_k (-s (\eps), \dots , -s(\eps))= 0.
	\end{equation}
\end{proposition}
\begin{proof} The expression $MC (-s (\eps))$ should be interpreted as an element of $\Gamma (\wedge^\bullet N_\ell S \otimes \ell)[[ \eps]]$. The proposition is then a consequence of (\ref{eq:exp}), $P(J) = 0$, and the following identities
	\begin{equation} \label{eq:mcl}
		P (\Ll_{I(\xi)}^k J) = \mathfrak{m}_k (-\xi, \dots, -\xi) , \quad k \ge 1,
	\end{equation}
	for $\xi \in \Gamma (NS)$, which immediately follow from the definition of $\mathfrak{m}_k$.
\end{proof}

Let $s$ be a section of $NS$. The \emph{Maurer-Cartan series of $s$} is the series
\[
MC(-s) := \sum_{k=1}^\infty \frac{1}{k!} \mathfrak{m}_k (-s, \dots, -s).
\]
In general, $MC(-s)$ does not converge, not even for a coisotropic $s$. However, we have the obvious

\begin{corollary}
	Let $s$ be a section of $NS$ such that the \emph{Maurer-Cartan series} $MC(-s)$ converges. Then $s$ is a coisotropic deformation of $S$ if and only if $MC (-s) = 0$.
\end{corollary}

\begin{corollary}\label{cor:inf1} A section $s$ of $NS$ is an
	infinitesimal coisotropic deformation of $S$ iff
	\begin{equation}
		\mathfrak{m}_1 (s) = 0.\label{eq:infdef}
	\end{equation}
\end{corollary}

By Remark \ref{rem:deform}.(1), $\mathfrak{m}_1$ coincides with the Jacobi algebroid de Rham differential $d_{N_\ell {}^\ast S, \ell}$. Hence, a similar argument as in the proof of Theorem 11.2 in \cite{OP2005} yields

\begin{corollary}\label{prop:rigid} Assume that the second cohomology group $H^2 (N_\ell {}^\ast S, \ell)$ of the Jacobi subalgebroid $N_\ell {}^\ast S \subset J^1 L$ with values in $\ell$ is zero. Then every infinitesimal coisotropic deformation can be prolonged to a formal coisotropic deformation, i.e.~for any given class $\alpha \in H^1(N_\ell {}^\ast S, \ell)$ Equation (\ref{eq:MC}) has a solution $s ( \eps) = \sum _{i=1} ^ \infty \eps ^ i s_i$ such that $\mathfrak{m}_1 (s_1) = 0$ and $[s_1] = \alpha$. In other words, the formal deformation problem is unobstructed.
\end{corollary}

There is also a simple criterion for non-prolongability of an infinitesimal coisotropic deformation to a formal coisotropic deformation based on the \emph{Kuranishi map}:
\[
Kr : H^1 (N_\ell {}^\ast S, \ell) \longrightarrow H^2 (N_\ell {}^\ast S, \ell), \quad [s] \longmapsto [\mathfrak m_2 (s, s)].
\]
Since $\mathfrak{m}_1$ is a derivation of the binary bracket $\mathfrak{m}_2$, the Kuranishi map is well-defined. Moreover, similarly as in \cite{OP2005} (Theorem 11.4) we have the following

\begin{proposition}
	\label{prop:Kuranishi}
	Let $\alpha = [s] \in H^1 (N_\ell {}^\ast S, \ell)$, where $s \in \Gamma (NS)$ is an in\-fin\-i\-tes\-i\-mal coisotropic deformation, i.e.~$d_{N_\ell {}^\ast S, \ell} s = \mathfrak{m}_1 s = 0$. If $Kr (\alpha) \neq 0$, then $s$ cannot be prolonged to a formal coisotropic deformation. In particular, the formal deformation problem is obstructed.
\end{proposition}

\subsection{Formal deformations and smooth deformations}\label{sec:radially_entire} In this subsection we establish a connection between formal coisotropic deformations and smooth coisotropic deformations. We do this introducing the notion of fiber-wise entire bi-symbol, which is a slight generalization of the notion of fiber-wise entire Poisson structure introduced by Sch\"atz and Zambon in \cite{SZ2012}, and is motivated by the Taylor expansion of the bi-linear form
$P (\exp I (-s)_* \Lambda_J)$ (Proposition \ref{prop:fana}).

Let $E \to S$ be a vector bundle. Recall that a smooth function on $E$ is called \emph{fiber-wise entire} if its restriction to each fiber of $E$ is \emph{entire}, i.e.~it is real analytic on the whole fiber. Now, let $\ell \to S$ be a line bundle, and $L := E \times_S \ell$. A section of $L$ is called \emph{fiber-wise entire} if it is a linear combination of fiber-wise constant sections, with coefficients being fiber-wise entire functions.
Let $\Theta \in \Gamma (\wedge^k (TE \otimes L^\ast)\otimes L)$. We regard $\Theta$ as a multi-linear map
\[
\Theta : \wedge^{k} (T^\ast E \otimes L) \longrightarrow L.
\]
The multi-linear map $\Theta$ is called {\it fiber-wise entire} if
\[
\Theta (df_1 \otimes \lambda_1 , \dots, df_{k} \otimes \lambda_{k})
\]
is fiber-wise entire, whenever $f_1, \dots, f_{k}$ are fiber-wise linear and $\lambda_1, \dots, \lambda_{k}$ are fiber-wise constant. Equivalently $\Theta$ is fiber-wise entire if its components in some (and therefore any) system of vector bundle coordinates are fiber-wise entire functions (cf.~\cite[Lemmas 1.4, 1.7]{SZ2012}).

Now, let $S$ and $(NS, L, J \equiv \{-,-\})$ be as in Subsection \ref{sec:cois_def}. The following proposition generalizes the main result of \cite{SZ2012} establishing a necessary and sufficient condition for the convergence of the Maurer-Cartan series $MC(-s)$
of a generic section $s \in \Gamma(NS)$.

\begin{proposition}\label{prop:fana} The bi-symbol $\Lambda_J$ of the Jacobi bi-differential operator $J$ is fiber-wise entire iff, for all sections $s \in \Gamma (NS)$, the Maurer-Cartan series $MC (-s)$ converges to $P ( \exp I(s)_\ast J) = P (\exp I(s)_\ast \Lambda_J )$ in the sense of point-wise convergence.
\end{proposition}
\begin{proof} Let $(z^\alpha) = (x^i, y^a)$ be vector bundle coordinates on $NS$, with $x^i$ coordinates on $S$, and $y^a$ linear coordinates along the fibers of $NS$. Moreover, let $\mu$ be a fiber-wise constant local generator of $\Gamma (L)$. The Jacobi bi-differential operator $J$ is locally given by Equation (\ref{eq:components_of_J}), or, equivalently, Equation~(\ref{eq:components_of_J_bis}):
	\[
	J  = \left( J^{a b}\nabla_a\wedge\nabla_b +2J^{ai}\nabla_a\wedge\nabla_i +J^{ij}\nabla_i\wedge\nabla_j \right) \otimes\mu+ \left( J^a\nabla_a +J^i\nabla_i \right) \wedge\operatorname{id}.
	\]
	Accordingly, the bi-symbol $\Lambda_J$ is locally given by
	\[
	\begin{aligned}
	\Lambda_J  = \left( J^{a b}\delta_a\wedge\delta_b +2J^{ai}\delta_a\wedge\delta_i +J^{ij}\delta_i\wedge\delta_j \right) \otimes\mu
	\end{aligned}
	\]
	where $\delta_\alpha := \partial_\alpha \otimes \mu^\ast$. In particular, $\Lambda_J$ is fiber-wise entire if and only if its components $J^{ab}, J^{ai}, J^{ij}$ are fiber-wise entire functions. Now, let $s\in \Gamma (NS)$ and denote by $\{ \Phi_t \}$ the one parameter group of automorphisms of $L$ generated by $I(s)$. Then, from $P(J) = P (\Lambda_J) = 0$, Equations (\ref{eq:mcl}), (\ref{eq:PJvsPLambda}), and the very definition of the Lie derivative, we get
	\begin{equation*}
		MC (-s) = P \sum_{k=0}^\infty \left. \frac{\partial^k (\Phi_{-t_1 - \cdots - t_k})_\ast \Lambda_J}{\partial t_1 \cdots \partial t_k}\right|_{t_1 = \cdots = t_k = 0}=P \sum_{k=0}^\infty \frac{1}{k!} \left. \frac{d^k}{d t^k} \right|_{t=0}(\Phi_{-t})_\ast \Lambda_J.
	\end{equation*}
	Let $(x,y,\lambda) \in L$, $x \in S$, $y \in N_x S$, $\lambda \in L_x $. Then
	$$
	\Phi_{-t} (x,y, \lambda) = (x, y -t s(x) , \lambda)
	$$
	and
	\[
	\begin{aligned}
	(\Phi_{-t})_\ast \Lambda_J & = (J^{a b} \circ \Phi_t)\delta_a\wedge\delta_b\otimes\mu +2(J^{ai} \circ \Phi_t) \delta_a\wedge(\delta_i - t s_{i}^b \delta_b )\otimes\mu\\
	&\quad +(J^{ij}\circ \Phi_t )(\delta_i - t s_{i}^a \delta_a )\wedge (\delta_j - t s_{j}^b \delta_b)\otimes\mu ,
	\end{aligned}
	\]
	where $s^a_i$ denotes the partial derivative with respect to~$x^i$ of the $a$-th local component of $s$ in the local basis $(\partial_a)$ of $\Gamma (NS)$. Hence
	\begin{equation}
	\label{eq:MC_Taylor}
		MC(-s)\! =\! \sum_{k=0}^\infty\!\frac{1}{k!}\!\left. \frac{d^k}{dt^k}\right|_{t=0}\!\!\left( J^{ab}\! \circ\! ts - 2t s_i^b (J^{ai}\!\circ\! ts ) + t^2 s_i^a s_j^b (J^{ij}\! \circ\! ts) \right) \delta_a \!\wedge\! \delta_b\! \otimes\! \mu .
	\end{equation}
	Assume that $\Lambda_J$ is fiber-wise entire. Then the Taylor expansions in $t$, around $t = 0$, of $J^{ab} \circ ts$, $J^{ai} \circ ts$, and $J^{ij} \circ ts$ converge for all $t$'s, in particular for $t = 1$. It immediately follows that the series in the right hand side of~(\ref{eq:MC_Taylor}) converges as well. This proves the ``only if'' part of the proposition (cf.~the proof of the analogous proposition in \cite{SZ2012}).
	
	For the ``if part'' of the proposition assume that the series in the right hand side of~(\ref{eq:MC_Taylor}) converges for all $s$. First of all, locally, we can choose $s$ to be ``constant'' with respect to~coordinates $(x^i, y^a)$. Then $s^a_i = 0$ and (\ref{eq:MC_Taylor}) reduces to
	\begin{equation}\label{Taylor}
		MC(-s) = \sum_{k=0}^\infty \frac{1}{k!}\left. \frac{d^k}{dt^k}\right |_{t=0}\left( J^{ab} \circ ts \right) \delta_a \wedge \delta_b \otimes \mu .
	\end{equation}
	Since $s$ is arbitrary, (\ref{Taylor}) shows that the $J^{ab}$'s are entire on any straight line through the origin in the fibers of $NS$. Since the Taylor series of the restriction to such a straight line is the same as the restriction of the Taylor series, we conclude that the $J^{ab}$'s are fiber-wise entire. Now, fix values $i_0, a_0$ for the indexes $i,a$ respectively, and choose $s$ so that $s^a_i = \delta_i^{i_0}\delta_{a_0}^a$ to see that the $J^{a_0 i_0}$'s are fiber-wise entire for all $a_0, i_0$. One can prove that the $J^{ij}$'s are fiber-wise entire in a similar way. This concludes the proof.
\end{proof}

\begin{corollary}\label{cor:conver}
	Let $(M, L, J \equiv \{-,-\})$ be a Jacobi manifold, and let $S \subset M$ be a coisotropic submanifold equipped with a fat tubular neighborhood $\tau : L_{NS} \INTO L$. If $\tau_\ast ^{-1} \Lambda_J $ is fiber-wise entire, then a section $s : S \to NS$ of $NS$ is coisotropic if and only if the Maurer-Cartan series $MC (-s)$ converges to zero.
\end{corollary}

\subsection{Moduli of coisotropic sections}
Jacobi diffeomorphisms, in particular Hamiltonian diffeomorphisms, preserve coisotropic submanifolds. Two coisotropic submanifolds are \emph{Hamiltonian equivalent} if there is an Hamiltonian isotopy (i.e.~a one parameter family of Hamiltonian diffeomorphisms) interpolating them. With this definition at hand one can define a moduli space of coisotropic submanifolds under \emph{Hamiltonian equivalence}. Now, let $S$ be a coisotropic submanifold. In this section we adapt the definition of Hamiltonian equivalence to the case of coisotropic sections of $NS \to S$ \cite[Definition 6.3]{LO2012}. In this way we define a local version of the moduli space under Hamiltonian equivalence.

\begin{definition}\label{def:hequi} (cf. \cite[Definition 10.2]{LO2012}).\
	\begin{enumerate}
		\item Two coisotropic sections $s_0, s_1 \in \Gamma (NS)$ are called \emph{Hamiltonian equivalent} if they are interpolated by a smooth family of sections $s_t \in \Gamma (NS)$ and there exists a family of Hamiltonian diffeomorphisms $\psi_t : NS \to NS$ of $(NS, L, J \equiv \{-,-\})$ (i.e.~the family $\{\psi_t\}$ is generated by a family $\{ X_{\lambda_t}\}$ of Hamiltonian vector fields, where the $\lambda_t$'s depend smoothly on $t$) and a family of diffeomorphisms $g_t : S \to S$, $t \in [0,1]$, such that $g_0 = \operatorname{id}_{S}$, $\psi_0 = \operatorname{id}_{NS}$ and
		$s_t = \psi_t \circ s_0 \circ g_t^{-1}$. A coisotropic deformation of $S$ is \emph{trivial} if it is Hamiltonian equivalent to the zero section.
		\item Two infinitesimal coisotropic deformations $s_0, s_1 \in \Gamma (NS)$ are called \emph{infinitesimally Hamiltonian equivalent} if $s_1 - s_0$ is the vertical component along $S$ of an Hamiltonian vector field. An infinitesimal coisotropic deformation is trivial if it is infinitesimally Hamiltonian equivalent to the zero section.
	\end{enumerate}
\end{definition}

Note that both Hamiltonian equivalence and infinitesimal Hamiltonian equivalence are equivalence relations. The notion of infinitesimal Hamiltonian equivalence is motivated by the following remark.

\begin{remark}
	Let $s_0,s_1$ be Hamiltonian equivalent coisotropic sections, and let $s_t$ be the family of sections interpolating them as in Definition \ref{def:hequi}.(1). Then $s_t$ is obviously a coisotropic section for all $t$. Moreover, $s_0$ and
	\[
	s_0 + \left. \frac{d}{dt}\right|_{t= 0} s_t
	\]
	are infinitesimally Hamiltonian equivalent coisotropic sections.
\end{remark}

\begin{proposition}\label{prop:Hameq}
	Let $s_0,s_1 \in \Gamma (NS)$ be Hamiltonian equivalent co\-iso\-tro\-pic sections. Then $s_0,s_1$ are interpolated by a smooth family of sections $s_t \in \Gamma (NS)$ and there exists a smooth family of sections $\lambda_t$ of the Jacobi bundle $L$ such that $s_t$ is a solution of the following evolutionary equation:
	\begin{equation}\label{eq:Hameq}
		\frac{d}{dt} s_t = P (\exp I(-s_t)_\ast \Delta_{\lambda_t}).
	\end{equation}
	If $S$ is compact, the converse is also true.
\end{proposition}
\begin{proof}
	Denote by $\pi : NS \to S$ the projection. First of all, let $s_0,s_1$ be Hamiltonian equivalent coisotropic sections, and let $s_t$, $\psi_t$, $g_t$ be as in Definition~\ref{def:hequi}.(1). The $g_t$'s are completely determined by the $\psi_t$'s via $g_t = \pi \circ \psi_t \circ s_0$. In their turn, the $\psi_t$'s are generated by a smooth family $\{ X_{\lambda_t} \}$ of Hamiltonian vector fields, $\lambda_t \in \Gamma (L)$. Differentiating the identity $s_t = \psi_t \circ s_0 \circ g_t^{-1}$ with respect to $t$, one finds
	\begin{equation}\label{eq:HameqPs}
		\frac{d}{dt} s_t= P^{s_t} (\Delta_{\lambda_t}),
	\end{equation}
	where, for a section $s \in \Gamma (NS)$, the projection $P^s :  (\Der^\bullet L)[1] \to \Gamma (\wedge^\bullet N_\ell S \otimes \ell)[1]$ is defined as in the proof of Proposition \ref{prop:coiss}. To see this, interpret the $s_t$'s as smooth maps, and consider their pull-backs $s_t^\ast : C^\infty (NS) \to C^\infty (S)$. Then $s_t^\ast = (g_t^{-1})^\ast \circ s_0^\ast \circ \psi^\ast_t$ and a straightforward computation shows that
	\[
	\frac{d }{dt} s_t^\ast= s_t^\ast \circ X_{\lambda_t} \circ (\id - \pi^\ast \circ s_t^\ast).
	\]
	which is equivalent to \eqref{eq:HameqPs}. Equation (\ref{eq:Hameq}) now follows from (\ref{eq:Ps}).
	
	Conversely, let $S$ be compact, $s_t$ be a solution of Equation (\ref{eq:Hameq}) interpolating $s_0$ and $s_1$, and let $\{ \psi_t \}$ be the one parameter family of Hamiltonian diffeomorphisms $NS \to NS$ generated by $\{ X_{\lambda_t} \}$. The compactness assumption guarantees that $\psi_t$ is well-defined for all $t \in [0,1]$ (see, e.g.~\cite[Lemma 3.15]{SZ2014}). In view of (\ref{eq:Ps}) again, $s_t$ is the (unique) solution of (\ref{eq:HameqPs}) starting at $s_0$. In particular, $\psi_t$ maps diffeomorphically the image of $s_0$ to the image of $s_t$. Hence, the map $g_t = \pi \circ \psi_t \circ s_0$ is a diffeomorphism and $s_t = \psi_t \circ s_0 \circ g_t^{-1}$.
\end{proof}

Note that if $\{ s_t \}$ is a solution of (\ref{eq:Hameq}) interpolating coisotropic sections $s_0,s_1$, then $s_t$ is a coisotropic section for all $t$. Proposition \ref{prop:Hameq} motivates the following

\begin{definition}\label{def:formHeq}
	Two formal coisotropic deformations
	\begin{equation*}
		s_0 (\eps), s_1(\eps) \in \Gamma (NS)[[\eps]]
	\end{equation*}
	are called \emph{Hamiltonian equivalent} if they are interpolated by a smooth family of formal coisotropic deformations $s_t (\eps) \in \Gamma (NS)[[\eps]]$ (i.e.~$s_t (\eps) = \sum_i s_{t,i} \eps^i$ and the $s_{t,i}$'s depend smoothly on $t$) and there exists a smooth family of formal sections $\lambda_t (\eps) \in \Gamma (L)[[ \eps]]$ of the Jacobi bundle such that
	\[
	\frac{d}{dt} s_t (\eps)= P (\exp \Ll_{I (s_t( \eps))} \Delta_{\lambda_t (\eps)}).
	\]
\end{definition}

We now show that formal coisotropic deformations $s_0 (\eps), s_1(\eps)$ are Hamiltonian equivalent if and only if $-s_0 (\eps), -s_1 (\eps)$ are gauge equivalent solutions of the Maurer-Cartan equation $MC(\xi(\eps)) = 0$. Two solutions $\xi_0 (\eps) , \xi_1 (\eps)$ of the Maurer-Cartan equation are \emph{gauge equivalent} if, by definition, they are interpolated by a smooth family of formal sections $\xi_t (\eps) \in \Gamma (NS) [[\eps]]= \Gamma (\wedge^1 N_\ell S \otimes \ell) [[\eps]]$ and there exists a smooth family of formal sections $\lambda_t (\eps) \in \Gamma (\ell)[[ \eps]] =\Gamma (\wedge^0 N_\ell S \otimes \ell) [[\eps]]$ such that
\begin{equation}\label{eq:HameqMC}
	\frac{d }{dt} \xi_t (\eps) = \sum_{k = 0}^\infty \frac{1}{k!} \mathfrak{m}_{k+1} (\xi_t (\eps) , \dots, \xi_t (\eps) , \lambda_t (\eps)).
\end{equation}
Gauge equivalence is an equivalence relation. Moreover, it follows from Equation (\ref{eq:HameqMC}) that $\xi_t(\eps)$ is a solution of the Maurer-Cartan equation for any $t$.

\begin{proposition}\label{prop:MC_gauge}
	Two formal coisotropic deformations
	\begin{equation*}
		s_0 (\eps), s_1 (\eps) \in \Gamma (NS)[[\eps]]
	\end{equation*}
	are Hamiltonian equivalent if and only if $-s_0(\eps)$ and $-s_1(\eps)$ are gauge equivalent solutions of the Maurer-Cartan equation.
\end{proposition}

\begin{proof} Recall that $\ker P \subset  (\Der^\bullet L)[1]$ is a Lie subalgebra. As Voronov notes \cite{Voronov2005}, this can be rephrased as:
	\begin{equation}\label{eq:kerP}
		P[\square_1, \square_2]^{SJ} = P[IP\square_1 , \square_2]^{SJ} + P[\square_1, IP \square_2]^{SJ},
	\end{equation}
	$\square_1, \square_2 \in  (\Der^\bullet L)[1]$. Now, let $\{ s_t (\eps) \}$ be a family of formal coisotropic deformations, and let $\{ \lambda_t (\eps) \}$ be a family of formal sections of $L$. Put
	\[
	J_k (\eps) := [ \cdots [ J, \underset{k \text{ times}}{\underbrace{I(-s_t(\eps))]^{SJ} \cdots, I(-s_t(\eps))]^{SJ}}},
	\]
	In particular, $P J_k (\eps) = \mathfrak m_k (-s_t (\eps), \dots, -s_t(\eps))$. Compute
	\begin{align*}
		\quad  P(\exp \Ll_{I (s_t (\eps))} \Delta_{\lambda_t (\eps)} ) &  = -\sum_{k=0}^\infty \frac{1}{k!} P [ J_k (\eps), \lambda_t (\eps)]^{SJ} \\
		& = -\sum_{k=0}^\infty \frac{1}{k!} P [ IPJ_k (\eps), \lambda_t (\eps)]^{SJ}
		-\sum_{k=0}^\infty \frac{1}{k!} P [ J_k (\eps), IP\lambda_t (\eps)]^{SJ} \\
		& = - P [ I (MC (-s_t(\eps))), \lambda_t (\eps)]^{SJ}
		-\sum_{k=0}^\infty \frac{1}{k!} P [ J_k (\eps), I (\lambda_t (\eps)|_S)]^{SJ} \\
		& = -\sum_{k=0}^\infty \frac{1}{k!} \mathfrak{m}_{k+1} (-s_t(\eps), \dots, -s_t(\eps), \lambda_t (\eps) |_S ),
	\end{align*}
	where we used (\ref{eq:kerP}), and the fact that $MC(-s_t (\eps)) = 0$ for all $t$. This concludes the proof.
\end{proof}

\begin{corollary}\label{cor:infequi} Two solutions of~\eqref{eq:infdef} are infinitesimally Hamiltonian equivalent if and only if they are cohomologous in the complex $(\Gamma (\wedge^\bullet N_\ell S \otimes \ell)[1], \mathfrak{m}_1)$. Hence, the infinitesimal moduli space (i.e.~the set of infinitesimal Hamiltonian equivalence classes) of infinitesimal coisotropic deformations of $S$ is $H^0(\Gamma (\wedge^\bullet N_\ell S \otimes \ell)[1], \mathfrak{m}_1) = H^1 (N_\ell {}^\ast S , \ell)$.
\end{corollary}

\begin{remark}\label{rem:moduli_1}\
	Corollary \ref{cor:infequi} generalizes \cite[Lemma 6.6]{LO2012}, which has been proved by a different method.
\end{remark}
Now, we establish necessary and sufficient conditions for the convergence of both the Maurer-Cartan series $MC(-s)$ and the series
\begin{equation}\label{eq:delta_MC}
	\delta_\lambda MC(-s) := \sum_{k=0}^\infty \frac{1}{k!} \mathfrak{m}_{k+1} (-s, \dots, -s, \lambda )
\end{equation}
for generic sections $s \in \Gamma(NS)$ and $\lambda \in \Gamma (\ell)$. In this way, we can describe moduli of coisotropic sections in terms of gauge equivalence classes of non-formal solutions of the Maurer-Cartan equation. First of all, let $E$ and $L$ be as in the beginning of Section \ref{sec:radially_entire}. A multi-differential operator $\Delta \in  (\Der^\bullet L)[1]$ is \emph{fiber-wise entire} if it maps linear sections (of $L$) to fiber-wise entire sections. Equivalently, $\Delta$ is fiber-wise entire if its components in vector bundle coordinates are fiber-wise entire.

\begin{theorem}\label{prop:fana1} The Jacobi bi-differential operator $J$ is fiber-wise entire iff, for all sections $s \in \Gamma (NS)$, and $\lambda \in \Gamma (L)$, the Maurer-Cartan series $MC (-s)$ converges to $P (\exp I(s)_\ast J)$, and the series $\delta_{\lambda|_S} MC (-s)$ (\ref{eq:delta_MC}) converges to $P(\exp I(s)_\ast \Delta_\lambda)$, in the sense of point-wise convergence.
\end{theorem}
\begin{proof} We already know that the bi-linear form $\Lambda_J$ is fiber-wise entire if and only if $MC(-s)$ converges for all $s$. Now, it is easy to see that $P (\exp \Ll_{I(s)} \Delta_\lambda) = P (\exp \Ll_{I(s)} X_\lambda)$ for all $s \in \Gamma (NS)$, and $\lambda \in \Gamma (L)$ (cf.~(\ref{eq:PJvsPLambda})). Moreover, from the proof of Proposition \ref{prop:MC_gauge}, we get
	\[
	\delta_{\lambda|_S} MC (-s) = - P (\exp \Ll_{I(s)} \Delta_\lambda) = - P (\exp \Ll_{I(s)} X_\lambda).
	\]
	Therefore, similarly as in the proof of Proposition \ref{prop:fana}, we find
	\[
	\delta_{\lambda|_S} MC (-s) = -P \sum_{k=0}^\infty \frac{1}{k!} \left. \frac{d^k}{d t^k} \right|_{t=0}(\Phi_{-t})_\ast X_\lambda.
	\]
	The bi-differential operator $J$ is locally given by (\ref{eq:components_of_J_bis}), hence a straightforward computation shows that
	\begin{equation*}
		\delta_{\lambda|_S} MC (-s) = \sum_{k=0}^\infty \frac{1}{k!}\! \left.\frac{d^k}{dt^k}\!\right|_{t=0}\!\! \left[ 2\partial_i g (J^{ai}\! \circ\! ts)\!-\!2ts_j^a\partial_i g (J^{ij}\! \circ\! ts)\! +\! g (J^a\! \circ\! ts) \!-\!t s^a_i g (J^i\! \circ\! s) \right]\! \partial_a,
	\end{equation*}
	where we used the same notations as in the proof of Proposition \ref{prop:fana}, and $g$ is the component of $\lambda|_S$ in the basis $\mu$.
	The assertion now follows in a very similar way as in the proof of Proposition \ref{prop:fana}.
\end{proof}

\begin{corollary}\label{cor:hequi}
	Let $(M, L, J \equiv \{-,-\})$ be a Jacobi manifold, and let $S \subset M$ be a compact coisotropic submanifold equipped with a fat tubular neighborhood $\tau : \ell \INTO L$. If $\tau_\ast^{-1}J$ is fiber-wise entire, then two solutions $s_0, s_1 : S \to NS$ of the (well-defined) Maurer-Cartan equation $MC (-s) = 0$ are Hamiltonian equivalent if and only if they are interpolated by a smooth family of sections $s_t \in \Gamma (NS)$ and there exists a smooth family of sections $\lambda_t$ of $\ell$ such that $s_t$ is a solution of the following well-defined evolutionary equation:
	\[
	\frac{d}{dt}(-s_t) = \delta_{\lambda_t} MC (-s_t).
	\]
\end{corollary}

\begin{remark}\label{rem:moduli_2} \
	Immediately after a preliminary version of the present work appeared on arXiv, Sch\"atz and Zambon, independently, finalized a pre-print, now published \cite{SZ2014}, where they discuss the moduli space of coisotropic submanifolds of a symplectic manifold. In particular, they use our same method to prove Corollary \ref{cor:hequi} in the symplectic case (see \cite[Theorem 3.21]{SZ2014}). Note that $\tau_\ast^{-1}J$ is automatically fiber-wise entire in Sch\"atz-Zambon situation and, therefore, convergence issues don't appear in their work.
\end{remark}

\section{The contact case}\label{sec:contact}
Contact manifolds form a distinguished class of Jacobi manifolds. In this section we consider in some details (regular) coisotropic submanifolds in a contact manifold $(M,C)$. A normal form theorem is available in this case. As a consequence, the $L_\infty$-algebra of a regular coisotropic submanifold $S$ in $(M,C)$ does only depend on the intrinsic pre-contact geometry of $S$. In particular, we get rather efficient formulas (from a computational point of view) for the multibrackets, analogous to those of Oh and Park in the symplectic case \cite[Equation (9.17)]{OP2005}.

\subsection{Coisotropic submanifolds in contact manifolds}\label{sec:cois_cont}

Let $C$ be an hyperplane distribution on a smooth manifold $M$.
Denote by $L$ the quotient line bundle $TM/C$, and by $\theta:TM\to L$, $X \mapsto \theta (X) := X \Mod C$ the projection. We will often interpret $\theta$ as an $L$-valued differential $1$-form on $M$, and call it the \emph{structure form of $C$}. The \emph{curvature form} of $(M,C)$ is the vector bundle morphism $\omega:\wedge^2C\to L$ well-defined by
$\omega(X,Y)=\theta([X,Y])$, with $X,Y\in\Gamma(C)$.
Consider also the vector bundle morphism $\omega^\flat : C \to C^\ast \otimes L$, $X\mapsto\omega^\flat(X):=\omega(X,-)$. The \emph{characteristic distribution} of $(M,C)$, is the (generically singular) distribution $\ker\omega^\flat = C^{\bot_\omega}$, where we denoted by $V^{\bot_\omega}$ the $\omega$-orthogonal complement of a subbundle $V \subset C$. Note that the definition of curvature form works verbatim for distributions of arbitrary codimension (See also \cite[Section 4]{OP2005} for a detailed exposition on the curvature form).

\begin{remark}
	The characteristic distribution of an hyperplane distribution $C$ is involutive.
\end{remark}

\begin{definition}
	\label{mydef:pre-contact}
	A \emph{pre-contact structure} on a smooth manifold $M$ is an hyperplane distribution $C$ on $M$ such that its characteristic distribution $\ker \omega^\flat$ has constant dimension.
	A \emph{pre-contact manifold} $(M,C)$ is a smooth manifold $M$ equipped with a pre-contact structure $C$. The integral foliation of $\ker \omega^\flat$ is called the \emph{characteristic foliation} of $C$ and will be denoted by $\mathcal{F}$.
\end{definition}

See \cite[Section 5]{OW2013} where essentially the same definition was given in terms of the one-form generating the
hyperplane distribution in relation
to the study of normal forms of a contact form of Morse-Bott type.

\begin{remark}
	The curvature form $\omega$ of $(M,C)$ measures how far is $C$ from being integrable.
	Indeed, $C$ is integrable if and only if $\omega = 0$, or, equivalently, $\omega^\flat=0$. Accordingly, $C$ is said to be \emph{maximally non-integrable} when $\omega$ is non degenerate, or, equivalently, $\ker \omega^\flat =0$. If $C$ is maximally non-integrable, then $C$ is even-dimensional, $M$ is odd-dimensional, and $\omega^\flat$ is a vector bundle isomorphism, whose inverse will be denoted by $\omega^\#:C^\ast\otimes L\to C$.
\end{remark}

\begin{definition}
	\label{mydef:contact}
	A \emph{contact structure} on a smooth manifold $M$ is a maximally non-integrable hyperplane distribution $C$ on $M$.
	A \emph{contact manifold} is a smooth manifold $M$ equipped with a contact structure $C$.
\end{definition}

Let $(M_1,C_1)$ and $(M_2, C_2)$ be contact manifolds. A contactomorphism $\phi:(M_1,C_1) \to (M_2,C_2)$ is a diffeomorphism $\phi:M_1\to M_2$ such that
\begin{equation*}
	(d\phi)C_1=C_2.
\end{equation*}

An \emph{infinitesimal contactomorphism} (or \emph{contact vector field}) of a contact manifold $(M,C)$ is a vector field $X\in\mathfrak X(M)$ whose flow consists of local contactomorphisms. Equivalently, $X \in \mathfrak X(M)$ is a contact vector field if
$
[X, \Gamma (C)] \subset \Gamma (C)
$.
Contact vector fields of $(M,C)$ form a Lie subalgebra of $\mathfrak X(M)$ which will be denoted by $\mathfrak X_C$ (see e.g. \cite[Proposition 2.3]{OW2013}).

\begin{proposition}[cf.~\cite{CS2013}, {\cite[Proposition 2.3]{OW2013}}]
	\label{prop:contact_splitting}
	Let $(M,C)$ be a contact manifold.
	There is a natural direct sum decomposition of $\mathbb R$-vector spaces:
	$
	\mathfrak X(M)=\mathfrak X_C\oplus\Gamma(C).
	$
\end{proposition}

\begin{proof}
	For $X\in\mathfrak X(M)$, let $\phi_X\in\Gamma(C^\ast\otimes L)$ be defined by $\phi_X(Y)=\theta([X,Y])$, $Y\in\Gamma(C)$. The first order differential operator
	$
	\phi:\mathfrak X(M)\to\Gamma(C^\ast\otimes L)$, $X\mapsto\phi_X,
	$
	sits in a short exact sequence of $\mathbb R$-linear maps
	\begin{equation}
		\label{eq:contact_splitting_short_exact_sequence}
		0\longrightarrow \mathfrak X_C \lhook\joinrel\longrightarrow \mathfrak X(M)\overset{\phi}{\longrightarrow} \Gamma(C^\ast\otimes L)\longrightarrow 0,
	\end{equation}
	where the second arrow is the inclusion. Now the $C^\infty(M)$-linear map $\Gamma (C^\ast \otimes L) \to \mathfrak{X}(M)$ given by the composition
	\[
	\Gamma(C^\ast\otimes L)\overset{\omega^\#}{\longrightarrow} \Gamma(C)\longrightarrow \mathfrak X(M)
	\]
	splits the sequence (\ref{eq:contact_splitting_short_exact_sequence}).
\end{proof}
In what follows, for $\lambda \in \Gamma (L)$, we denote by $X_\lambda$ the unique contact vector field such that $\theta (X_\lambda) = \lambda$.

\begin{proposition}
	\label{prop:Jacobi_from_contact}
	A contact structure $C$ induces a canonical Jacobi structure $(L,\{-,-\})$, where the Lie bracket $\{-,-\}$ on $\Gamma(L)$ is uniquely determined by
	$
	X_{\{\lambda,\mu\}}=[X_\lambda,X_\mu]$. The symbol of the first order differential operator $\Delta_\lambda := \{\lambda,-\}\in  \der L$ is $X_\lambda$.
\end{proposition}

Now, let $(M,C)$ be a contact manifold, and let $S \subset M$ be a submanifold.
The intersection $C_S := C\cap TS$ is a generically singular distribution on $S$. More precisely $S$ is the union of two disjoint subsets $S_0, S_1$ defined by
\begin{itemize}
	\item $p \in S_0$ if and only if $\dim(C_S)_p=\dim S$,
	\item $p \in S_1$ if and only if $\dim(C_S)_p=\dim S-1$.
\end{itemize}
If $S = S_0$ then $S$ is said to be an \emph{isotropic submanifold} of $(M,C)$. In other words, an isotropic submanifold of $(M,C)$ is an integral manifold of the contact distribution $C$. Locally maximal isotropic, or, equivalently, locally maximal integral submanifolds of $C$ are \emph{Legendrian submanifolds}.

\begin{proposition}
	\label{prop:coisotropics_contact_setting}
	Let $S = S_1$. The following conditions are equivalent:
	\begin{enumerate}
		\item $C_S$ is a pre-contact structure on $S$, with characteristic distribution given by $(C_S)^{\bot_\omega} \subset C|_S$,
		\item $(C_S)_p$ is a coisotropic subspace in the symplectic vector space $(C_p,\omega_p)$, i.e.~$(C_S)_p^{\bot_\omega} \subset (C_S)_p$, for all $p\in S$,
		\item $S$ is a coisotropic submanifold of the associated Jacobi manifold\linebreak $(M,L,J \equiv \{-,-\})$.
	\end{enumerate}
\end{proposition}

\begin{proof}
	The equivalence $1)\Longleftrightarrow 2)$ amounts to a standard argument in symplectic linear algebra. The equivalence $2)\Longleftrightarrow 3)$ is based on the following facts. Let $(L, J \equiv \{-,-\})$ be the Jacobi structure associated to $(M,C)$. For $\lambda\in\Gamma(L)$, and $f\in C^\infty(M)$ put $Y_{f,\lambda}:=\Lambda_J^\#(df \otimes \lambda) =X_{f\lambda}-fX_\lambda$. We have the following:
	\begin{itemize}
		\item $Y_{f, \lambda} \in \Gamma(C)$.
		
		\item Let $I(S) \subset C^\infty (M)$ be the ideal of functions vanishing on $S$. Then
		$
		Y_{f,\lambda}$ is tangent to $S$ if and only if $X_{f\lambda}$ is tangent to $S$,
		for all $f\in I(S)$, and $\lambda\in\Gamma(L)$.
		
		\item
		$
		\omega(Y_{f,\lambda},X)=X(f)\lambda,
		$
		for all $f\in\ C^\infty(M)$, $\lambda\in\Gamma(L)$, and $X\in\Gamma(C)$.

	\end{itemize}
	Now it is easy to see that $(C_S)^{\bot_\omega}\subset C_S$ if and only if $S$ is coisotropic in $(M,L,\{-,-\})$.
\end{proof}

\begin{definition}
	If the equivalent conditions $1)$--$3)$ in Proposition \ref{prop:coisotropics_contact_setting} are satisfied, then $S$ is said to be a \emph{regular coisotropic submanifold} of $(M,C)$.
\end{definition}

\begin{remark}
	Unlike the equivalence $1)\!\Longleftrightarrow\! 2)$, in Proposition \ref{prop:coisotropics_contact_setting}, the equivalence $2)\!\Longleftrightarrow\! 3)$ continues to hold also without assuming that $S \!=\! S_1$.
\end{remark}

\begin{remark}
	Let $(M, L , \{-,-\})$ be a Jacobi manifold. Then $(L,\{-,-\})$ is the Jacobi structure induced by a (necessarily unique) contact structure if and only if the associated bi-linear form $J : \wedge^2 J^1 L \rightarrow L$ is non-degenerate (see \cite{V2015}). In particular, Hamiltonian derivations of a contact manifold, exhaust all infinitesimal Jacobi automorphisms, and Hamiltonian vector fields exhaust all Jacobi vector fields.
\end{remark}

\subsection{Coisotropic embeddings and $L_\infty$-algebras from pre-contact manifolds}
\label{sec:tubular_neighborhood} From now till the end of this section we consider only closed regular co\-iso\-tro\-pic submanifolds.
The intrinsic pre-contact geometry of a regular coisotropic submanifold $S$ in a contact manifold $M$ contains a full information about the coisotropic embedding of $S$ into $M$, at least locally around $S$. This is an immediate consequence of the \emph{Tubular Neighborhood Theorem} in contact geometry (see \cite{L1998}, \cite[Section 5]{OW2013},
see also \cite{G1982} for the analogous result in symplectic geometry).

Let $(S,C_S)$ be a pre-contact manifold, with characteristic foliation $\Ff$.

\begin{definition}
	\label{mydef:coisotropic_embedding}
	A \emph{coisotropic embedding} of $(S,C_S)$ into a contact manifold $(M,C)$ is an embedding $i: S \INTO M$ such that
	$
	(di) C_S = C_{i(S)}$, and $(di)T\Ff=C_{i(S)}^{\bot_{\omega}}$,
	where $\omega $ is the curvature form of $(M,C)$.
\end{definition}

\begin{remark}
	Clearly, in view of Proposition~\ref{prop:coisotropics_contact_setting}, if $i : S \INTO M$ is a co\-iso\-tro\-pic embedding of $(S,C_S)$ into $(M,C)$, then $i(S)$ is a coisotropic submanifold of $(M,C)$.
\end{remark}

Let $i_1$ and $i_2$ be coisotropic embeddings of $(S,C_S)$ into contact manifolds $(M_1,C_1)$ and $(M_2,C_2)$, respectively.
\begin{definition}
	\label{mydef:local_equivalence_coisotropic_embeddings}
	The coisotropic embeddings $i_1$ and $i_2$ are said to be \emph{locally equivalent} if there exist open neighborhoods $U_j$ of $i_j(S)$ in $M_j$, $j=1,2$, and a contactomorphism $\phi:(U_1,C_1|_{U_1})\to(U_2,C_2|_{U_2})$ such that $\phi\circ i_1=i_2$.
\end{definition}

\begin{theorem}
	[Coisotropic embedding of pre-contact manifolds: existence and uniqueness]
	\label{teor:coisotropic_embedding}
	Every pre-contact manifold admits a co\-iso\-tro\-pic embedding. Additionally, any two coisotropic embeddings of a given pre-contact manifold are locally equivalent.
\end{theorem}

Theorem \ref{teor:coisotropic_embedding} is a special case of Theorem 3 in \cite{L1998}. We do not repeat the ``uniqueness part'' of the proof here. The ``existence part'' can be proved constructively via \emph{contact thickening}. This is done for later purposes in the next subsection.

\begin{corollary}[$L_\infty$-algebra of a pre-contact manifold] \label{cor:L_infty_precont}
	Every pre-contact manifold determines a natural isomorphism class of $L_\infty$-algebras.
\end{corollary}

\begin{proof}
	The ``existence part'' of Theorem~\ref{teor:coisotropic_embedding} and Proposition \ref{prop:linfty} guarantee that a pre-contact manifold $(S,C_S)$ determines a $L_\infty$-algebra up to the choice of a coisotropic embedding $(S,C_S)\subset (M,C)$, and a fat tubular neighborhood $\tau : NS \times_S \ell \INTO L$ of $\ell$ in $L$, where $\ell = TS / C_S$ and $L$ is the Jacobi bundle of $(M,C)$. Any two such $L_\infty$-algebras are $L_\infty$-isomorphic because of Proposition~\ref{prop:gauge_invariance} and the ``uniqueness part'' of Theorem~\ref{teor:coisotropic_embedding}.
\end{proof}

\subsection{Contact thickening}
\label{sec:contact_thickening}
We now show that every pre-contact manifold $(S,C_S)$ admits a coisotropic embedding into a suitable contact manifold uniquely determined by $(S,C_S)$ up to the choice of a complementary distribution to the characteristic distribution. Thus, let $(S,C_S)$ be a pre-contact manifold, $\Ff $ its characteristic foliation, $\ell = TS/C_S$ the quotient line bundle, and let $\theta : TS \rightarrow \ell$ be the structure form. Theorem \ref{teor:coisotropic_embedding} is a ``contact version'' of a theorem by Gotay \cite{G1982} and can be proved by a similar technique as the \emph{symplectic thickening} of \cite{OP2005}. Accordingly, we will speak about \emph{contact thickening}.
See also \cite{OW2013} for a relevant discussion on contact thickenings in a different context.

Pick a distribution $G$ on $S$ complementary to $T\Ff $, and let $p_{T\Ff;G} : TS \to T\Ff $ be the projection determined by the splitting $TS = G \oplus T\Ff $.
Put\linebreak $T_\ell {}^\ast \Ff := T^\ast\Ff \otimes \ell$, and let $q : T_\ell {}^\ast\Ff \to S$ be the natural projection. We equip the manifold $T_\ell {}^\ast\Ff$
with the line bundle $L := q^\ast \ell$.
The $\ell$-valued $1$-form $\theta$ can be pulled-back via $q$ to an $L$-valued $1$-form $q^\ast \theta$ on $T_\ell {}^\ast\Ff $.
There is also another $L$-valued $1$-form $\theta_G$ on $T_\ell {}^\ast\Ff$.
It is defined as follows: for $\alpha \in T_\ell {}^\ast\Ff $, and $\xi \in T_\alpha(T_\ell {}^\ast\Ff )$
\[
(\theta_G)_\alpha (\xi) := (\alpha \circ p_{T\Ff;G} \circ d q)(\xi) \in \ell_x = L_\alpha, \quad x := q (\alpha),
\]
where $\alpha$ is interpreted as a linear map $T_x\Ff \to\ell_x$. By definition, $\theta_G$ depends on the choice of the splitting $G$.

\begin{proposition}
	\label{prop:contact_thickening}
	The distribution $C := \ker (\theta_G + q^\ast \theta )$ is a contact structure on a neighborhood $U$ of $\operatorname{im}\mathbf{0}$, the image of the zero section $\mathbf{0}$ of $q$.
	Additionally $\mathbf{0}$ is a coisotropic embedding of $(S,C_S)$ into the contact manifold $(U,C|_U)$.
\end{proposition}

\begin{proof}
	Use Darboux lemma and choose local coordinates $(x^i,u^a,z)$ on $S$ \emph{adapted to $C_S$}, i.e.
	\begin{gather*}
		\Gamma(T \Ff)=\left\langle \partial/\partial x^i \right\rangle, \quad \Gamma(C_S)=\left\langle \partial/\partial x^i, \mathbb C_a\right\rangle, \quad \mathbb C_a = \frac{\partial}{\partial u^a} - C_a \frac{\partial}{\partial z},
	\end{gather*}
	where the $C_a$'s are \emph{linear functions} of the $u^b$'s only. The section $\mu:= \theta(\partial/ \partial z)$ is a local generator of $\Gamma (\ell)$. Moreover $\theta$ is locally given by
	$
	\theta = (dz - C_a du^a) \otimes \mu,
	$
	and the curvature form $\omega_S$ of $C_S$ is locally given by
	\[
	\omega_S = \frac{1}{2} \omega_{ab} du^a |_C \wedge du^b|_C \otimes \mu, \quad \omega_{ab} = \frac{\partial C_b}{\partial u^a} - \frac{\partial C_a}{\partial u^b}.
	\]
	In particular, the skew-symmetric matrix $(\omega_{ab})$ is non-degenerate.
	We will use the following local frame on $S$ adapted to both $C_S$ and $G$:
	\[
	\left(\frac{\partial}{\partial x^i}, \mathbb C'_a,Z \right),
	\]
	where
	$
	\mathbb C'_a := (\mathrm{id} - p_{T\Ff;G}) (\mathbb C_a)
	$,
	and
	$
	Z := (\mathrm{id} - p_{T\Ff;G})(\partial/\partial z)
	$.
	Now, let $\boldsymbol p = (p_i)$ be linear coordinates along the fibers of $q:T_\ell {}^\ast\Ff \to S$ corresponding to the local frame $(dx^i|_{T \mathcal F} \otimes \mu)$. Then $(\partial / \partial x^i,\mathbb C'_a, Z, \frac{\partial}{\partial p_i})$ is a local frame on $T_\ell {}^\ast\Ff$. It is easy to check that locally
	\[
	\Gamma(C) = \left\langle X_i, \mathbb C'_a,\frac{\partial}{\partial p_i}\right\rangle ,
	\]
	where $X_i :=\partial / \partial x^i - p_i Z$.
	Finally, the representative matrix of the curvature of $C$ with respect to the local frames $(X_i, \mathbb C'_a,\frac{\partial}{\partial p_i})$ of $C$ and $Z\, \mathrm{mod}\, C$ of $T (T_\ell {}^\ast \Ff)/C = L$ is
	\begin{equation}\label{eq:cont_thick}
		\left( \begin{array}{ccc}
			0 & 0 & \delta_{i}^j \\
			0 & \omega_{ab}  & 0 \\
			-\delta_{j}^i & 0 & 0
		\end{array} \right) \ \text{up to infinitesimals $O(\boldsymbol p)$}
	\end{equation}
	This shows that $C$ is maximally non-integrable around the zero section of $T_\ell {}^\ast\Ff $.
	Moreover, it immediately follows from (\ref{eq:cont_thick}) that the zero section of $T_\ell {}^\ast\Ff $ is a coisotropic embedding (transversal to fibers of $q$). This concludes the proof.
\end{proof}

The contact manifold $(U,C|_U)$ is called a \emph{contact thickening} of $(S,C_S)$. Now, let $NS$ be the normal bundle of $S$ in $U$. Clearly $NS = T_\ell {}^\ast \Ff$, hence $N_\ell S = T^\ast \mathcal F$. According to the proof of Corollary \ref{cor:L_infty_precont} the choice of a complementary distribution $G$ determines an $L_\infty$-algebra structure on $\Gamma (\wedge^\bullet N_\ell S \otimes \ell)[1]  = \Gamma (\wedge^\bullet T^\ast \Ff \otimes \ell)[1]$. Moreover, such $L_\infty$-structure is actually independent of the choice of $G$ up to $L_\infty$-isomorphisms. Sections of $\wedge^\bullet T^\ast \Ff \otimes \ell$ are $\ell$-valued leaf-wise differential forms on $S$ and we also denote them by $\Omega^\bullet (\Ff, \ell)$ (see below).

\subsection{The transversal geometry of the characteristic foliation}
\label{sec:transversal_geometry}

Similarly as in the symplectic case (cf.~\cite[Section 9.3]{OP2005}), the multi-brackets in the $L_\infty$-algebra of a pre-contact manifold can be expressed in terms of the ``geometry transversal to the characteristic foliation''.
To write down this expression, we have to describe the relevant transversal geometry.
Let $(S,C_S)$ be a pre-contact manifold with characteristic foliation $\Ff $.
Denote by $N\Ff :=TS/T\Ff $ the normal bundle to $\Ff $, and by $N^\ast\Ff =(N\Ff )^\ast = T^0 \Ff \subset T^\ast S$ the conormal bundle to $\Ff $.

Recall that $T \Ff $ is a Lie algebroid. The standard Lie algebroid differential in $\Omega^\bullet(\Ff ):=\Gamma(\wedge^\bullet T^\ast\Ff )$ will be denoted by $d_{\Ff }$ and called the \emph{leaf-wise de Rham differential}. There is a flat $T \Ff $-connection $\nabla$ in $N^\ast \Ff $ well-defined by
\[
\nabla_X \eta :=\mathcal L_X \eta, \quad X \in \Gamma (T \Ff ), \quad \eta \in \Gamma (N^\ast \Ff ).
\]

\begin{remark}
	The connection $\nabla$ is ``dual to the Bott connection'' in $N \Ff $.
\end{remark}

As usual, $\nabla$ determines a differential in $\Omega^\bullet(\Ff ,N^\ast\Ff )\!:=\!\Gamma(\wedge^\bullet T^\ast\Ff \otimes N^\ast \Ff )$ denoted again by $d_{\Ff }$.
There exists also a flat $T \Ff $-connection in $\ell$, denoted again by $\nabla$, and defined by
\begin{equation*}
	\nabla_X \theta(Y) := \theta([X,Y]),\quad X\in\Gamma(T\Ff ),\quad Y\in\mathfrak X(M).
\end{equation*}
The corresponding differential in $\Omega^\bullet(\Ff ,\ell):=\Gamma(\wedge^\bullet T^\ast\Ff \otimes \ell)$ will be also de\-noted by $d_\Ff $. Now, let $J^1_\bot \ell$ be the vector subbundle of $J^1\ell$ given by the kernel of the vector bundle epimorphism
\[
\varphi_\nabla :J^1\ell\longrightarrow T^\ast\Ff \otimes \ell, \quad j^1_x\lambda\longmapsto(d_\Ff \lambda)_x.
\]
Sections of $J^1_\bot \ell$ will be interpreted as sections of $J^1\ell$ ``transversal to $\Ff $''. Note also that the Spencer sequence $0 \rightarrow T^\ast S \otimes \ell\rightarrow J^1\ell \rightarrow \ell \rightarrow 0$ restricts to a ``transversal Spencer sequence'' $0 \rightarrow N^\ast \Ff \otimes \ell\rightarrow J_\bot ^1\ell \rightarrow \ell\rightarrow 0$ and the two fit in the following exact commutative diagram of vector bundle morphisms
\begin{equation*}
		\xymatrix{ &0\ar[d]&0\ar[d]&0\ar[d]&\\
			0\ar[r]&N^\ast\Ff \otimes \ell \ar[r]\ar[d]&J^1_\bot \ell\ar[r]\ar[d]&\ell\ar[r]\ar@{=}[d]&0\\
			0\ar[r]&T^\ast S\otimes \ell\ar[r]\ar[d]&J^1 \ell\ar[r]\ar[d]^-{\varphi_\nabla}&\ell\ar[d]\ar[r]&0\\
			0\ar[r]&T^\ast\Ff \otimes \ell\ar[d]\ar@{=}[r]&T^\ast\Ff \otimes \ell\ar[r]\ar[d]&0&&\\
			&0&0&&&}
\end{equation*}

In what follows the embeddings $\gamma: T^\ast S\otimes \ell \INTO J^1\ell$ and $N^\ast\Ff \otimes \ell \INTO J^1_\bot \ell$ will be understood, and we will identify $df\otimes\lambda$ with $j^1(f \lambda) - f j^1\lambda$, for any $f\in C^\infty (S)$, and $\lambda\in\Gamma(\ell)$.
Recall that an arbitrary $\alpha\in\Gamma(J^1\ell)$ can be uniquely decomposed as $\alpha=j^1\lambda+\eta$, with $\lambda\in\Gamma(\ell)$, and $\eta\in\Gamma(T^\ast S\otimes \ell)$.
Then, by definition, for $p\in S$, $\alpha_p$ is in $J^1_\bot \ell$ if and only if $\varphi_\nabla(\eta_p)=-(d_{\Ff }\lambda)_p$.
Finally, there is a flat $T\Ff $-connection in $J^1_\bot \ell$, also denoted by $\nabla$, well-defined by
\begin{equation}\label{eq:nabla_J1}
	\nabla_X \psi= \Ll_{\nabla_X} \psi,
\end{equation}
for all $\psi \in \Gamma (J^1_\bot L)$ and $X\in\Gamma(T\mathcal{F})$. 
Accordingly, there exists a differential in $\Omega^\bullet(\Ff ,J^1_\bot \ell) :=\Gamma(\wedge^\bullet T^\ast\Ff \otimes J^1_\bot \ell)$ which we also denote by $d_{\Ff }$.

Now, note that the curvature form of $(S,C_S)$, $\omega_S:\wedge^2 C_S\rightarrow \ell$, descends to a(n $\ell$-valued) symplectic form $\omega_\bot:\wedge^2(C_S/T\Ff )\rightarrow \ell$. In particular, it determines a vector bundle isomorphism $\omega_\bot^\flat:C_S/T\Ff \rightarrow(C_S/T\Ff )^\ast\otimes \ell$ (see Section \ref{sec:cois_cont}).

\begin{remark}
	Let $p\in S$, $X\in\mathfrak X(S)$, and $\lambda=\theta(X)$.
	Recall that $\phi_X\in\Gamma(C_S^\ast\otimes \ell)$ is defined by $\phi_X(Y)=\theta([X,Y])$, for all $Y\in\Gamma(C_S)$ (cf. Section~\ref{sec:cois_cont}).
	Then we have that
	$
	j_p^1\lambda\in J^1_\bot \ell$ if and only if $(\phi_X)_p\in (C_S/T\Ff )^\ast\otimes \ell
	$.
	Furthermore it is easy to check, for instance using local coordinates, that when $j^1_p\lambda=0$ the following holds:
	\begin{enumerate}
		\item $X_p\in (C_S)_p$, and
		\item $\omega(X_p,Y_p)=\theta([X,Y]_p)$, for all $Y\in\Gamma (C_S)$.
	\end{enumerate}
	Therefore, if $j^1_p\lambda=0$, then $X_p\Mod T_p\Ff =(\omega_\bot^\flat)^{-1}(\phi_X)_p$, and the following definition is well-posed.
\end{remark}

\begin{definition}
	Define $\sigma J_\bot^\#:J^1_\bot \ell\to N\Ff $ to be the vector bundle morphism uniquely determined by:
	\begin{equation}\label{eq:sharp_bot}
		\sigma J_\bot^\#(j_p^1\lambda):=X_p{}\Mod T_p\Ff -(\omega_\bot^\flat)^{-1}(\phi_X)_p,
	\end{equation}
	where $p\in M$, $\lambda\in\Gamma(\ell)$, and $X\in\mathfrak X(S)$, such that $j^1_p\lambda \in J^1_\bot L$, and $\lambda=\theta(X)$.
\end{definition}

\begin{proposition}
	There exists a vector bundle morphism $J_\bot:\wedge^2 J^1_\bot \ell\to \ell$ uniquely determined by putting
	\begin{equation}\label{eq:Lambda_bot}
		J_\bot(j^1_p \lambda,j^1_p \lambda ')=\theta([Y,Y^\prime]_p),
	\end{equation}
	where $p\in M$, $\lambda, \lambda '$ are $\nabla$-constant local sections of $\ell$ and $Y,Y^\prime\in\mathfrak X(S)$ are such that $\sigma J_\bot^\#(j^1 \lambda)=Y \Mod \Gamma (T\Ff )$ and $\sigma J_\bot^\#(j^1 \lambda^\prime)=Y^\prime \Mod \Gamma (T\Ff ) $.
\end{proposition}
\begin{proof}
	First of all notice that every point in $J^1_\bot \ell$ is the first jet of a $\nabla$-constant local section of $\ell$. Hence Definition (\ref{eq:Lambda_bot}) makes sense. Moreover, the right hand side of (\ref{eq:Lambda_bot}) does only depend on $\lambda, \lambda^\prime$. Indeed, first of all, $\theta (Y) = \lambda$, and $\theta (Y ') = \lambda '$. Moreover, if $Y \in \Gamma (T \Ff)$, then, $0 = \nabla_Y \lambda ' = \theta ([Y,Y '])$. Finally, one can check, e.g.~using local coordinates, that the right hand side of (\ref{eq:Lambda_bot}) does actually depend on the first jets at $p$ of $\lambda, \lambda '$ only. This shows that $J_\bot$ is well-defined.
\end{proof}

The vector bundle morphism $J{}_\bot :\wedge^2 J^1_\bot \ell\to \ell$ will be interpreted as the \emph{trans\-ver\-sal version} of the bi-linear form $J $ associated to a Jacobi bi-differential operator $J$.

\subsection{An explicit formula for the multi-brackets}
\label{sec:multi-brackets}

Retaining the notations from the previous subsection, choose a distribution $G$ on $S$ which is complementary to $T\Ff $, i.e.~$TS=G\oplus T\Ff $. There is a dual splitting $T^\ast S\cong T^\ast\Ff \oplus N^\ast\Ff $ and there are identifications
$
N\Ff \cong G$, $T^\ast\Ff \cong G^0
$.
Furthermore the induced splitting of $0\rightarrow N^\ast\Ff \otimes \ell \rightarrow T^\ast S\otimes \ell\rightarrow T^\ast\Ff \otimes \ell \rightarrow 0$ lifts to a splitting of $0\rightarrow J^1_\bot \ell \rightarrow J^1 \ell\rightarrow T^\ast\Ff \otimes \ell \rightarrow 0$. Hence $J^1\ell\cong J^1_\bot \ell\oplus (T^\ast\Ff \otimes \ell)$. Let $F\in\Gamma( \wedge^2 G^\ast \otimes TS/G )$ be the curvature form of $G$. The curvature $F$ will be also understood as an element $F\in\Gamma(\wedge^2 N^\ast\Ff \otimes T\Ff )\subset\Gamma(\wedge^2(J^1_\bot \ell\otimes \ell^\ast)\otimes T\Ff )$, where we used the embedding $N^\ast\Ff \otimes \ell \INTO J^1_\bot \ell$.

Let $d_G:C^\infty(S)\rightarrow\Gamma(N^\ast\Ff )$ be the composition of the de Rham differential $d:C^\infty(S)\rightarrow\Omega^1(S)$ followed by the projection $\Omega^1(S)\rightarrow\Gamma(N^\ast\Ff )$ determined by the decomposition $T^\ast S=T^\ast\Ff \oplus N^\ast\Ff $. Then $d_G$ is a $\Gamma (N^\ast \Ff )$-valued derivation of $C^\infty (S)$ and will be interpreted as the ``transversal de Rham differential''.

\begin{proposition}\label{prop:eps}
	There exists a unique degree zero, graded $\mathbb R$-linear map $\varepsilon:\Omega(\Ff )\rightarrow\Omega(\Ff ,N^\ast\Ff )$ such that
	\begin{enumerate}
		\item $\varepsilon|_{C^\infty(S)}=d_G$,
		\item $[\varepsilon,d_{\Ff }]=0$, and
		\item the following identity holds
		\begin{equation*}
			\varepsilon(\tau \wedge \tau')=\tau \wedge \varepsilon(\tau') + (-)^{|\tau ||\tau' |} \tau' \wedge \varepsilon(\tau),
		\end{equation*}
		for all homogeneous $\tau, \tau' \in \Omega(\Ff )$.
	\end{enumerate}
\end{proposition}

In order to prove Proposition \ref{prop:eps}, the following Lemma will be useful:
\begin{lemma}\label{lem:eps}
	Let $f $ be a leaf-wise constant local function on $S$, i.e.~$d_\Ff f = 0$, then $d_\Ff d_G f = 0$ as well.
\end{lemma}
\begin{proof}
	Let $f$ be as in the statement. First of all, note that $d f$ takes values in $N^\ast \Ff$, so that $d_G f = df$. Now recall that $d_\Ff d_G f = 0$ iff
	$0 = \langle d_\Ff d_G f, X \rangle = \nabla_X d_G f\! =\! \Ll_X d_G f $
	for all $X \!\in\! \Gamma (T \Ff)$, where $\nabla$ is the canonical $T \Ff$-connection in $N^\ast \Ff$. But
	$
	\Ll_X d_G f = \Ll_X df = d (X f) = 0
	$.
	This completes the proof.
\end{proof}

\begin{proof}[Proof of Proposition {\ref{prop:eps}}]
	The graded algebra $\Omega (\Ff)$ is generated in degree $0$ and $1$. In order to define $\eps$, we first define it on the degree one piece $\Omega^1 (\Ff)$ of $\Omega (\Ff)$. Thus, note that $\Omega^1 (\Ff)$ is generated, as a $C^\infty (S)$-module, by leaf-wise de Rham differentials $d_\Ff f \in \Omega^1 (\Ff)$ of functions $f \in C^\infty (S)$. The only relations among these generators are the following
	\begin{equation} \label{eq:rel}
		\begin{aligned}
			d_\Ff (f+g) & = d_\Ff f + d_\Ff g, \\
			d_\Ff (fg) &= f d_\Ff g + g d_\Ff f , \\
			d_\Ff f & = 0 \text{ on every open domain where $f$ is leaf-wise constant},
		\end{aligned}
	\end{equation}
	where $f,g \in C^\infty (S)$. Now define $\eps : \Omega^1 (\Ff) \to \Omega^1 (\Ff, N^\ast \Ff)$ on generators by putting
	\[
	\eps f := d_G f \quad \text{and} \quad \eps d_\Ff f := d_\Ff d_G f,
	\]
	and extend it to the whole $\Omega^1 (\mathcal{F})$ by prescribing $\R$-linearity and the following Leibniz rule:
	\begin{equation}\label{eq:leib_dG}
		\eps (f \sigma) = f \eps (\sigma) + \sigma \otimes d_G f ,
	\end{equation}
	for all $f \in C^\infty (S)$, and $\sigma \in \Omega^1 (\mathcal{F})$. In order to see that $\eps$ is well defined it suffices to check that it preserves relations (\ref{eq:rel}). Compatibility with the first two relations can be checked by a straightforward computation that we omit. Compatibility with the third relation immediately follows from Lemma \ref{lem:eps}.
	Finally, in view of the Leibniz rule (\ref{eq:leib_dG}), $d_G$ and $\eps$ combine and extend to a well-defined derivation $\Omega (\Ff) \to \Omega (\Ff, N^\ast \Ff)$. By construction, the extension satisfies all required properties. Uniqueness is obvious.
\end{proof}

The graded differential operator $\varepsilon$ will be also denoted by $d_G$.

Similarly, there is a ``transversal version of the first jet prolongation $j^1$''. Namely, let $j^1_G:\Gamma(\ell)\rightarrow\Gamma(J^1_\bot \ell)$ be the composition of the first jet prolongation $j^1:\Gamma(\ell)\rightarrow\Gamma(J^1\ell)$ followed by the projection $\Gamma(J^1\ell)\rightarrow\Gamma(J^1_\bot \ell)$ determined by the decomposition $J^1\ell=J^1_\bot \ell \oplus (N^\ast\Ff \otimes \ell )$.
Then $j^1_G$ is a first order differential operator from $\Gamma(\ell)$ to $\Gamma(J^1_\bot \ell)$ such that
\begin{equation}\label{eq:leib_j1G}
	j^1_G(f\lambda)=fj^1_G\lambda+(d_Gf)\otimes\lambda,
\end{equation}
$\lambda\in\Gamma(\ell)$ and $f\in C^\infty(S)$, where, similarly as above, we understood the embedding $N^\ast\Ff \otimes \ell \INTO J^1_\bot \ell$. As announced, the operator $j^1_G$ will be interpreted as the ``transversal first jet prolongation''.

\begin{proposition}\label{prop:delta}
	There exists a unique degree zero graded $\mathbb R$-linear map $\delta:\Omega(\Ff ,\ell)\rightarrow\Omega(\Ff ,J^1_\bot \ell)$ such that
	\begin{enumerate}
		\item $\delta|_{\Gamma(\ell)}=j^1_G$,
		\item $[\delta,d_{\Ff }]=0$, and
		\item the following identity holds
		\begin{equation*}
			\delta(\tau \wedge \Omega)=\tau \wedge \delta(\omega)+d_G \tau \otimes\omega,
		\end{equation*}
		for all $\tau\in\Omega(\Ff )$, and $\omega\in\Omega(\Ff ,\ell)$, where the tensor product is over $\Omega(\Ff )$, and we understood both the isomorphism
		\begin{equation}\label{eq:isoOmega}
			\Omega(\Ff ,N^\ast\Ff )\underset{\Omega(\Ff )}{\otimes}\Omega(\Ff ,\ell)\cong\Omega(\Ff ,N^\ast\Ff \otimes \ell)
		\end{equation}
		and the embedding $N^\ast\Ff \otimes \ell \INTO J_\bot^1 \ell$.
	\end{enumerate}
\end{proposition}

In order to prove Proposition \ref{prop:delta}, the following Lemma will be useful:
\begin{lemma}\label{lem:delta}
	Let $\mu$ be a leaf-wise constant local section of $\ell$, i.e.~$d_\Ff \mu = 0$, then $d_\Ff j^1_G \mu = 0$ as well.
\end{lemma}
\begin{proof}
	Let $\mu$ be as in the statement. First of all note that, by the very definition of $J^1_\bot \ell$, $j^1 \mu$ takes values in $J^1_\bot \ell$ so that $j^1_G \mu = j^1 \mu$. Now recall that $d_\Ff j^1_G \mu = 0$ iff
	$0 = \langle d_\Ff j^1_G \mu, X \rangle = \nabla_X j^1_G \mu $
	for all $X \in \Gamma (T \Ff)$, where $\nabla$ is the canonical $T \Ff$-connection in $J^1_\bot \ell$. But
	$
	\nabla_X j^1_G \mu = \nabla_X j^1 \mu = j^1 \nabla_X \mu = 0
	$,
	where we used (\ref{eq:nabla_J1}). This completes the proof.
\end{proof}

\begin{proof}[Proof of Proposition {\ref{prop:delta}}]
	In this proof a tensor product $\otimes$ will be over $C^\infty (S)$ unless otherwise stated. We can regard $\Omega (\Ff , \ell) = \Omega (\Ff ) \otimes \Gamma (\ell)$ as a quotient of $\Omega (\Ff ) \otimes_\R \Gamma (\ell)$ in the obvious way. Our strategy is defining an operator $\delta ': \Omega (\Ff ) \otimes_\R \Gamma (\ell) \to \Omega (\Ff, J^1_\bot \ell)$ and prove that it descends to an operator $\delta : \Omega(\Ff ,\ell)\rightarrow\Omega(\Ff ,J^1_\bot \ell)$ with the required properties. Thus, for $\sigma \in \Omega (\Ff)$ and $\lambda \in \Gamma (\ell)$ put
	\begin{equation}\label{eq:delta'}
		\delta ' (\sigma \otimes_\R \lambda):= \sigma \otimes j^1_G \lambda+ d_G \sigma\otimes_{\Omega (\Ff)} \lambda \in \Omega (\Ff, J^1_\bot \ell),
	\end{equation}
	where, in the second summand, we understood both the isomorphism (\ref{eq:isoOmega}) and the embedding $N^\ast \Ff \otimes \ell \INTO J^1_\bot \ell$ (just as in the statement of the proposition). In order to prove that $\delta '$ descends to an operator $\delta$ on $\Omega (\Ff, \ell)$ it suffices to check that $\delta ' (f \sigma \otimes_\R \lambda) = \delta ' (\sigma \otimes_\R f \lambda)$ for all $\sigma, \lambda$ as above, and all $f \in C^\infty (S)$. This can be easily obtained using the derivation property of $d_G$ and (\ref{eq:leib_j1G}). Now, Properties 1) and 3) immediately follows from (\ref{eq:delta'}). In order to prove Property 2), it suffices to check that $\delta d_\Ff \lambda = d_\Ff j^1_G \lambda$ for all $\lambda \in \Gamma (\ell)$ (and then use Property 3)). It is enough to work locally. Thus, let $\mu$ be a local generator of $\Gamma (\ell)$ with the further property that $d_\Ff \mu = 0$. Moreover, let $f \in C^\infty (S)$, and compute
	\begin{align*}
		\delta d_\Ff ( f \mu ) & = \delta (d_\Ff f\! \otimes\! \mu) = d_\Ff f\! \otimes\! j^1_G \mu + d_G d_\Ff f\! \otimes\! \mu = d_\Ff f\! \otimes\!j^1_G \mu + d_\Ff d_G f\! \otimes\! \mu \\
		&= d_\Ff ( f j^1_G \mu + d_G f\! \otimes\! \mu) = d_\Ff (j^1_G f \mu),
	\end{align*}
	where we used $d_\Ff \mu = 0$, Proposition \ref{prop:eps}, Lemma \ref{lem:delta}, and (\ref{eq:leib_j1G}). Uniqueness of $\delta$ is obvious.
\end{proof}

The graded differential operator $\delta$ will be also denoted by $j^1_G$.

Now, interpret $J{}_\bot \in\Gamma(\wedge^2(J^1_\bot \ell)^\ast\otimes \ell)$ as a section $\# \in\Gamma((J^1_\bot \ell\otimes \ell^\ast)^\ast\otimes (J^1_\bot \ell)^\ast)$.
The interior product of $\# $ and $F\in\Gamma(\wedge^2(J^1_\bot \ell\otimes \ell^\ast)\otimes T\Ff )$ is a section $F^\#\in\Gamma(\operatorname{End}(J^1_\bot \ell)\otimes T\Ff \otimes \ell^\ast)$.
For any $\mu \in\Omega^{k+1}(\Ff ,\ell)$, the interior product of $F^\#$ and $\mu$ is a section $i_{F^\#}\mu\in\Omega^{k}(\Ff ,\operatorname{End}J^1_\bot \ell)$. Now, we extend
\begin{enumerate}
	\item the bi-linear map $J_{\bot} : \wedge^2 J^1_\bot \ell \to \ell $ to a degree $+1$, $\Omega (\Ff)$-bilinear, symmetric form
	\[
	\langle-,-\rangle_{C}:\Omega(\Ff , J^1_\bot \ell)[1]\times\Omega(\Ff , J^1_\bot \ell)[1]\longrightarrow \Omega(\Ff , \ell)[1]
	\]
	\item the natural bilinear map $\circ : \operatorname{End} J^1_\bot \ell \otimes  \operatorname{End} J^1_\bot \ell \to\operatorname{End} J^1_\bot \ell $ to a degree $+1$, $\Omega (\Ff)$-bilinear map
	\[
	\Omega (\Ff, \operatorname{End} J^1_\bot \ell )[1] \times \Omega (\Ff, \operatorname{End} J^1_\bot \ell )[1] \longrightarrow \Omega (\Ff, \operatorname{End} J^1_\bot \ell )[1],
	\] also denoted by $\circ$, and
	\item the tautological action $\operatorname{End} J^1_\bot \ell \otimes   J^1_\bot \ell \to J^1_\bot \ell $ to a degree $+1$, $\Omega (\Ff)$-linear action
	\[
	\Omega (\Ff, \operatorname{End} J^1_\bot \ell )[1] \times \Omega (\Ff, J^1_\bot \ell )[1] \longrightarrow \Omega (\Ff, J^1_\bot \ell )[1].
	\]
\end{enumerate}

\begin{theorem}\label{prop:multi}
	The first (unary) bracket in the $L_\infty$-algebra structure on $\Omega(\Ff ,\ell)[1] $ is $d_{\Ff }$. Moreover, for $k > 1$, the $k$-th multi-bracket is given by
	\begin{equation}\label{multi}
		\mathfrak{m}_k (\nu_1,\dots,\nu_k) = \frac{1}{2}\sum_{\sigma\in S_k}\epsilon(\sigma, \boldsymbol{\nu})\left\langle j^1_G\nu_{\sigma(1)},(i_{F^\#}\nu_{\sigma(2)}\circ\cdots\circ i_{F^\#}\nu_{\sigma(k-1)})j^1_G\nu_{\sigma(k)}\right\rangle_{C},
	\end{equation}
	for all homogeneous $\nu_1\,\dots,\nu_k\in\Omega(\Ff ,\ell)[1] $, where $\epsilon (\sigma, \boldsymbol{\mu})$ is the Koszul sign prescribed by the permutations of the $\mu$'s.
\end{theorem}

\begin{proof}
	See Appendix \ref{sec:OP_formula}.
\end{proof}

\begin{remark}
	The explicit form of the contact thickening (see Subsection~\ref{sec:contact_thickening}) shows that the Jacobi bracket is actually fiber-wise entire. In particular Corollaries \ref{cor:conver} and \ref{cor:hequi} always apply to the contact case.
\end{remark}

\section{An example}
\label{subsec:}

In \cite{Z2008}, Zambon presents an example of a coisotropic submanifold $S_0$ in a symplectic manifold whose coisotropic deformation problem is obstructed. Zambon's example was also considered by Oh and Park in \cite{OP2005}, and in the latter paper the obstruction is discussed in terms of the $L_\infty$-algebra of $S_0$. More recently the same example was reconsidered by L\^e and Oh in \cite{LO2012}, where it is proved that $S_0$ is also obstructed when seen as a coisotropic submanifold in a l.c.s. manifold. There is a contact analogue of Zambon's example, discussed in some details in \cite{T2017} (see also \cite{T2016}). Here, we describe  another example of a regular coisotropic submanifold $S$ in a contact manifold whose coisotropic deformation problem is formally obstructed. Unlike the example in \cite[Section 4.8]{T2017}, $S$ has a non-simple characteristic foliation. From this point of view, this section is closely inspired by \cite[Section 12]{OP2005} (symplectic case, see also \cite{K2010}). Actually, the $S$ in this section can be guessed from that in \cite[Section 12]{OP2005} via ``contactization''. Nonetheless the contact and the symplectic cases seem to be independent: seemingly no result about the one could be found from the other.

Consider the $7$-dimensional coorientable contact manifold $(M,C)$, with $M:=\R^6\times\bbS^1$ and $C:=\ker\theta$, where the global contact $1$-form $\theta\in\Omega^1(M)$ is given by
\begin{equation*}
	\theta:=d\phi-\sum_{i=1}^3p_idq^i.
\end{equation*}
Here $(q^i,p_i)$ are the Cartesian coordinates on $\R^6\cong T^\ast\R^3$ and $\phi$ is the angle coordinate on $\bbS^1$.
We will also use polar coordinates $(r_i,\phi_i)$ on each plane $\R^2=\{(q^i,p_i)\}$, $i=1,2,3$.

The contact distribution $C$ possesses a global frame given by
\begin{equation*}
	\frac{\partial}{\partial p_i},\qquad D_i:=\frac{\partial}{\partial q^i}+p_i\frac{\partial}{\partial\phi},
\end{equation*}
and, for $f\in\Gamma(\R_M)= C^\infty(M)$, the corresponding contact vector field $X_f$ is given by
\begin{equation*}
	X_f=D_if\frac{\partial}{\partial p_i}-\sum_{i=1}^3\frac{\partial f}{\partial p_i}D_i+f\frac{\partial}{\partial\phi}.
\end{equation*}
In particular, $\partial/\partial\phi$ is the Reeb vector field $X_1$.
As we know, there is an induced Jacobi bracket $J \equiv \{-,-\}$ on the trivial line bundle $\R_M\to M$.
It is straightforward to check that
\begin{equation*}
	J=D_i\wedge\frac{\partial}{\partial p_i}+\id\wedge\frac{\partial}{\partial\phi}.
\end{equation*}

Take the functions $H_i :=\frac{1}{2}r_i^2 \in C^\infty(M)$, $i=1,2,3$.
For every positive real number $\alpha>0$, put $H_{(\alpha)}:=H_1+\alpha H_2$, and define the $5$-dimensional submanifold $S_\alpha\subset M$ by putting
\begin{equation*}
	S_\alpha:=H_{(\alpha)}^{-1}\left(1/4\right)\cap H_3^{-1}\left(1/2\right).
\end{equation*}
Since $\{H_{(\alpha)},H_3\}=0$, and $\theta$, $dH_{(\alpha)}$, $dH_3$, are linearly independent on a neighborhood of $S_\alpha$,
from~Proposition~\ref{prop:coisotropics_contact_setting} we get that $S_\alpha$ is a regular coiso\-tropic submanifold of $(M,C)$.
Hence, it inherits the structure of a pre-contact manifold, with pre-contact distribution $C_\alpha:=C\cap TS_\alpha$, i.e.~$C_\alpha$ is the kernel of the global pre-contact form $\theta_\alpha:=\theta|_{TS_\alpha}\in\Omega^1(S_\alpha)$.
Moreover its characteristic distribution $T\Ff$ possesses a global frame consisting of $X_{H_{(\alpha)}-1/4}|_{S_\alpha}$ and $X_{H_3-1/2}|_{S_\alpha}$.
In particular, all characteristic leaves of $(S_\alpha, C_\alpha)$ are orientable.

\begin{remark}
	For $\alpha=1$, the characteristic foliation $\Ff$ is simple, and its leaf space is diffeomorphic to $\mathbb{CP}^1\times\bbS^1$.
	On the other hand, for $\alpha\neq 1$, $\Ff$ is not simple.
	Specifically, for $\alpha\notin\Q$, every characteristic leaf contained in $S_\alpha\cap H_1^{-1}(]0,1/4[)$ is dense in $S_\alpha$.
	Finally, for $\alpha=m/n$, with $m$ and $n$ coprime integers, there are characteristic leaves with non-trivial holonomy: characteristic leaves contained in $S_\alpha\cap H_1^{-1}(0)$ (resp.~$S_\alpha\cap H_1^{-1}(1/4)$) have cyclic holonomy group of order $m$ (resp.~$n$).
\end{remark}

Put $U_\alpha:=S_\alpha\cap H_1^{-1}(]0,1/4[)$.
Then $U_\alpha$ is an open and dense subset of $S_\alpha$, covered by charts with local coordinates $(u_1,u_2,x,y,z)$ defined by
\begin{gather*}
	u_1=\phi_3,\quad 
	u_2=\phi_1+\alpha\phi_2,\\ 
	x=H_2,\quad
	y=\phi_2-\alpha\phi_1,\quad 
	z=\phi+\sum_{i=1}^3 H_i\left(\phi_i-\frac{1}{2}\sin(2\phi_i)\right).
\end{gather*}
The latter are actually (local) Darboux coordinates on $S_\alpha$, i.e.~locally $\theta_\alpha=dz-ydx$.
So, locally, we also have
\begin{equation}
	\label{eq:second_obstructed_example_Darboux}
	C_\alpha=\left\langle\frac{\partial}{\partial u_1},\frac{\partial}{\partial u_2},\frac{\partial}{\partial y},D:=\frac{\partial}{\partial x}+y\frac{\partial}{\partial z}\right\rangle,\quad T\Ff=\left\langle\frac{\partial}{\partial u_1},\frac{\partial}{\partial u_2}\right\rangle.
\end{equation}

Note that the vector fields $\frac{\partial}{\partial u_1}$, $\frac{\partial}{\partial u_2}$, $\frac{\partial}{\partial y}$, $D$, $\frac{\partial}{\partial z}$ do not depend on the Darboux chart, and are globally defined on $U_\alpha$.
Moreover, the vector fields $\frac{\partial}{\partial u_1}$ and $\frac{\partial}{\partial u_2}$ (resp.~leaf-wise differential forms $d_\Ff u_1\equiv(du_1)|_{T\Ff}$ and $d_\Ff u_2\equiv(du_2)|_{T\Ff}$) uniquely prolong to a global frame of $T\Ff$ (resp.~$T^\ast\Ff$). 
Hence, for any $0<\eps<1/8$, we can pick a distribution $G$ on $S_\alpha$ complementary to $T\Ff$ and satisfying the following additional property
\begin{equation}
	\label{eq:second_obstructed_example_complementary}
	\left.G\right|_{U_{\alpha,\eps}}=\left.\left\langle\frac{\partial}{\partial y},\ D,\ \frac{\partial}{\partial z}\right\rangle\right|_{U_{\alpha,\eps}},
\end{equation}
where $U_{\alpha,\eps}\subset U_\alpha$ is the open subset defined by $U_{\alpha,\eps}:=S_\alpha\cap H_1^{-1}(]\eps,1/4-\eps[)$.
From now on we assume we have fixed such a distribution $G$.
After this choice:
\begin{itemize}
	\item around $S_\alpha$, $(M,C)$ identifies with the contact thickening of $(S_\alpha,C_\alpha)$ determined by the splitting $TS_\alpha=T\Ff\oplus G$ 
	(see Section~\ref{sec:contact_thickening}),
	\item the $L_\infty$-algebra of $S_\alpha$ is given by $(\Omega^\bullet(\Ff),\{\mathfrak{m}_k\})$ with the multibrackets determined by $G$ as in Theorem~\ref{prop:multi}.
\end{itemize}

Focus on the explicit expressions of $\mathfrak{m}_1$ and $\mathfrak{m}_2$.
From coorientability, $\mathfrak{m}_1:\Omega^\bullet(\Ff)\to\Omega^\bullet(\Ff)$ boils down to the leaf-wise de Rham differential $d_\Ff : \Omega^\bullet(\Ff) \to \Omega^\bullet (\Ff)$.
Hence, for $f,g\in C^\infty(S_\alpha)$, the following identities hold:
\begin{equation}
	\label{eq:second_obstructed_example_contact2}
	\begin{gathered}
		\mathfrak{m}_1(f)=\frac{\partial f}{\partial u_1}d_\Ff u_1+\frac{\partial f}{\partial u_2}d_\Ff u_2,\\
		\mathfrak{m}_1(fd_\Ff u_1+gd_\Ff u_2)=\left(\frac{\partial g}{\partial u_1}-\frac{\partial f}{\partial u_2}\right)d_\Ff u_1\wedge d_\Ff u_2.
	\end{gathered}
\end{equation}
Let 
\[
J_\alpha\equiv\{-,-\}_\alpha:C^\infty(U_\alpha)\times C^\infty(U_\alpha)\to C^\infty(U_\alpha)
\]
be the bi-differential operator defined by
\begin{equation*}
	J_\alpha=D\wedge\frac{\partial}{\partial y}+\id\wedge\frac{\partial}{\partial z}.
\end{equation*}
From \eqref{eq:second_obstructed_example_Darboux},  \eqref{eq:second_obstructed_example_complementary}, and Theorem~\ref{prop:multi} 
we get that
\begin{equation}
	\label{eq:second_obstructed_example_contact4}
	\begin{aligned}
		&\mathfrak{m}_2(f,g)=-\{f,g\}_\alpha,\\
		&\mathfrak{m}_2(f,g_1d_\Ff u_1+g_2d_\Ff u_2)=-\{f,g_1\}_\alpha d_\Ff u_1-\{f,g_2\}_\alpha d_\Ff u_2,\\
		&\mathfrak{m}_2(f_1d_\Ff u_1+f_2d_\Ff u_2,g_1d_\Ff u_1+g_2d_\Ff u_2)\\
		&\qquad\qquad\quad =\left(\{f_1,g_2\}_\alpha-\{f_2,g_1\}_\alpha\right)d_\Ff u_1\wedge d_\Ff u_2,
	\end{aligned}
\end{equation}
on $U_{\alpha,\eps}$.

We can extract from~\eqref{eq:second_obstructed_example_contact2} and~\eqref{eq:second_obstructed_example_contact4} information about the coisotropic deformation problem of $S_\alpha$.
Take $s=fd_\Ff u_1+gd_\Ff u_2\in\Omega^1(\Ff)$. From Corollary~\ref{cor:inf1}, it is an infinitesimal coisotropic deformation if and only if 
\begin{equation}
	\label{eq:second_obstructed_example_contact_deformations}
	\frac{\partial g}{\partial u_1}-\frac{\partial f}{\partial u_2}=0.
\end{equation}
Additionally, from Corollary~\ref{cor:infequi},
two infinitesimal coisotropic deformations $s_i=f_id_\Ff u_1+g_id_\Ff u_2$, with $i=0,1$, are infinitesimally Hamiltonian equivalent if and only if there exists $h\in C^\infty(S_\alpha)$ such that
\begin{equation*}
	f_1=f_0+\frac{\partial h}{\partial u_1},\quad g_1=g_0+\frac{\partial h}{\partial u_2}.
\end{equation*}
Let $s=fd_\Ff u_1+gd_\Ff u_2$ be an infinitesimal coisotropic deformation, with $\operatorname{supp}(s)\subset U_\alpha$. Assume that $s$ can be prolonged to a formal coisotropic deformation.
Since $\eps$ can be chosen arbitrarily small,
from Proposition~\ref{prop:Kuranishi}, there exist $h,k\in C^\infty(S_\alpha)$ such that
\begin{equation}
	\label{eq:second_obstructed_example_contact5}
	f\frac{\partial g}{\partial z}-g\frac{\partial f}{\partial z}+(Df)\frac{\partial g}{\partial y}-(Dg)\frac{\partial f}{\partial y}
	=\frac{\partial k}{\partial u_1}-\frac{\partial h}{\partial u_2}.
\end{equation}
Integrating~\eqref{eq:second_obstructed_example_contact5} over a compact characteristic leaf $\Ll$, we get the following (weaker) necessary condition for the formal prolongability of $s$
\begin{equation}
	\label{eq:second_obstructed_example_contact6}
	\iint\limits_{\Ll}\left(f\frac{\partial g}{\partial z}-g\frac{\partial f}{\partial z}+(Df)\frac{\partial g}{\partial y}-(Dg)\frac{\partial f}{\partial y}\right)d_\Ff u_1 d_\Ff u_2=0.
\end{equation}

\begin{proposition}
	\label{prop:second_obstructed_example_contact}
	If $\alpha\in\Q$, then the coisotropic submanifold $S_\alpha$ of $(M,C)$ is formally obstructed.
\end{proposition}

\begin{proof}
	Let $\alpha=\frac{m}{n}$, with $m$ and $n$ coprime integers.
	In this case the characteristic foliation $\Ff_\alpha$ has orientable compact leaves.
	Pick two non-constant functions $\chi\in C^\infty(\bbS^1)$ and $\rho\in C^\infty(\R)$ such that $\operatorname{supp}(\rho) \subset \left]0,1/4\alpha\right[$.
	Then there exist two functions $f,g\in C^\infty(S_\alpha)$ uniquely determined by
	\begin{equation}
		\label{eq:prop:second_obstructed_example_contact}
		f(u_1,u_2,x,y,z)=\rho(x),\qquad g(u_1,u_2,x,y,z)=\rho(x)\chi(ny).
	\end{equation}
	Put $s:=fd_\Ff u_1+gd_\Ff u_2\in\Omega^1(\Ff)$.
	The latter is an infinitesimal coisotropic deformation of $S_\alpha$ which is formally obstructed.
	Indeed $s$ fulfills \eqref{eq:second_obstructed_example_contact_deformations}, but it fails to fulfill the constraint \eqref{eq:second_obstructed_example_contact6}:
	\begin{equation*}
		\iint\limits_{\Ll(\bar x,\bar y,\bar z)}\left(f\frac{\partial g}{\partial z}-g\frac{\partial f}{\partial z}+(Df)\frac{\partial g}{\partial y}-(Dg)\frac{\partial f}{\partial y}\right)d_\Ff u_1 d_\Ff u_2 =\tfrac{m^2+n^2}{n}(2\pi)^2\rho(\bar x)\rho'(\bar x)\chi'(n\bar y)\neq 0,
	\end{equation*}
	where, for any $(\bar x,\bar y,\bar z)$, we denoted by $\Ll(\bar x,\bar y,\bar z)$ the characteristic leaf given by the level set $x=\bar x, y=\bar y, z=\bar z$.
\end{proof}

\begin{remark}
	The case $\alpha \notin \Q$ is more involved. In particular, it requires a better understanding of the characteristic foliation of $(S_\alpha, C_\alpha)$. We hope to discuss it in details elsewhere.
\end{remark}

\appendix

\section{Derivations, infinitesimal automorphisms of vector bundles and the Schouten--Jacobi algebra}\label{sec:app_0}

Let $M$ be a smooth manifold, and let $E \to M$ be a vector bundle over $M$. A first order differential operator $\Delta : \Gamma (E) \to \Gamma (E)$  is a derivation of $E$ if there exists a (necessarily unique) vector field $X$ such that
$
\Delta (fe) = X(f) e + f\Delta e
$
for all $f \in C^\infty (M)$, and $e \in \Gamma (E)$. In this case we write $\sigma (\Delta) = X$, and call it the \emph{symbol} of $\Delta$.
The space of derivations of $E$ will be denoted by $\Der E$.
It is the space of sections of a (transitive) Lie algebroid $\der E \to M$ over $M$, sometimes called the \emph{gauge algebroid of} $E$, whose Lie bracket is the commutator of derivations, and whose anchor is the symbol $\sigma : \der E \to TM$ (see, e.g., \cite[Theorem 1.4]{KM2002} for details). The fiber $\der_x E$ of $\der E$ through $x \in M$ consists of $\R$-linear maps $\delta : \Gamma (E) \to E_x$ such that there exists a, necessarily unique, tangent vector $v \in T_x M$, called the \emph{symbol of $\delta$} and also denoted by $\sigma (\delta)$, satisfying the obvious Leibniz rule
$
\delta (f e) = v (f) e (x) + f (x) \delta (e),
$
for all $f \in C^\infty (M)$ and $e \in \Gamma (E)$.

\begin{remark}
	If $E$ is a line bundle, then every first order differential operator $\Gamma (E) \to \Gamma (E)$ is a derivation of $E$. Consider the trivial line bundle $\R_M := M \times \R$. Then $\Gamma (\R_M) = C^\infty (M)$. First order differential operators $\Gamma (\R_M) \to \Gamma (\R_M)$ or, equivalently, derivations of $\R_M$, are the operators of the form $X + a : C^\infty (M) \to C^\infty (M)$, where $X$ is a vector field on $M$ and $a \in C^\infty (M)$ is interpreted as an operator (multiplication by $a$). Accordingly, in this case, there is a natural direct sum decomposition $ \Der\R_M = \mathfrak{X} (M) \oplus C^\infty (M)$, the projection $ \Der\R_M \to C^\infty (M)$ being given by $\Delta \mapsto \Delta 1$.
\end{remark}

The construction of the gauge algebroid of a vector bundle is functorial, in the following sense. Let $\phi : E \to F$ be a morphism of vector bundles $E \to M$, $F \to N$, over a smooth map $\underline{\phi} : M \to N$. We assume that $\phi$ is \emph{regular}, in the sense that it is an isomorphism when restricted to fibers. In particular a section $f$ of $F$ can be pulled-back to a section $\phi^\ast f$ of $E$, defined by $(\phi^\ast f)(x) := (\phi|_{E_x}^{-1} \circ f \circ \underline{\phi}) (x)$, for all $x \in M$. Then $\phi$ induces a morphism of Lie algebroids $\der \phi : \der E \to \der F$ defined by
\[
\der \phi (\delta) f := \phi(\delta (\phi^\ast f)), \quad \delta \in \der E, \quad f \in \Gamma (F).
\]
We also denote $\phi_\ast := \der \phi$.

Derivations of a vector bundle $E$ can be also understood as infinitesimal automorphisms of $E$ as follows. First of all, a derivation $\Delta$ of $E$ determines a derivation $\Delta^\ast$ of the dual bundle $E^\ast$, with the same symbol as $\Delta$. Derivation $\Delta^\ast$ is defined by
$
\Delta^\ast \varphi := \sigma (\Delta) \circ \varphi - \varphi \circ \Delta
$,
where $\varphi : \Gamma (E) \to C^\infty (M)$ is a $C^\infty (M)$-linear form, i.e.~a section of $E^\ast$. Now, recall that an automorphism of $E$ is a regular morphism $\phi : E \to E$ covering a diffeomorphism $\underline{\phi} : M \to M$. An \emph{infinitesimal automorphism} of $E$ is a vector field $Y$ on $E$ whose flow consists of (local) automorphisms. In particular, $Y$ projects onto a (unique) vector field $\underline{Y} \in \mathfrak{X} (M)$. Note that one parameter families of infinitesimal automorphisms generate one parameter families of automorphisms and vice-versa, any one parameter family of automorphisms is generated by a one parameter family of infinitesimal automorphisms. Infinitesimal automorphisms of $E$ are sections of a (transitive) Lie algebroid over $M$, whose Lie bracket is the commutator of vector fields on $E$, and whose anchor is $Y \mapsto \underline{Y}$. It can be proved that a vector field $Y$ on $E$ is an infinitesimal automorphism if and only if it preserves fiber-wise linear functions on $E$, i.e.~sections of the dual bundle $E^\ast$. Finally, note that the restriction of an infinitesimal automorphism to fiber-wise linear functions $Y|_{\Gamma (E^\ast)} : \Gamma (E^\ast) \to \Gamma (E^\ast)$ is a derivation of $E^\ast$, and the correspondence $Y \mapsto Y|^\ast_{\Gamma (E^\ast)}$ is a well-defined isomorphism between the Lie algebroid of infinitesimal automorphisms and the gauge algebroid of $E$.

If $\Delta$ is a derivation of $E$, $Y$ is the corresponding infinitesimal automorphism, and $\{ \phi_t \}$ is its flow, then we will also say that $\Delta$ \emph{generates the flow $\{ \phi_t\}$ by automorphisms}. We have
\[
\left. \frac{d}{dt}\right|_{t = 0} \phi_t^\ast e = \Delta e,
\]
for all $e \in \Gamma (E)$. Similarly, if $\{ \Delta_t \}$ is a smooth one parameter family of derivations of $E$, $\{ Y_t \}$ is the corresponding one parameter family of infinitesimal automorphisms, and $\{ \psi_t \}$ is the associated one parameter family of automorphisms, then we will say that $\{ \Delta_t \}$ \emph{generates} $\{ \psi_t \}$. We have
\[
\frac{d}{dt} \psi_t^\ast e = (\psi_t^\ast \circ \Delta_t) e .
\]

We now pass to multiderivations. We limit ourselves to the case when $E$ is a line bundle, and we denote it by $L$. First of all, notice that, in this case, $\der L \otimes L^\ast$ is the dual vector bundle to the first jet bundle $J^1 L \to M$ of $L$. In the paper we often adopt the following notation: $J_1L := \der L \otimes L^\ast$. The exterior algebra $\Gamma (\wedge^\bullet J_1 L)$ consists of alternating, first order multi-differential operators from $\Gamma(L)$ to $C^\infty(M)$, i.e.~$\R$-multi-linear maps which are first order differential operators on each entry separately. Let $\Delta \in \Gamma (\wedge^k J_1 L)$, and $\Delta^\prime \in \Gamma (\wedge^{k^\prime} J_1 L)$. If we interpret $\Delta$ and $\Delta^\prime$ as multi-differential operators, then their exterior product is given by
\begin{equation}
	\label{eq:multi-differential_operators}
	(\Delta\wedge\Delta^\prime)(\lambda_1,\dots,\lambda_{k+k^\prime}) = \sum_{\sigma\in S_{k,k^\prime}}(-)^\sigma\Delta(\lambda_{\sigma(1)},\dots,\lambda_{\sigma(k)})\Delta^\prime(\lambda_{\sigma(k+1)},\dots,\lambda_{\sigma(k+k^\prime)}),
\end{equation}
where $\lambda_1,\dots,\lambda_{k+k^\prime}\in\Gamma(L)$, and $S_{k, k^\prime}$ denotes $(k,k^\prime)$-unshuffles.
Similarly, $\Gamma (\wedge^\bullet J_1 L \otimes L)$ consists of alternating, first order multi-differential operators from $\Gamma(L)$ to itself. For this reason we often denote $\Der^\bullet L := \Gamma (\wedge^\bullet J_1 L \otimes L)$, where $\Der^0 L = \Gamma (L)$ and $\Der^1L=\Der L$. Note that $\Der^\bullet L$ does also identify with $L$-valued, skew-symmetric forms on $J^1 L$. We will often understand this identification.

We also consider the graded space $(\Der^\bullet L)[1]$ obtained from $\Der^\bullet L$ by shifting degrees by $1$. Beware that an element of $\Der^k L$ is a multi-differential operator with $k$-entries but its degree in $(\Der^\bullet L)[1]$ is $k-1$. There is a $\Gamma (\wedge^\bullet J_1 L)$-module structure on $(\Der^\bullet L)[1]$ given by the same formula (\ref{eq:multi-differential_operators}) as above.

\begin{remark}
	A Jacobi bracket $\{-,-\}$ on $L$ will be interpreted as an element of $\Der^2 L$. So, it corresponds to the associated bi-linear form $J : \wedge^2 J^1 L \to L$ via the identification $\Der^2 L = \Gamma (\operatorname{Hom}(\wedge^2 J^1 L, L))$. Accordingly, we will sometimes identify $\{-,-\}$ and $J$ (see Section \ref{sec:abstract_jac_mfd} for more details).
\end{remark}

The Lie bracket on $\Der^1 L = \Gamma (DL)$ and the tautological action of $DL$ on $L$ extend to a Lie bracket on $(\Der^\bullet L)[1]$. This Lie bracket is a ``Jacobi version'' of the Schouten bracket between multi-vector fields, therefore we call it the \emph{Schouten-Jacobi bracket} and denote it by $[-,-]^{SJ}$. It is defined by
\[
[ \square , \square^\prime ]^{SJ} := (-)^{k k^\prime}\square \circ \square^\prime - \square^\prime \circ \square,
\]
where $\square \in \Der^{k+1} L$, $\square^\prime \in \Der^{k^\prime +1} L$, and $\square \circ \square^\prime$ is given by the following ``\emph{Gerstenhaber formula}'':
\begin{equation*}
	(\square \circ \square^\prime) (\lambda_1, \dots, \lambda_{k+k^\prime +1}) = \sum_{\tau \in S_{k^\prime+1, k}}(-)^{\tau} \square (\square^\prime (\lambda_{\tau(1)},\dots, \lambda_{\tau(k^\prime+1)}) ,\lambda_{\tau (k^\prime+2)}, \dots, \lambda_{\tau (k+k^\prime +1)}),
\end{equation*}
where $\lambda_1,\dots, \lambda_{k+k^\prime +1} \in \Gamma (L)$.

The Schouten-Jacobi bracket satisfies the following \emph{Leibniz property}: there is an action by (graded) derivation $\square \mapsto X_\square$ of $((\Der^\bullet L)[1], [-,-]^{SJ})$ on the graded algebra $\Gamma (\wedge^\bullet J_1 L)$ such that
\begin{equation}
	\label{eq:genleib2}
	[\square, \Delta \wedge \square']^{SJ} = X_\square (\Delta) \wedge \square' + (-)^{|\square||\Delta|} \Delta \wedge [\square, \square']^{SJ},
\end{equation}

for all $\square, \square \in (\Der^\bullet L)[1]$ and all $\Delta \in \Gamma (\wedge^\bullet J_1 L)$. The action $\square \mapsto X_\square$ is defined as follows. For $\square \in \Der^{k+1} L$, the \emph{symbol} of $\square $, denoted by $\sigma_\square \in \Gamma (TM \otimes \wedge^k J_1 L )$, is, by definition, the $\wedge^k J_1 L$-valued vector field on $M$ implicitly defined by:
\[
\sigma_\square (f)(\lambda_1 , \dots, \lambda_{k}) \lambda := \square (f \lambda, \lambda_1,\dots,\lambda_{k}) - f \square ( \lambda, \lambda_1,\dots,\lambda_{k}),
\]
where $f\in C^\infty (M)$.
Finally, for any $\Delta \in \Gamma (\wedge^l J_1 L)$, and $\square \in \Der^{k+1} L$, the section $X_\square(\Delta)\in\Gamma (\wedge^{k+l} J_1 L)$ is given by
\begin{equation}
\begin{aligned}
	X_\square (\Delta) (\lambda_1, \dots, \lambda_{k+l}) :=& (-)^{k(l-1)}\sum_{\tau \in S_{l,k}}(-)^\tau\sigma_\square (\Delta (\lambda_{\tau(1)}, \dots, \lambda_{\tau (l)}))(\lambda_{\tau(l+1)}, \dots, \lambda_{\tau (k+l)}) \\
	&- \sum_{\tau\in S_{k+1,l-1}}(-)^\tau \Delta (\square (\lambda_{\tau(1)}, \dots, \lambda_{\tau(k+1)}), \lambda_{\tau(k+2)}, \dots, \lambda_{\tau (k+l)}).
\end{aligned}
\end{equation}

\begin{remark}
	Denote by $\mathfrak{X}^\bullet (M) = \Gamma (\wedge^\bullet TM)$ the Gerstenhaber algebra of (skew-symmetric) multi-vector fields on $M$. When $L = \R_M$, then $ \Der^k L = \Gamma (\wedge^k J_1 L)$. Moreover, there is a canonical direct sum decomposition $\Der^{k+1} L = \mathfrak{X}^{k+1} (M) \oplus \mathfrak{X}^k (M)$, where the projection $\Der^{k+1} L \to \mathfrak{X}^k (M)$ is given by $\square \mapsto \square ( 1, -, \dots, -)$. In particular, the Schouten--Jacobi bracket on $(\Der^\bullet L)[1]$ can be expressed in terms of the Schouten--Nijenhuis bracket on multi-vector fields (see \cite{GM2001} for more details).
\end{remark}

\section{The $L_\infty$-algebra of a pre-contact manifold}
\label{sec:OP_formula}

In this appendix we provide a coordinate proof of Theorem \ref{prop:multi}.

Let $(S, C_S)$ be a pre-contact manifold, with normal line bundle $\ell := TS / C_S$, and characteristic foliation $\Ff$, and let $G$ be a complementary distribution to $T  \Ff$, i.e., $T  S =  G \oplus T \Ff$. As shown in Subsection \ref{sec:contact_thickening}, the bundle $T_\ell^\ast \Ff := T^\ast \Ff \otimes \ell$ is equipped with an hyperplane distribution $C$ which is contact in a neighborhood of the zero section $\mathbf 0$: the contact thickening of $(S, C_S)$. Moreover $\mathbf 0$ is a coisotropic embedding. In particular, there is an $L_\infty$-algebra $(\Gamma (\wedge^\bullet N_\ell S \otimes \ell)[1], \{ {\mathfrak m}_k \})$ attached to $(S, C_S)$. In this case, $N S = T_\ell^\ast \Ff$, so that $\Gamma (\wedge^\bullet N_\ell S \otimes \ell) \cong \Omega (\Ff, \ell)$. In the following we will understand this isomorphism. We will show below that the multi-brackets ${\mathfrak m}_k$ are given by formula (\ref{multi}) which is the contact analogue of Oh-Park formula (see \cite[Formula (9.17)]{OP2005}). We will do this in local coordinates. From now on, we freely use notations and conventions from Subsections \ref{sec:contact_thickening}, \ref{sec:transversal_geometry} and \ref{sec:multi-brackets}.

Let $(x^i,u^a,z,p_i)$ be local coordinates on $T^\ast_\ell \Ff$ chosen as in the proof of Proposition \ref{prop:contact_thickening}. Distribution $G$ on $S$ is then locally spanned by vector fields of the form
\[
\G_a := \frac{\partial}{ \partial u^a} + G_a^i \frac{\partial}{\partial x^i},\quad \G = \frac{\partial}{\partial z} + G^i \frac{\partial}{\partial x^i},
\] and the structure and curvature forms of $C_S$ are locally
\[
\theta = (dz -C_a du^a) \otimes \mu, \quad  \omega=\frac{1}{2}\omega_{ab}du^a\wedge du^b.
\]
The matrix $(\omega_{ab})$ is invertible. Denote by $(\omega^{ab})$ its inverse. 
We also need the curvature form $F \in \Gamma (\wedge^2 N^\ast \Ff \otimes T\Ff)$ of $G$. It is locally given by
\[
F = \left( \frac{1}{2} F_{ab}^i du^a \wedge du^b + F^i_a du^a \wedge dz \right)  \otimes \frac{\partial}{\partial x^i},
\]
where 
\[
F^i_{ab}  := \G_a (G^i_b) - \G_b (G^i_a) \quad \text{and} \quad F^i_a = \G_a (G^i) - \G (G^i_a).
\]
It is easy to see that the structure form $\Theta$ of the contact distribution on the contact thickening is locally given by
\[
\Theta = \left[(1 - p_i G^i)dz - (C_a + p_i G^i_a)du^a + p_i dx^i \right] \otimes \mu,
\]
A long, but straightforward computation then shows that the bi-linear form $J \in \Gamma (\wedge^2 J_1L \otimes L)$ of the Jacobi stucture on the contact thickening is locally given by
\[
J = \left(\frac{1}{2} (\W_{\boldsymbol p}^{-1})^{\alpha \beta} \square_\alpha \wedge \square_\beta + \nabla^i \wedge \nabla_i \right) \otimes \mu,
\]
where $\W_{\boldsymbol p} := \W + p_i \F^i$, and
\[
\W := \left( 
\begin{array}{ccc}
0 & C_b & -1 \\
-C_a & \omega_{ab} & 0 \\
1 & 0 & 0
\end{array}
\right) \quad \text{and} \quad \F^i :=  \left( 
\begin{array}{ccc}
0 & 0 & 0 \\
0 & F^i_{ab} & F_a^i \\
0 & -F^i_b & 0
\end{array}
\right).
\]
Moreover $\nabla^i, \nabla_i \in \operatorname{Diff}_1 (L, \R_{T^\ast_\ell \Ff}) = \Gamma (J_1 L)$ are given by
\[
\nabla^i (f \mu) := \frac{\partial f}{\partial p_i} \quad \text{and} \quad \nabla_i (f \mu) = \frac{\partial f}{\partial x_i}.
\]
Finally, $\square_\alpha = \square, \square_a, \square_\circ \in \Gamma (J_1 L)$ with
\[
\begin{aligned}
\square & := \mu^\ast - p_i \nabla^i, \\
\square_a & := \nabla_a - p_j \frac{\partial G^{j}_a}{\partial x^i} \nabla^i + G^i_a \nabla_i, \\
\square_\circ & := \nabla - p_j \frac{\partial G^{j}}{\partial x^i} \nabla^i + G^i \nabla_i,
\end{aligned}
\]
where
\[
\mu^\ast (f \mu) := f, \quad \nabla_a (f \mu) := \frac{\partial f}{\partial u^a} \quad \text{and} \quad \nabla (f \mu) := \frac{\partial f}{\partial z}.
\]
Now, the ${\mathfrak m}_k$'s are graded first order multi-differential operators. In particular, they are completely determined by their action on elements in $\Omega (\Ff, \ell)$ of the form $f \mu$, $f \in C^\infty (S)$, and $d_\Ff x^i \otimes \mu$. The right hand side of Equation (\ref{multi}) is also a graded first order multi-differential operator in its arguments. We conclude that Equation (\ref{multi}) is satisfied, provided only it is satisfied for $\nu_1, \dots, \nu_k$ being generators of the above mentioned kind.

An easy computation in local coordinates shows that ${\mathfrak{m}}_1 = - d_{\Ff}$. Moreover, from Corollary \ref{prop:multi-brackets_coordinates} we see that ${\mathfrak{m}}_{k}$ depends on the derivatives of $\W_{\boldsymbol p}^{-1}$ with respect to the $p_i$'s at $\boldsymbol p := (p_i) = 0$ up to order $k$. By induction on $k$ we get
\begin{equation}\label{eq:induction}
	\left. \frac{\partial^k \W_{\boldsymbol p}}{\partial p_{i_1} \cdots \partial p_{i_k}} \right|_{\boldsymbol p = 0} = (-)^k \sum_{\sigma \in S_k} \W^{-1} \F^{i_{\sigma(1)}} \W^{-1} \cdots \F^{i_{\sigma(k)}} \W^{-1}.
\end{equation}
Now, formula (\ref{multi}) follows from Corollary \ref{prop:multi-brackets_coordinates}, equation (\ref{eq:induction}) and the remark that
\[
j^1_G (f \mu) = f j^1 \mu + (\G_a f ) du^a\otimes \mu + (\G f ) dz\otimes \mu,
\]
and
\begin{align*}
	j^1_G (d_\Ff x^i \otimes \mu) &= d_\Ff x^i \otimes j^1 \mu + \frac{\partial G^i_a}{\partial x^j} d_\Ff x^j \otimes (du^a \otimes \mu) \\
	&\quad +  \frac{\partial G^i}{\partial x^j} d_\Ff x^j \otimes (dz \otimes \mu),
\end{align*}
after a straightforward computation.

\section*{Acknowledgements}
H.V.L.~thanks Nguyen Tien Zung for various help in preparing this note.
She acknowledges the IBS CGP at Pohang for financial support and hospitality during her visit, where a part of this paper has been written.
Y.G.O.~thanks Institute of Mathematics of ASCR at Zitna for its hospitality during his visit.
A.G.T. is partially supported by GNSAGA of INdAM, Italy.
L.V. is member of the GNSAGA of INdAM, Italy.

\end{document}